\documentclass{article}%
\usepackage[T1]{fontenc}
\usepackage[francais,english]{babel}
\usepackage{amsthm}
\usepackage{amsmath}
\usepackage{amssymb}
\usepackage[numbers]{natbib}
\usepackage{mathrsfs}
\usepackage{stmaryrd}
\usepackage{thmtools}
\usepackage{graphicx}
\usepackage[top=2.5cm, bottom=2.5cm, left=2cm, right=2cm, headheight=0.5cm,
footskip=60pt]{geometry}
\usepackage{fancyhdr}
\usepackage{dsfont}
\usepackage{setspace}
\usepackage{listings}
\usepackage{amsfonts}%
\providecommand{\U}[1]{\protect\rule{.1in}{.1in}}
\fancypagestyle{plain}{
	\fancyhead{}
	
}

\newtheorem{theorem}{Theorem}

\newtheorem{definition}{Definition}

\newtheorem{proposition}{Proposition}
\newtheorem{remark}{Remark}

\numberwithin{lemma}{section}
\numberwithin{proposition}{section}
\numberwithin{equation}{section}
\numberwithin{remark}{section}
\numberwithin{definition}{section}
\numberwithin{theorem}{section}
\allowdisplaybreaks

\begin{document}

\title{Optimal well-posedness for the inhomogeneous incompressible Navier-Stokes system with general viscosity}
\author{Cosmin Burtea\thanks{ Email address:
cosmin.burtea@u-pec.fr} \thanks{This work was partially supported by a grant of
the Romanian National Authority for Scientific Research and Innovation, CNCS -
UEFISCDI, project number PN-II-RU-TE-2014-4-0320}\\Universit\'e Paris-Est Cr\'eteil, LAMA - CNRS UMR 8050,\\61 Avenue du G\'{e}n\'{e}ral de Gaulle, 94010 Cr\'{e}teil, France}
\maketitle

\begin{abstract}
In this paper we obtain new well-possedness results concerning a linear
inhomogenous Stokes-like system. These results are used to
establish local well-posedness in the critical spaces for initial density
$\rho_{0}$ and velocity $u_{0}$ such that $\rho_{0}-\rho\in\dot{B}%
_{p,1}^{\frac{3}{p}}(\mathbb{R}^{3})$, $u_{0}\in\dot{B}_{p,1}^{\frac{3}{p}%
-1}(\mathbb{R}^{3})$, $p\in\left(  \frac{6}{5},4\right)  $, for the
inhomogeneous incompressible Navier-Stokes system with variable viscosity. To the best of our
knowledge, regarding the $3D$ case, this is the first result in a truly
critical framework for which one does not assume any smallness condition on
the density.
\begin{description}
\item[Keywords] Inhomogeneous Navier-Stokes system; critical regularity; Lagrangian coordinates;
\item[MSC:] 35Q30, 76D05
\end{description}
\end{abstract}
\section{Introduction}

In this paper we deal with the well-posedness of the inhomogeneous,
incompressible Navier-Stokes system:%
\begin{equation}
\left\{
\begin{array}
[c]{r}%
\partial_{t}\rho+\operatorname{div}\left(  \rho u\right)  =0,\\
\partial_{t}\left(  \rho u\right)  +\operatorname{div}\left(  \rho u\otimes
u\right)  -\operatorname{div}\left(  \mu\left(  \rho\right)  D\left(
u\right)  \right)  +\nabla P=0,\\
\operatorname{div}u=0,\\
u_{|t=0}=u_{0}.
\end{array}
\right.  \label{NavierStokes}%
\end{equation}
In the above, $\rho>0$ stands for the density of the fluid, $u\in
\mathbb{R}^{n}$ is the fluid's velocity field while $P$ is the pressure. The
viscosity coefficient $\mu$ is assumed to be a smooth, strictly positive
function of the density while%
\[
D\left(  u\right)  =\nabla u+Du.
\]
is the deformation tensor. This system is used to study fluids obtained as a
mixture of two (or more) incompressible fluids that have different densities:
fluids containing a melted substance, polluted air/water etc. 

There is a very rich literature devoted to the study of the well-posedness of
$\left(  \text{\ref{NavierStokes}}\right)  $ which we will review in the
following lines. Briefly, the question of existence of weak solutions with
finite energy was first considered by Kazhikov in \cite{Kaz2} (see also
\cite{Kaz}) in the case of constant viscosity. The case with a general
viscosity law was treated in \cite{PLions1}. Weak solutions for more regular
data were considered in \cite{Dej}. Recently, weak solutions were investigated
by Huang, Paicu and Zhang in \cite{HuangPaicuZhang3}.

The unique solvability of $\left(  \text{\ref{NavierStokes}}\right)  $ was
first addressed in the seminal work of Lady\v{z}enskaja and Solonnikov in
\cite{LadizSolon}. More precisely, considering $u_{0}\in W^{2-\frac{2}{p}%
,p}\left(  \Omega\right)  $, with $p>2$ , a divergence free vector field that
vanishes on $\partial\Omega$ and $\rho_{0}\in C^{1}\left(  \Omega\right)  $
bounded away from zero, they construct a global strong solution in the $2D$
case respectively a local solution in the $3D$ case. Moreover, if $u_{0}$ is
small in $W^{2-\frac{2}{p},p}\left(  \Omega\right)  $ then global
well-posedness holds true.

The question of weak-strong uniqueness was addressed in \cite{kim} for the
case of sufficiently smooth data with vanishing viscosity.

Over the last thirteen years, efforts were made to obtain well-posedness
results in the so called critical spaces i.e. the spaces which have the same
invariance with respect to time and space dilation as the system itself,
namely%
\[
\left\{
\begin{array}
[c]{r}%
\left(  \rho_{0}\left(  x\right)  ,u_{0}\left(  x\right)  \right)
\rightarrow\left(  \rho_{0}\left(  lx\right)  ,u_{0}\left(  lx\right)
\right)  ,\\
\left(  \rho\left(  t,x\right)  ,u\left(  t,x\right)  \right)  \rightarrow
\left(  \rho\left(  l^{2}t,lx\right)  ,lu\left(  l^{2}t,lx\right)
,l^{2}P\left(  l^{2}t,lx\right)  \right)  .
\end{array}
\right.
\]
For more details and explanations for nowadays a classical approach we refer
to \cite{Dan6} or \cite{Dan3}. In the Besov space context, which includes in
particular the more classical Sobolev spaces, these are
\begin{equation}
\rho_{0}-\bar{\rho}\in\dot{B}_{p_{1},r_{1}}^{\frac{n}{p_{1}}}\text{ and }%
u_{0}\in\dot{B}_{p_{2},r_{2}}^{\frac{n}{p_{2}}-1}. \label{critique}%
\end{equation}
where $\bar{\rho}$ is some constant density state and $n$ is the space
dimension. Working with densities close (in some appropriate norm) to a
constant has led to a rich literature. In \cite{Dan6} local and global
existence results are obtained for the case of constant viscosity and by
taking the initial data%
\[
\rho_{0}-\bar{\rho}\in L^{\infty}\cap\dot{B}_{2,\infty}^{\frac{n}{2}}\text{,
}u_{0}\in\dot{B}_{2,1}^{\frac{n}{2}-1}%
\]
and under the assumption that $\left\Vert \rho_{0}-\bar{\rho}\right\Vert
_{L^{\infty}\cap\dot{B}_{2,\infty}^{\frac{n}{2}}}$ is sufficiently small. The
case with variable viscosity and for initial data
\[
\rho_{0}-\bar{\rho}\in\dot{B}_{p,1}^{\frac{n}{p}}\text{ and }u_{0}\in\dot
{B}_{p,1}^{\frac{n}{p}-1},
\]
$p\in\lbrack1,2n)$, is treated in \cite{Abidi1}. However, uniqueness is
guaranteed once $p\in\lbrack1,n)$. These results where further extended by H.
Abidi and M. Paicu in \cite{AbidiPaicu1} by noticing that $\rho_{0}-\bar{\rho
}$ can be taken in a larger Besov space. In \cite{Haspot1}, B. Haspot
established results in the same spirit as those mentioned above (however, the
results are obtained in the nonhomogeneous framework and thus do not fall into
the critical framework) in the case where the velocity field is not Lipschitz.
In \cite{Dan5}, using the Lagrangian formulation, R. Danchin and P.B. Mucha
establish local and global results for $\left(  \text{\ref{NavierStokes}%
}\right)  $ with constant viscosity when $\rho_{0}-\bar{\rho}\in
\mathcal{M(}\dot{B}_{p,1}^{\frac{n}{p}-1})$, $u_{0}\in\dot{B}_{p,1}^{\frac
{n}{p}-1}$ and under the smallness condition:
\[
\left\Vert \rho_{0}-\bar{\rho}\right\Vert _{\mathcal{M(}\dot{B}_{p,1}%
^{\frac{n}{p}-1})}\ll1,
\]
where $\mathcal{M(}\dot{B}_{p,1}^{\frac{n}{p}-1})$ stands for the multiplier
space of $\dot{B}_{p,1}^{\frac{n}{p}-1}$. In particular, functions with small
jumps enter this framework. Moreover, as a consequence of their approach, the
range of Lebesgue exponents for which uniqueness of solutions holds is
extended to $p\in\lbrack1,2n)$. In \cite{PaicuZhang1}, \cite{HuangPaicuZhang1}%
, \cite{HuangPaicuZhang2}, \cite{HuangPaicuZhang3} the authors improve the
smallness assumptions used in order to obtain global existence. To summarize,
all the previous well-posedness results in critical spaces were established
assuming that the density is close in some sense to a constant state.

When the later assumption is removed, one must impose more regularity on the
data. For the case of constant viscosity, in \cite{Dan7}, R. Danchin obtains
local well posedness respectively global well posedness in dimension $n=2$ for
data drawn from the nonhomogeneous Sobolev spaces: $\left(  \rho_{0}-\bar
{\rho},u_{0}\right)  \in H^{\frac{n}{2}+\alpha}\times H^{\frac{n}{2}-1+\beta}$
with $\alpha,\beta>0$. The same result for the case of general viscosity law
is established in \cite{Abidi1}. For data with non Lipschitz velocity results
were established in \cite{Haspot1}. Concerning rougher densities, in
\cite{Dan9}, considering $\rho_{0}\in L^{\infty}(\mathbb{R}^{d})$ bounded from
below and $u_{0}\in H^{2}\left(  \mathbb{R}^{d}\right)  $ Danchin and Mucha
construct a unique local solution. Again, supposing that the density is close
to some constant state they prove global well-posedness. These results are
generalized in \cite{PaicuZhangZhifei}. Taking the density as above the
authors construct: a global unique solution provided that $u_{0}\in
H^{s}\left(  \mathbb{R}^{2}\right)  $ for any $s>0$ in the $2D$ case
respectively a\ local unique solution in the $3D$ case considering $u_{0}\in
H^{1}\left(  \mathbb{R}^{3}\right)  $. Moreover, assuming that $u_{0}$ is
suitably small the solution constructed is global even in the three
dimensional case.

In critical spaces of the Navier-Stokes system i.e. $\left(
\text{\ref{critique}}\right)  $ there are few well posedness results. Very
recently, in the $2D$ case and allowing variable viscosity, H. Xu, Y. Li and
X. Zhai, \cite{XuLiZhai} constructed a unique local solution to $\left(
\text{\ref{NavierStokes}}\right)  $ provided that the initial data satisfies
$\rho_{0}-\bar{\rho}\in\dot{B}_{p,1}^{\frac{2}{p}}\left(  \mathbb{R}%
^{2}\right)  $ and $u_{0}\in\dot{B}_{p,1}^{\frac{2}{p}-1}\left(
\mathbb{R}^{2}\right)  $. Moreover, if $\rho_{0}-\bar{\rho}\in L^{p}\cap
\dot{B}_{p,1}^{\frac{2}{p}}\left(  \mathbb{R}^{2}\right)  $ and the viscosity
is supposed constant, their solution becomes global. In the $3D$ situation, to
the best of our knowledge, the results that are closest to the critical
regularity are those presented in \cite{AbidiZhang1} and \cite{AbidiZhang2}
(for a similar result in the periodic case one can consult \cite{Poulon}).
More precisely, in $3D$, assuming that%
\[
\rho_{0}-\bar{\rho}\in L^{2}\cap\dot{B}_{2,1}^{\frac{3}{2}}\text{ and }%
u_{0}\in\dot{B}_{2,1}^{\frac{1}{2}}%
\]
and taking constant viscosity, H. Abidi, G. Gui and P.Zhang,
\cite{AbidiZhang1}, show the local well-posedness of system $\left(
\text{\ref{NavierStokes}}\right)  $. Moreover, if the initial velocity is
small then global well-posedness holds true. In \cite{AbidiZhang2} they
establish the same kind of result for initial data%
\[
\rho_{0}-\bar{\rho}\in L^{\lambda}\cap\dot{B}_{\lambda,1}^{\frac{3}{\lambda}%
}\text{ and }u_{0}\in\dot{B}_{p,1}^{\frac{3}{p}-1}%
\]
where $\lambda\in\lbrack1,2]$, $p\in\left[  3,4\right]  $ are such that
$\frac{1}{\lambda}+\frac{1}{p}>\frac{5}{6}$ and $\frac{1}{\lambda}-\frac{1}%
{p}\leq\frac{1}{3}$.

One of the goals of the present paper is to establish local well-posedness in
the critical spaces:%
\[
\rho_{0}-\bar{\rho}\in\dot{B}_{p,1}^{\frac{3}{p}}\left(  \mathbb{R}%
^{3}\right)  \text{, }u_{0}\in\dot{B}_{p,1}^{\frac{3}{p}-1}\left(
\mathbb{R}^{3}\right)  \text{, }p\in\left(  \frac{6}{5},4\right)
\]
for System $\left(  \text{\ref{NavierStokes}}\right)  $

\begin{itemize}
\item with general smooth variable viscosity law,

\item without any smallness assumption on the density,

\item without the extra low frequencies assumption. In particular, we
generalize the local existence and uniqueness result of H. Abidi, G. Gui and
P. Zhang from \cite{AbidiZhang1} thus achieving the critical regularity.
\end{itemize}

As in \cite{Dan5} we will not work directly with system $\left(
\text{\ref{NavierStokes}}\right)  $ instead we will rather use its Lagrangian
formulation. By proceeding so, we are naturally led to consider the following
Stokes problem with time independent, nonconstant coefficients:
\begin{equation}
\left\{
\begin{array}
[c]{r}%
\partial_{t}u-a\operatorname{div}\left(  bD(u)\right)  +a\nabla P=f,\\
\operatorname{div}u=\operatorname{div}R,\\
u_{|t=0}=u_{0}.
\end{array}
\right.  \label{Stokes2}%
\end{equation}
We establish global well-posedness results for System $\left(
\text{\ref{Stokes2}}\right)  $. This can be viewed as a first step towards
generalizing the results of Danchin and Mucha obtained in \cite{Dan3}, Chapter
$4$, for the case of general viscosity and without assuming that the density
is close to a constant state. Let us mention that the estimates that we obtain
for System $\left(  \text{\ref{Stokes2}}\right)  $ have a wider range of
applications: in a forthcoming paper we will investigate the well-posedness
issue of the Navier-Stokes-Korteweg system under optimal regularity assumptions.

To summarize all the above, our main result reads:

\begin{theorem}
\label{NS2}Let us consider $p\in\left(  \frac{6}{5},4\right)  $.
Assume that there exists positive constants $\left(  \bar{\rho},\rho_{\star
},\rho^{\star}\right)  $ such that $\rho_{0}-\bar{\rho}\in\dot{B}_{p,1}%
^{\frac{3}{p}}\left(  \mathbb{R}^{3}\right)  $ and $0<\rho_{\star}<\rho
_{0}<\rho^{\star}$. Furthermore, consider $u_{0}$ a divergence free vector
field with coefficients in $\dot{B}_{p,1}^{\frac{3}{p}-1}\left(
\mathbb{R}^{3}\right)  $. Then, there exists a time $T>0$ and a unique
solution $\left(  \rho,u,\nabla P\right)  $ of system $\left(
\text{\ref{NavierStokes}}\right)  $ with
\[
\rho-\bar{\rho}\in\mathcal{C}_{T}(\dot{B}_{p,1}^{\frac{3}{p}}(\mathbb{R}%
^{3}))\cap L_{T}^{\infty}(\dot{B}_{p,1}^{\frac{3}{p}}(\mathbb{R}^{3}))\text{,
}u\in\mathcal{C}_{T}(\dot{B}_{p,1}^{\frac{3}{p}-1}(\mathbb{R}^{3}))\text{ and
}\left(  \partial_{t}u,\nabla^{2}u,\nabla P\right)  \in L_{T}^{1}(\dot
{B}_{p,1}^{\frac{3}{p}-1}(\mathbb{R}^{3})).
\]

\end{theorem}

One salutary feature of the Lagrangian formulation is that the density becomes
independent of time. More precisely, considering $\left(  \rho,u,\nabla
P\right)  $ a solution of $\left(  \text{\ref{NavierStokes}}\right)  $ and
denoting by $X$ the flow associated to the vector field $u$:%
\[
X\left(  t,y\right)  =y+\int_{0}^{t}u\left(  \tau,X(\tau,y)\right)  dy
\]
we introduce the new Lagrangian variables:%
\[
\bar{\rho}\left(  t,y\right)  =\rho\left(  t,X\left(  t,y\right)  \right)
,\text{ }\bar{u}\left(  t,y\right)  =u\left(  t,X\left(  t,y\right)  \right)
\text{ and }\bar{P}\left(  t,y\right)  =P\left(  t,X\left(  t,y\right)
\right)  .
\]
Then, using the chain rule and Proposition \ref{Formule} we gather that
$\bar{\rho}\left(  t,\cdot\right)  =\rho_{0}$ and%

\begin{equation}
\left\{
\begin{array}
[c]{r}%
\rho_{0}\partial_{t}\bar{u}-\operatorname{div}\left(  \mu\left(  \rho
_{0}\right)  A_{\bar{u}}D_{A_{\bar{u}}}\left(  \bar{u}\right)  \right)
+A_{\bar{u}}^{T}\nabla\bar{P}=0,\\
\operatorname{div}\left(  A_{\bar{u}}\bar{u}\right)  =0,\\
\bar{u}_{|t=0}=u_{0}.
\end{array}
\right.  \label{NavierStokes_modif}%
\end{equation}
where $A_{\bar{u}}$ is the inverse of the differential of $X$, and
\[
D_{A}\left(  \bar{u}\right)  =D\bar{u}A_{\bar{u}}+A_{\bar{u}}^{T}\nabla\bar
{u}.
\]
Note that we can give a meaning to $\left(  \text{\ref{NavierStokes_modif}%
}\right)  $ independently of the Eulerian formulation by stating:%
\[
X\left(  t,y\right)  =y+\int_{0}^{t}\bar{u}\left(  \tau,y\right)  d\tau.
\]
Theorem \ref{NS2} will be a consequence of the following result:

\begin{theorem}
\label{NS1}Let us consider $p\in\left(  \frac{6}{5},4\right)  $. Assume that
there exists positive $\left(  \bar{\rho},\rho_{\star},\rho^{\star}\right)  $
such that $\rho_{0}-\bar{\rho}\in\dot{B}_{p,1}^{\frac{3}{p}}\left(
\mathbb{R}^{3}\right)  $ and $0<\rho_{\star}<\rho_{0}<\rho^{\star}$.
Furthermore, consider $u_{0}$ a divergence free vector field with coefficients
in $\dot{B}_{p,1}^{\frac{3}{p}-1}\left(  \mathbb{R}^{3}\right)  $. Then, there
exists a time $T>0$ and a unique solution $\left(  \bar{u},\nabla\bar
{P}\right)  $ of system $\left(  \text{\ref{NavierStokes_modif}}\right)  $
with
\[
\bar{u}\in\mathcal{C}_{T}(\dot{B}_{p,1}^{\frac{3}{p}-1}(\mathbb{R}^{3}))\text{
and }\left(  \partial_{t}\bar{u},\nabla^{2}\bar{u},\nabla\bar{P}\right)  \in
L_{T}^{1}(\dot{B}_{p,1}^{\frac{3}{p}-1}(\mathbb{R}^{3})).
\]
Moreover, there exists a positive constant $C=C\left(  \rho_{0}\right)  $ such
that:
\[
\left\Vert u\right\Vert _{L_{T}^{\infty}(\dot{B}_{p,1}^{\frac{3}{p}-1}%
)}+\left\Vert \left(  \nabla^{2}u,\nabla P\right)  \right\Vert _{L_{T}%
^{1}(\dot{B}_{p,1}^{\frac{3}{p}-1})}\leq\left\Vert u_{0}\right\Vert _{\dot
{B}_{p,1}^{\frac{3}{p}-1}}\exp\left(  CT\right)  .
\]

\end{theorem}

$\ $The study of system $\left(  \text{\ref{NavierStokes_modif}}\right)  $
naturally leads to the Stokes-like system $\left(  \text{\ref{Stokes2}%
}\right)  $. In Section \ref{Section2} we establish the global well-posedness
of System $\left(  \text{\ref{Stokes2}}\right)  $. More precisely, we prove:

\begin{theorem}
\label{Teorema_Stokes}Let us consider $n\in\{2,3\}$ and $p\in\left(
1,4\right)  $ if $n=2$ or \ $p\in\left(  \frac{6}{5},4\right)  $ if $n=3$.
Assume there exist positive constants $\left(  a_{\star},b_{\star},a^{\star
},b^{\star},\bar{a},\bar{b}\right)  $ such that $a-\bar{a}\in\dot{B}%
_{p,1}^{\frac{n}{p}}\left(  \mathbb{R}^{n}\right)  $, $b-\bar{b}\in\dot
{B}_{p,1}^{\frac{n}{p}}\left(  \mathbb{R}^{n}\right)  $ and
\begin{align*}
0  &  <a_{\star}\leq a\leq a^{\star},\\
0  &  <b_{\star}\leq b\leq b^{\star}.
\end{align*}
Furthermore, consider the vector fields $u_{0}$ and $f$ with coefficients in
$\dot{B}_{p,1}^{\frac{n}{p}-1}\left(  \mathbb{R}^{n}\right)  $ respectively in
$L_{loc}^{1}(\dot{B}_{p,1}^{\frac{n}{p}-1}\left(  \mathbb{R}^{n}\right)  )$.
Also, let \ us consider the vector field $R\in(\mathcal{S}^{\prime}\left(
\mathbb{R}^{n}\right)  )^{n}$ with\footnote{$\mathcal{P}$ is the Leray
projector over divergence free vector fields, $\mathcal{Q}=Id-\mathcal{P}$}%
\[
\mathcal{Q}R\in\mathcal{C(}[0,\infty);\dot{B}_{p,1}^{\frac{n}{p}-1}\left(
\mathbb{R}^{n}\right)  )\text{ and }\left(  \partial_{t}R,\nabla
\operatorname{div}R\right)  \in L_{loc}^{1}(\dot{B}_{p,1}^{\frac{n}{p}%
-1}\left(  \mathbb{R}^{n}\right)  )
\]
such that:%
\[
\operatorname{div}u_{0}=\operatorname{div}R\left(  0,\cdot\right)  .
\]
Then, system $\left(  \text{\ref{Stokes2}}\right)  $ has a unique solution
$\left(  u,\nabla P\right)  $ with:
\[
u\in C([0,\infty),\dot{B}_{p,1}^{\frac{n}{p}-1}\left(  \mathbb{R}^{n}\right)
)\text{ and }\partial_{t}u,\nabla^{2}u,\nabla P\in L_{loc}^{1}(\dot{B}%
_{p,1}^{\frac{n}{p}-1}\left(  \mathbb{R}^{n}\right)  ).
\]
Moreover, there exists a constant $C=C\left(  a,b\right)  $ such that:%
\begin{gather}
\left\Vert u\right\Vert _{L_{t}^{\infty}(\dot{B}_{p,1}^{\frac{n}{p}-1}%
)}+\left\Vert \left(  \partial_{t}u,\nabla^{2}u,\nabla P\right)  \right\Vert
_{L_{t}^{1}(\dot{B}_{p,1}^{\frac{n}{p}-1})}\nonumber\\
\leq\left(  \left\Vert u_{0}\right\Vert _{\dot{B}_{p,1}^{\frac{n}{p}-1}%
}+\left\Vert \left(  f,\partial_{t}R,\nabla\operatorname{div}R\right)
\right\Vert _{L_{t}^{1}(\dot{B}_{p,1}^{\frac{n}{p}-1})}\right)  \exp\left(
C(t+1)\right)  , \label{Estimare}%
\end{gather}
for all $t\in\lbrack0,\infty)$.
\end{theorem}

The difficulty in establishing such a result comes from the fact that the
pressure and velocity are "strongly" coupled as opposed to the case where
$\rho$ is close to a constant see Remark \ref{111} below. The key idea is to
use the high-low frequency splitting technique first used in \cite{Dan8}
combined with the special structure of the "incompressible" part of $a\nabla
P$ i.e.
\begin{align*}
\mathcal{P}(a\nabla P) &  =\mathcal{P}((a-\bar{a})\nabla P)=\mathcal{P}%
((a-\bar{a})\nabla P)-(a-\bar{a})\mathcal{P(}\nabla P)\\
&  :=[\mathcal{P},a-\bar{a}]\nabla P.
\end{align*}
which is, loosely speaking, more regular than $\nabla P$. Let us mention that
a similar principle holds for $u$ which is divergence free\footnote{and thus
$\mathcal{Q}u=0$.}: whenever we estimate some term of the form $\mathcal{Q}%
\left(  bM\left(  D\right)  u\right)  $ where $b$ lies in an appropriate Besov
space and $M\left(  D\right)  $ is some pseudo-differential operator then we
may write it%
\[
\mathcal{Q}\left(  bM\left(  D\right)  u\right)  =[\mathcal{Q},b]M\left(
D\right)  u
\]
and use the fact that the later expression in more regular than $M\left(
D\right)  u$, see Proposition \ref{A7}.

The proof of Theorem \ref{Teorema_Stokes} in the $3$-dimensional case is more
subtle: first we prove a more restrictive result by demanding an extra
low-frequency information on the initial data. Then, using a perturbative
version of Danchin and Mucha's results of \cite{Dan3} we arrive at
constructing a solution with the optimal regularity. The uniqueness is
obtained by a duality method.

Once the estimates of Theorem \ref{Teorema_Stokes} are established, we proceed
with the proof of Theorem \ref{NS1} which is the object of Section
\ref{Section3}. Finally, we show the equivalence between system $\left(
\text{\ref{NavierStokes_modif}}\right)  $ and system $\left(
\text{\ref{NavierStokes}}\right)  $ thus achieving the proof of Theorem
\ref{NS2}. We end this paper with an Appendix where results of
Littlewood-Paley theory used through the text are gathered.

\section{The Stokes system with nonconstant coefficients\label{Section2}}

\subsection{Pressure estimates}

Before handling System $\left(  \text{\ref{Stokes2}}\right)  $ we shall study
the following elliptic equation:%
\begin{equation}
\operatorname{div}\left(  a\nabla P\right)  =\operatorname{div}f. \label{E}%
\end{equation}
For the reader's convenience let us cite the following classical result, a
proof of which can be found, for instance in \cite{Dan10}:

\begin{proposition}
\label{presiuneL2} For all vector field $f$ with coefficients in
$L^{2}\left(  \mathbb{R}^{n}\right)  $, there exists a tempered distribution
$P$ unique up to constant functions such that $\nabla P\in L^{2}\left(
\mathbb{R}^{n}\right)  $ and Equation $\left(  \text{\ref{E}}\right)  $ is
satisfied. In addition, we have:%
\[
a_{\star}\left\Vert \nabla P\right\Vert _{L^{2}}\leq\left\Vert \mathcal{Q}%
f\right\Vert _{L^{2}}.
\]

\end{proposition}

Recently, in \cite{XuLiZhai}, in the $2D$ case, H. Xu, Y. Li and X. Zhai
studied the eliptic equation $\left(  \text{\ref{E}}\right)  $ with the data
$\left(  a-\bar{a},f\right)  $ in Besov spaces. Using a different approach, we
obtain estimates in both two dimensional and three dimensional situations. Let
us also mention that our method allows to obtain a wider range of indices than
the one of Proposition $3.1.$ $i)$ of \cite{XuLiZhai}.

We choose to focus on the $3D$ case. We aim at establishing the following result:

\begin{proposition}
\label{Estimare_eliptica1}Let us consider $p\in\left(  \frac{6}{5},2\right)  $
and $q\in\lbrack1,\infty)$ such that $\frac{1}{p}-\frac{1}{q}\leq\frac{1}{2}$.
Assume that there exists positive constants $\left(  \bar{a},a_{\star
},a^{\star}\right)  $ such that $a-\bar{a}\in\dot{B}_{q,1}^{\frac{3}{q}%
}\left(  \mathbb{R}^{3}\right)  $ and $0<a_{\star}\leq a\leq a^{\star}$.
Furthermore, consider $f\in\dot{B}_{p,2}^{\frac{3}{p}-\frac{3}{2}}\left(
\mathbb{R}^{3}\right)  $. Then there exists a tempered distribution $P$ unique
up to constant functions such that $\nabla P\in\dot{B}_{p,2}^{\frac{3}%
{p}-\frac{3}{2}}\left(  \mathbb{R}^{3}\right)  $ and Equation $\left(
\text{\ref{E}}\right)  $ is satisfied. Moreover, the following estimate holds
true:%
\begin{equation}
\left\Vert \nabla P\right\Vert _{\dot{B}_{p,2}^{\frac{3}{p}-\frac{3}{2}}%
}\lesssim\left(  \frac{1}{\bar{a}}+\left\Vert \frac{1}{a}-\frac{1}{\bar{a}%
}\right\Vert _{\dot{B}_{q,1}^{\frac{3}{q}}}\right)  \left(  1+\frac
{1}{a_{\star}}\left\Vert a-\bar{a}\right\Vert _{\dot{B}_{q,1}^{\frac{3}{q}}%
}\right)  \left\Vert \mathcal{Q}f\right\Vert _{\dot{B}_{p,2}^{\frac{3}%
{p}-\frac{3}{2}}}. \label{p_mic}%
\end{equation}

\end{proposition}

\begin{remark}
Working in Besov spaces with third index $r=2$ is enough in view of the
applications that we have in mind. However similar estimates do hold true when
the third index is choosen in the interval $\left[  1,2\right]  $.
\end{remark}

\begin{proof}
Because $p<2$, Proposition \ref{Embedding} ensures that $\dot{B}_{p,2}%
^{\frac{3}{p}-\frac{3}{2}}\hookrightarrow L^{2}$ and owing to Proposition
\ref{presiuneL2}, we get the existence of $P\in\mathcal{S}^{\prime}\left(
\mathbb{R}^{3}\right)  $ with $\nabla P\in L^{2}$ and
\begin{equation}
a_{\star}\left\Vert \nabla P\right\Vert _{L^{2}}\leq\left\Vert \mathcal{Q}%
f\right\Vert _{L^{2}}. \label{E1}%
\end{equation}
Moreover, as $\mathcal{Q}$ is a continuous operator on $L^{2}$ we deduce from
$\left(  \text{\ref{E}}\right)  $ that
\begin{equation}
\mathcal{Q}\left(  a\nabla P\right)  =\mathcal{Q}f\text{ .} \label{Eliptic1}%
\end{equation}
Using the Bony decomposition (see Definition \ref{paraprodus_rest} and the
remark that follows) and the fact that $\mathcal{P}\left(  \nabla P\right)
=0$ we write that:
\[
\mathcal{P}\left(  a\nabla P\right)  =\mathcal{P}\left(  \dot{T}_{\nabla
P}^{\prime}(a-\bar{a})\right)  +\left[  \mathcal{P},\dot{T}_{a-\bar{a}%
}\right]  \nabla P.
\]
Using Proposition \ref{ParaRem} along with Proposition \ref{Embedding} and
relation $\left(  \text{\ref{E1}}\right)  $, we get that
\begin{equation}
\left\Vert \mathcal{P}\left(  \dot{T}_{\nabla P}^{\prime}(a-\bar{a})\right)
\right\Vert _{\dot{B}_{p,2}^{\frac{3}{p}-\frac{3}{2}}}\lesssim\left\Vert
\nabla P\right\Vert _{L^{2}}\left\Vert a-\bar{a}\right\Vert _{\dot
{B}_{p^{\star},2}^{\frac{3}{p}-\frac{3}{2}}}\lesssim\frac{1}{a_{\star}%
}\left\Vert \mathcal{Q}f\right\Vert _{L^{2}}\left\Vert a-\bar{a}\right\Vert
_{\dot{B}_{q,1}^{\frac{3}{q}}} \label{E2}%
\end{equation}
where%
\[
\frac{1}{p}=\frac{1}{2}+\frac{1}{p^{\star}}.
\]
Next, proceeding as in Proposition \ref{A71} we get that:%
\begin{equation}
\left\Vert \left[  \mathcal{P},\dot{T}_{a-\bar{a}}\right]  \nabla P\right\Vert
_{\dot{B}_{p,2}^{\frac{3}{p}-\frac{3}{2}}}\lesssim\left\Vert \nabla
a\right\Vert _{\dot{B}_{p^{\star},2}^{\frac{3}{p}-\frac{5}{2}}}\left\Vert
\nabla P\right\Vert _{L^{2}}\lesssim\frac{1}{a_{\star}}\left\Vert
\mathcal{Q}f\right\Vert _{L^{2}}\left\Vert a-\bar{a}\right\Vert _{\dot
{B}_{q,1}^{\frac{3}{q}}}. \label{E4}%
\end{equation}
Putting together relations $\left(  \text{\ref{E2}}\right)  $ and $\left(
\text{\ref{E4}}\right)  $ we get%
\[
\left\Vert \mathcal{P}\left(  a\nabla P\right)  \right\Vert _{\dot{B}%
_{p,2}^{\frac{3}{p}-\frac{3}{2}}}\lesssim\frac{1}{a_{\star}}\left\Vert
\mathcal{Q}f\right\Vert _{L^{2}}\left\Vert a-\bar{a}\right\Vert _{\dot
{B}_{q,1}^{\frac{3}{q}}}.
\]
Combining this with $\left(  \text{\ref{Eliptic1}}\right)  $ and Proposition
\ref{Embedding}, we find that:%
\[
\left\Vert a\nabla P\right\Vert _{\dot{B}_{p,2}^{\frac{3}{p}-\frac{3}{2}}%
}\lesssim\left(  1+\frac{1}{a_{\star}}\left\Vert a-\bar{a}\right\Vert
_{\dot{B}_{q,1}^{\frac{3}{q}}}\right)  \left\Vert \mathcal{Q}f\right\Vert
_{\dot{B}_{p,2}^{\frac{3}{p}-\frac{3}{2}}}.
\]
Of course, writing
\[
\nabla P=\frac{1}{a}a\nabla P
\]
using product rules one gets that:%
\begin{equation}
\left\Vert \nabla P\right\Vert _{\dot{B}_{p,2}^{\frac{3}{p}-\frac{3}{2}}%
}\lesssim\left(  \frac{1}{\bar{a}}+\left\Vert \frac{1}{a}-\frac{1}{\bar{a}%
}\right\Vert _{\dot{B}_{q,1}^{\frac{3}{q}}}\right)  \left(  1+\frac
{1}{a_{\star}}\left\Vert a-\bar{a}\right\Vert _{\dot{B}_{q,1}^{\frac{3}{q}}%
}\right)  \left\Vert \mathcal{Q}f\right\Vert _{\dot{B}_{p,2}^{\frac{3}%
{p}-\frac{3}{2}}} \label{estimP2}%
\end{equation}

\end{proof}
 Applying the same technique as above leads to the $2$ dimensional estimate:

\begin{proposition}
\label{Estimare_eliptica1_2D}Let us consider $p\in\left(  1,2\right)  $ and
$q\in\lbrack1,\infty)$ such that $\frac{1}{p}-\frac{1}{q}\leq\frac{1}{2}$.
Assume that there exists positive constants $\left(  \bar{a},a_{\star
},a^{\star}\right)  $ such that $a-\bar{a}\in\dot{B}_{q,1}^{\frac{2}{q}%
}\left(  \mathbb{R}^{2}\right)  $ and $0<a_{\star}\leq a\leq a^{\star}$.
Furthermore, consider $f\in\dot{B}_{p,2}^{\frac{2}{p}-1}\left(  \mathbb{R}%
^{2}\right)  $. Then there exists a tempered distribution $P$ unique up to
constant functions such that $\nabla P\in\dot{B}_{p,2}^{\frac{2}{p}-1}\left(
\mathbb{R}^{2}\right)  $ and Equation $\left(  \text{\ref{E}}\right)  $ is
satisfied. Moreover, the following estimate holds true:%
\begin{equation}
\left\Vert \nabla P\right\Vert _{\dot{B}_{p,2}^{\frac{2}{p}-1}}\lesssim\left(
\frac{1}{\bar{a}}+\left\Vert \frac{1}{a}-\frac{1}{\bar{a}}\right\Vert
_{\dot{B}_{q,1}^{\frac{2}{q}}}\right)  \left(  1+\frac{1}{a_{\star}}\left\Vert
a-\bar{a}\right\Vert _{\dot{B}_{q,1}^{\frac{2}{q}}}\right)  \left\Vert
\mathcal{Q}f\right\Vert _{\dot{B}_{p,2}^{\frac{2}{p}-1}}. \label{p_mic2D}%
\end{equation}

\end{proposition}

Let us point out that the restriction $p>\frac{6}{5}$ comes from the fact that
we need $\frac{3}{p}-\frac{5}{2}<0$ in relation $\left(  \text{\ref{E4}%
}\right)  $. In $2D$ instead of $\frac{3}{p}-\frac{5}{2}$ we will have
$\frac{2}{p}-2$ which is negative provided that $p>1$.

The next result covers the range of integrability indices larger than $2:$

\begin{proposition}
\label{Estimare_eliptica2}Let us consider $p\in\left(  2,6\right)  $ and
$q\in\lbrack1,\infty)$ such that $\frac{1}{p}+\frac{1}{q}\geq\frac{1}{2}$.
Assume that there exists positive constants $\left(  \bar{a},a_{\star
},a^{\star}\right)  $ such that $a-\bar{a}\in\dot{B}_{q,1}^{\frac{3}{q}%
}\left(  \mathbb{R}^{3}\right)  $ and $0<a_{\star}\leq a\leq a^{\star}$.
Furthermore, consider $f\in\dot{B}_{p,2}^{\frac{3}{p}-\frac{3}{2}}\left(
\mathbb{R}^{3}\right)  $ and a tempered distribution $P$ \ with $\nabla
P\in\dot{B}_{p,2}^{\frac{3}{p}-\frac{3}{2}}\left(  \mathbb{R}^{3}\right)  $
such that Equation $\left(  \text{\ref{E}}\right)  $ is satisfied. Then, the
following estimate holds true:%
\begin{equation}
\left\Vert \nabla P\right\Vert _{\dot{B}_{p,2}^{\frac{3}{p}-\frac{3}{2}}%
}\lesssim\left(  \frac{1}{\bar{a}}+\left\Vert \frac{1}{a}-\frac{1}{\bar{a}%
}\right\Vert _{\dot{B}_{q,1}^{\frac{3}{q}}}\right)  \left(  1+\frac
{1}{a_{\star}}\left\Vert a-\bar{a}\right\Vert _{\dot{B}_{q,1}^{\frac{3}{q}}%
}\right)  \left\Vert \mathcal{Q}f\right\Vert _{\dot{B}_{p,2}^{\frac{3}%
{p}-\frac{3}{2}}}. \label{p_mare}%
\end{equation}

\end{proposition}

\begin{proof}
Let us notice that $p^{\prime}$ the conjugate Lebesgue exponent of $p$
satisfies $p^{\prime}\in\left(  \frac{6}{5},2\right)  $ and $\frac
{1}{p^{\prime}}-\frac{1}{q}\leq\frac{1}{2}$. Thus, according to Proposition
\ref{Estimare_eliptica1}, for any $g$ belonging to the unit ball of
$\mathcal{S\cap}\dot{B}_{p^{\prime},2}^{\frac{3}{p^{\prime}}-\frac{3}{2}}$
there exists a $P_{g}\in\mathcal{S}^{\prime}\left(  \mathbb{R}^{3}\right)  $
with $\nabla P_{g}\in\mathcal{S\cap}$ $\dot{B}_{p^{\prime},2}^{\frac
{3}{p^{\prime}}-\frac{3}{2}}$ such that%
\[
\operatorname{div}\left(  a\nabla P_{g}\right)  =\operatorname{div}g
\]
and
\[
\left\Vert \nabla P_{g}\right\Vert _{\dot{B}_{p^{\prime},2}^{\frac
{3}{p^{\prime}}-\frac{3}{2}}}\lesssim\left(  \frac{1}{\bar{a}}+\left\Vert
\frac{1}{a}-\frac{1}{\bar{a}}\right\Vert _{\dot{B}_{q,1}^{\frac{3}{q}}%
}\right)  \left(  1+\frac{1}{a_{\star}}\left\Vert a-\bar{a}\right\Vert
_{\dot{B}_{q,1}^{\frac{3}{q}}}\right)  .
\]
We write that%
\begin{align*}
\left\langle \nabla P,g\right\rangle  &  =\left\langle P,\operatorname{div}%
g\right\rangle =\left\langle P,\operatorname{div}\left(  a\nabla P_{g}\right)
\right\rangle \\
&  =\left\langle \operatorname{div}\mathcal{Q}f,P_{g}\right\rangle
=\left\langle \mathcal{Q}f,\nabla P_{g}\right\rangle ,
\end{align*}
and consequently%
\begin{align*}
\left\vert \left\langle \nabla P,g\right\rangle \right\vert  &  \lesssim
\left\Vert \mathcal{Q}f\right\Vert _{\dot{B}_{p,2}^{\frac{3}{p}-\frac{3}{2}}%
}\left\Vert \nabla P_{g}\right\Vert _{\dot{B}_{p^{\prime},2}^{\frac
{3}{p^{\prime}}-\frac{3}{2}}}\\
&  \lesssim\left(  \frac{1}{\bar{a}}+\left\Vert \frac{1}{a}-\frac{1}{\bar{a}%
}\right\Vert _{\dot{B}_{q,1}^{\frac{3}{q}}}\right)  \left(  1+\frac
{1}{a_{\star}}\left\Vert a-\bar{a}\right\Vert _{\dot{B}_{q,1}^{\frac{3}{q}}%
}\right)  \left\Vert \mathcal{Q}f\right\Vert _{\dot{B}_{p,2}^{\frac{3}%
{p}-\frac{3}{2}}}.
\end{align*}
Using Proposition \ref{Caracterizare} we get that relation $\left(
\text{\ref{p_mare}}\right)  $ holds true.
\end{proof}

As in the previous situation, by applying the same technique we get a similar
result in $2D$:

\begin{proposition}
\label{Estimare_eliptica2_2D}Let us consider $p\in(2,\infty)$ and $q\in
\lbrack1,\infty)$ such that $\frac{1}{p}+\frac{1}{q}\geq\frac{1}{2}$. Assume
that there exists positive constants $\left(  \bar{a},a_{\star},a^{\star
}\right)  $ such that $a-\bar{a}\in\dot{B}_{q,1}^{\frac{2}{q}}\left(
\mathbb{R}^{2}\right)  $ and $0<a_{\star}\leq a\leq a^{\star}$. Furthermore,
consider $f\in\dot{B}_{p,2}^{\frac{2}{p}-1}\left(  \mathbb{R}^{2}\right)  $
and a tempered distribution $P$ \ with $\nabla P\in\dot{B}_{p,2}^{\frac{2}%
{p}-1}\left(  \mathbb{R}^{2}\right)  $ such that Equation $\left(
\text{\ref{E}}\right)  $ is satisfied. Then, following estimate holds true:%
\begin{equation}
\left\Vert \nabla P\right\Vert _{\dot{B}_{p,2}^{\frac{2}{p}-1}}\lesssim\left(
\frac{1}{\bar{a}}+\left\Vert \frac{1}{a}-\frac{1}{\bar{a}}\right\Vert
_{\dot{B}_{q,1}^{\frac{2}{q}}}\right)  \left(  1+\frac{1}{a_{\star}}\left\Vert
a-\bar{a}\right\Vert _{\dot{B}_{q,1}^{\frac{2}{q}}}\right)  \left\Vert
\mathcal{Q}f\right\Vert _{\dot{B}_{p,2}^{\frac{2}{p}-1}}. \label{p_mare2D}%
\end{equation}

\end{proposition}

\subsection{Some preliminary results}

In this section we derive\ estimates for a Stokes-like problem with time
independent, nonconstant coefficients. Before proceeding to the actual proof,
for the reader's convenience, let us cite the following results pertaining to
the case $a=\bar{a}$, $b=\bar{b}$ constants:

\begin{proposition}
\label{Stokes_constant}Let us consider $u_{0}\in\dot{B}_{p,1}^{\frac{n}{p}-1}$
and $\left(  f,\partial_{t}R,\nabla\operatorname{div}R\right)  \in L_{T}%
^{1}(\dot{B}_{p,1}^{\frac{n}{p}-1})$ with $\mathcal{Q}R\in C_{T}(\dot{B}%
_{p,1}^{\frac{n}{p}-1})$ such that%
\[
\operatorname{div}u_{0}=\operatorname{div}R\left(  0,\cdot\right)  .
\]
Then, system
\[
\left\{
\begin{array}
[c]{r}%
\partial_{t}u-\bar{a}\bar{b}\Delta u+\bar{a}\nabla P=f,\\
\operatorname{div}u=\operatorname{div}R,\\
u_{|t=0}=u_{0},
\end{array}
\right.
\]
has a unique solution $\left(  u,\nabla P\right)  $ with:%
\[
u\in\mathcal{C(}[0,T);\dot{B}_{p,1}^{\frac{n}{p}-1})\text{ and }\partial
_{t}u,\nabla^{2}u,\nabla P\in L_{T}^{1}(\dot{B}_{p,1}^{\frac{n}{p}-1})
\]
and the following estimate is valid:%
\[
\left\Vert u\right\Vert _{L_{T}^{\infty}(\dot{B}_{p,1}^{\frac{n}{p}-1}%
)}+\left\Vert \left(  \partial_{t}u,\bar{a}\bar{b}\nabla^{2}u,\bar{a}\nabla
P\right)  \right\Vert _{L_{T}^{1}(\dot{B}_{p,1}^{\frac{n}{p}-1})}%
\lesssim\left\Vert u_{0}\right\Vert _{\dot{B}_{p,1}^{\frac{n}{p}-1}%
}+\left\Vert \left(  f,\partial_{t}R,\bar{a}\bar{b}\nabla\operatorname{div}%
R\right)  \right\Vert _{L_{T}^{1}(\dot{B}_{p,1}^{\frac{n}{p}-1})}.
\]

\end{proposition}

As a consequence of the previous result, one can establish via a perturbation argument:

\begin{proposition}
\label{Stokes_perturb}Let us consider $u_{0}\in\dot{B}_{p,1}^{\frac{n}{p}-1}$
and $\left(  f,\partial_{t}R,\nabla\operatorname{div}R\right)  \in L_{T}%
^{1}(\dot{B}_{p,1}^{\frac{n}{p}-1})$ with $\mathcal{Q}R\in C_{T}(\dot{B}%
_{p,1}^{\frac{n}{p}-1})$ such that%
\[
\operatorname{div}u_{0}=\operatorname{div}R\left(  0,\cdot\right)  .
\]
Then, there exists a $\eta=\eta\left(  \bar{a}\right)  $ small enough such
that for all $c\in\dot{B}_{p,1}^{\frac{n}{p}}$ with%
\[
\left\Vert c\right\Vert _{\dot{B}_{p,1}^{\frac{n}{p}}}\leq\eta\text{ ,}%
\]
the system
\[
\left\{
\begin{array}
[c]{r}%
\partial_{t}u-\bar{a}\bar{b}\Delta u+\left(  \bar{a}+c\right)  \nabla P=f,\\
\operatorname{div}u=\operatorname{div}R,\\
u_{|t=0}=u_{0},
\end{array}
\right.
\]
has a unique solution $\left(  u,\nabla P\right)  $ with:%
\[
u\in\mathcal{C(}[0,T);\dot{B}_{p,1}^{\frac{n}{p}-1})\text{ and }\partial
_{t}u,\nabla^{2}u,\nabla P\in L_{T}^{1}(\dot{B}_{p,1}^{\frac{n}{p}-1})
\]
and the following estimate is valid:%
\[
\left\Vert u\right\Vert _{L_{T}^{\infty}(\dot{B}_{p,1}^{\frac{n}{p}-1}%
)}+\left\Vert \left(  \partial_{t}u,\bar{a}\bar{b}\nabla^{2}u,\bar{a}\nabla
P\right)  \right\Vert _{L_{T}^{1}(\dot{B}_{p,1}^{\frac{n}{p}-1})}%
\lesssim\left\Vert u_{0}\right\Vert _{\dot{B}_{p,1}^{\frac{n}{p}-1}%
}+\left\Vert \left(  f,\partial_{t}R,\bar{a}\bar{b}\nabla\operatorname{div}%
R\right)  \right\Vert _{L_{T}^{1}(\dot{B}_{p,1}^{\frac{n}{p}-1})}.
\]

\end{proposition}

The above results were established by Danchin and Mucha in \cite{Dan4} and
\cite{Dan3}.

In all what follows we denote by $E_{loc}$ the space of $\left(  u,\nabla
P\right)  $ such that:
\[
u\in\mathcal{C(}[0,\infty);\dot{B}_{p,1}^{\frac{n}{p}-1})\text{ and }\left(
\nabla^{2}u,\nabla P\right)  \in L_{loc}^{1}(\dot{B}_{p,1}^{\frac{n}{p}%
-1})\times L_{loc}^{1}(\dot{B}_{p,2}^{\frac{n}{p}-\frac{n}{2}}\cap\dot
{B}_{p,1}^{\frac{n}{p}-1}).
\]
Also, let us introduce the space $E_{T}$ of $u\in\mathcal{C}_{T}(\dot{B}%
_{p,1}^{\frac{n}{p}-1})$ with $\nabla^{2}u\in L_{T}^{1}(\dot{B}_{p,1}%
^{\frac{n}{p}-1})$ and $\nabla P\in L_{T}^{1}(\dot{B}_{p,2}^{\frac{n}{p}%
-\frac{n}{2}}\cap\dot{B}_{p,1}^{\frac{n}{p}-1})$ such that:%
\[
\left\Vert \left(  u,\nabla P\right)  \right\Vert _{E_{T}}=\left\Vert
u\right\Vert _{L_{T}^{\infty}(\dot{B}_{p,1}^{\frac{n}{p}-1})}+\left\Vert
\nabla^{2}u\right\Vert _{L_{T}^{1}(\dot{B}_{p,1}^{\frac{n}{p}-1})}+\left\Vert
\nabla P\right\Vert _{L_{T}^{1}(\dot{B}_{p,2}^{\frac{n}{p}-\frac{n}{2}}%
\cap\dot{B}_{p,1}^{\frac{n}{p}-1})}<\infty.
\]
The first ingredient in proving Theorem \ref{Teorema_Stokes} is the following:

\begin{proposition}
\label{Propozitia1Stokes} Let us consider $n\in\{2,3\}$ and
$p\in\left(  1,4\right)  $ if $n=2$ or \ $p\in\left(  \frac{6}{5},4\right)  $
if $n=3$. Assume there exists positive constants $\left(  a_{\star},b_{\star
},a^{\star},b^{\star},\bar{a},\bar{b}\right)  $ such that $a-\bar{a}\in\dot
{B}_{p,1}^{\frac{n}{p}}\left(  \mathbb{R}^{n}\right)  $, $b-\bar{b}\in\dot
{B}_{p,1}^{\frac{n}{p}}\left(  \mathbb{R}^{n}\right)  $ and
\begin{align*}
0  &  <a_{\star}\leq a\leq a^{\star},\\
0  &  <b_{\star}\leq b\leq b^{\star}.
\end{align*}
Furthermore, let us consider $u_{0},f$ vector fields with coefficients in
$\dot{B}_{p,1}^{\frac{n}{p}-1}\left(  \mathbb{R}^{n}\right)  $ respectively in
$L_{loc}^{1}(\dot{B}_{p,2}^{\frac{n}{p}-\frac{n}{2}}\left(  \mathbb{R}%
^{n}\right)  \cap\dot{B}_{p,1}^{\frac{n}{p}-1}\left(  \mathbb{R}^{n}\right)
)$ and a vector field $R\in(\mathcal{S}^{\prime}\left(  \mathbb{R}^{n}\right)
)^{n}$ with $\mathcal{Q}R\in\mathcal{C(}[0,\infty);\dot{B}_{p,1}^{\frac{n}%
{p}-1}\left(  \mathbb{R}^{n}\right)  )$ and $\left(  \partial_{t}%
R,\nabla\operatorname{div}R\right)  \in L_{loc}^{1}(\dot{B}_{p,2}^{\frac{n}%
{p}-\frac{n}{2}}\left(  \mathbb{R}^{n}\right)  \cap\dot{B}_{p,1}^{\frac{n}%
{p}-1}\left(  \mathbb{R}^{n}\right)  )$ such that%
\[
\operatorname{div}u_{0}=\operatorname{div}R\left(  0,\cdot\right)  .
\]
Then, there exists a constant $C_{ab}$ depending on $a$ and $b$ such that any
solution $\left(  u,\nabla P\right)  \in E_{T}$ of the Stokes system $\left(
\text{\ref{Stokes2}}\right)  $ will satisfy:%
\begin{gather}
\left\Vert u\right\Vert _{L_{t}^{\infty}(\dot{B}_{p,1}^{\frac{n}{p}-1}%
)}+\left\Vert \nabla^{2}u\right\Vert _{L_{t}^{1}(\dot{B}_{p,1}^{\frac{n}{p}%
-1})}+\left\Vert \nabla P\right\Vert _{L_{t}^{1}(\dot{B}_{p,2}^{\frac{n}%
{p}-\frac{n}{2}}\cap\dot{B}_{p,1}^{\frac{n}{p}-1})}\nonumber\\
\leq\left(  \left\Vert u_{0}\right\Vert _{\dot{B}_{p,1}^{\frac{n}{p}-1}%
}+\left\Vert \left(  f,\partial_{t}R,\nabla\operatorname{div}R\right)
\right\Vert _{L_{t}^{1}(\dot{B}_{p,2}^{\frac{n}{p}-\frac{n}{2}}\cap\dot
{B}_{p,1}^{\frac{n}{p}-1})}\right)  \exp\left(  C_{ab}(t+1)\right)
\label{Estimare_Prop1}%
\end{gather}
for all $t\in(0,T]$.
\end{proposition}

Before proceeding with the proof, a few remarks are in order:

\begin{remark}
Proposition \ref{Propozitia1Stokes} is different from Theorem
\ref{Teorema_Stokes} when $n=3$. Indeed, in the $3$ dimensional case the
theory is more subtle and thus, as a first step we construct a unique solution
for the case of more regular initial data.
\end{remark}

\begin{remark}
\label{111} The difficulty when dealing with the Stokes system with non
constant coefficients lies in the fact that the pressure and the velocity $u$
are coupled. Indeed, in the constant coefficients case, in view of%
\[
\operatorname{div}u=\operatorname{div}R,
\]
one can apply the divergence operator in the first equation of $\left(
\text{\ref{Stokes2}}\right)  $ in order to obtain the following elliptic
equation verified by the pressure:%
\begin{equation}
a\Delta P=\operatorname{div}\left(  f-\partial_{t}R+2ab\nabla
\operatorname{div}R\right)  . \label{eliptic1}%
\end{equation}
From $\left(  \text{\ref{eliptic1}}\right)  $ we can construct the pressure.
Having built the pressure, the velocity satisfies a classical heat equation.
In the non constant coefficient case, proceeding as above we find that:%
\begin{equation}
\operatorname{div}\left(  a\nabla P\right)  =\operatorname{div}\left(
f-\partial_{t}R+a\operatorname{div}(bD(u))\right)  . \label{eliptic2}%
\end{equation}
such that the strategy used in the previous case is not well-adapted. We will
establish a priori estimate and use a continuity argument like in \cite{Dan2}.
In order to be able to close the estimates on $u$, we have to bound
$\left\Vert a\nabla P\right\Vert _{L_{t}^{1}(\dot{B}_{p,1}^{\frac{n}{p}-1})}$
in terms of $\left\Vert u\right\Vert _{L_{t}^{\infty}(\dot{B}_{p,1}^{\frac
{n}{p}-1})}^{\beta}\left\Vert \nabla^{2}u\right\Vert _{L_{t}^{1}(\dot{B}%
_{p,1}^{\frac{n}{p}-1})}^{1-\beta}$ for some $\beta\in(0,1)$. Thus, the
difficulty is to find estimates for the pressure which do not feature the time
derivative of the velocity.
\end{remark}

In view of Proposition \ref{Stokes_constant}, let us consider $\left(
u_{L},\nabla P_{L}\right)  $ the unique solution of the system
\begin{equation}
\left\{
\begin{array}
[c]{r}%
\partial_{t}u-\bar{a}\operatorname{div}(\bar{b}D(u))+\bar{a}\nabla P=f,\\
\operatorname{div}u=\operatorname{div}R,\\
u_{|t=0}=u_{0},
\end{array}
\right.  \label{LiniarStokes}%
\end{equation}
with%
\[
u_{L}\in\mathcal{C(}[0,\infty);\dot{B}_{p,1}^{\frac{n}{p}-1})\text{ and
}(\partial_{t}u_{L},\nabla^{2}u_{L},\nabla P_{L})\in L_{loc}^{1}(\dot{B}%
_{p,1}^{\frac{n}{p}-1}).
\]
Recall that for any $t\in\lbrack0,\infty)$ we have%
\begin{equation}
\left\Vert u_{L}\right\Vert _{L_{t}^{\infty}(\dot{B}_{p,1}^{\frac{n}{p}-1}%
)}+\left\Vert \left(  \partial_{t}u_{L},\bar{a}\bar{b}\nabla^{2}u_{L},\bar
{a}\nabla P_{L}\right)  \right\Vert _{L_{t}^{1}(\dot{B}_{p,1}^{\frac{n}{p}%
-1})}\leq C(\left\Vert u_{0}\right\Vert _{\dot{B}_{p,1}^{\frac{n}{p}-1}%
}+\left\Vert \left(  f,\partial_{t}R,\bar{a}\bar{b}\nabla\operatorname{div}%
R\right)  \right\Vert _{L_{t}^{1}(\dot{B}_{p,1}^{\frac{n}{p}-1})}).
\label{estimare_uL}%
\end{equation}
In what follows, we will use the notation:
\begin{equation}
\tilde{u}=u-u_{L},\text{ }\nabla\tilde{P}=\nabla P-\nabla P_{L}.
\label{utilde}%
\end{equation}
Obviously, we have%
\begin{equation}
\operatorname{div}\tilde{u}=0\text{.} \label{RelatieImportanta}%
\end{equation}
Thus, the system $\left(  \text{\ref{Stokes2}}\right)  $ is recasted into%
\begin{equation}
\left\{
\begin{array}
[c]{r}%
\partial_{t}\tilde{u}-a\operatorname{div}\left(  bD(\tilde{u})\right)
+a\nabla\tilde{P}=\tilde{f},\\
\operatorname{div}\tilde{u}=0,\\
\tilde{u}_{|t=0}=0,
\end{array}
\right.  \label{StokesModif}%
\end{equation}
where
\[
\tilde{f}=a\operatorname{div}(bD(u_{L}))-\bar{a}\operatorname{div}(\bar
{b}D(u_{L}))-(a-\bar{a})\nabla P_{L}.
\]
Using the last equality along with Proposition \ref{Produs}, we infer that:%
\begin{align}
\left\Vert \tilde{f}\right\Vert _{\dot{B}_{p,1}^{\frac{n}{p}-1}}  &
\leq\left\Vert a\operatorname{div}(bD(u_{L}))-\bar{a}\operatorname{div}%
(\bar{b}D(u_{L}))\right\Vert _{\dot{B}_{p,1}^{\frac{n}{p}-1}}+\left\Vert
(a-\bar{a})\nabla P_{L}\right\Vert _{\dot{B}_{p,1}^{\frac{n}{p}-1}}\nonumber\\
&  \lesssim(\bar{a}+\left\Vert a-\bar{a}\right\Vert _{\dot{B}_{p,1}^{\frac
{n}{p}}})(\bar{b}+\left\Vert b-\bar{b}\right\Vert _{\dot{B}_{p,1}^{\frac{n}%
{p}}})\left\Vert \nabla u_{L}\right\Vert _{\dot{B}_{p,1}^{\frac{n}{p}}%
}+\left\Vert a-\bar{a}\right\Vert _{\dot{B}_{p,1}^{\frac{n}{p}}}\left\Vert
\nabla P_{L}\right\Vert _{\dot{B}_{p,1}^{\frac{n}{p}-1}}.
\label{estimare_Besov_f}%
\end{align}
Let us estimate the pressure $a\nabla\tilde{P}$. First, we write that%
\[
\left\Vert a\nabla\tilde{P}\right\Vert _{\dot{B}_{p,1}^{\frac{n}{p}-1}}%
\leq\left\Vert \mathcal{Q}\left(  a\nabla\tilde{P}\right)  \right\Vert
_{\dot{B}_{p,1}^{\frac{n}{p}-1}}+\left\Vert \mathcal{P}\left(  a\nabla
\tilde{P}\right)  \right\Vert _{\dot{B}_{p,1}^{\frac{n}{p}-1}}.
\]
Applying the $\mathcal{Q}$ operator in the first equation of $\left(
\text{\ref{StokesModif}}\right)  $ we get that%
\[
\mathcal{Q}\left(  a\nabla\tilde{P}\right)  =\mathcal{Q}\tilde{f}%
+\mathcal{Q}(a\operatorname{div}\left(  bD(\tilde{u})\right)  ).
\]
Thus, we get that:%
\begin{equation}
\left\Vert \mathcal{Q}\left(  a\nabla\tilde{P}\right)  \right\Vert _{\dot
{B}_{p,1}^{\frac{n}{p}-1}}\leq\left\Vert \mathcal{Q}\tilde{f}\right\Vert
_{\dot{B}_{p,1}^{\frac{n}{p}-1}}+\left\Vert \mathcal{Q}\left(
a\operatorname{div}\left(  bD(\tilde{u}\right)  )\right)  \right\Vert
_{\dot{B}_{p,1}^{\frac{n}{p}-1}}. \label{compresibil1}%
\end{equation}
Let write that:%
\begin{align}
\mathcal{Q}\left(  a\operatorname{div}\left(  bD(\tilde{u})\right)  \right)
&  =\mathcal{Q}\left(  D(\tilde{u})\dot{S}_{m}\left(  a\nabla b\right)
\right)  +\mathcal{Q}\left(  \dot{S}_{m}\left(  ab-\bar{a}\bar{b}\right)
\Delta\tilde{u}\right) \label{compresibil2.1}\\
&  +\mathcal{Q}\left(  D(\tilde{u})\left(  Id-\dot{S}_{m}\right)  \left(
a\nabla b\right)  \right) \label{compresibil2.2}\\
&  +\mathcal{Q}\left(  \left(  Id-\dot{S}_{m}\right)  \left(  ab-\bar{a}%
\bar{b}\right)  \Delta\tilde{u}\right)  . \label{compresibil2.3}%
\end{align}
According to Proposition \ref{Produs} we have:%
\begin{equation}
\left\Vert \mathcal{Q}\left(  D\left(  \tilde{u}\right)  \dot{S}_{m}\left(
a\nabla b\right)  \right)  \right\Vert _{\dot{B}_{p,1}^{\frac{n}{p}-1}%
}\lesssim\left\Vert \dot{S}_{m}\left(  a\nabla b\right)  \right\Vert _{\dot
{B}_{p,1}^{\frac{n}{p}-\frac{1}{2}}}\left\Vert \nabla\tilde{u}\right\Vert
_{\dot{B}_{p,1}^{\frac{n}{p}-\frac{1}{2}}}. \label{compresibil3}%
\end{equation}
Owing to the fact that $\tilde{u}$ is divergence free we can write that%
\begin{equation}
\mathcal{Q}\left(  \dot{S}_{m}\left(  ab-\bar{a}\bar{b}\right)  \Delta
\tilde{u}\right)  =[\mathcal{Q},\dot{S}_{m}\left(  ab-\bar{a}\bar{b}\right)
]\Delta\tilde{u}, \label{smecherie1}%
\end{equation}
such that applying Proposition \ref{A7} we get that%
\begin{align}
\left\Vert \mathcal{Q}\left(  \dot{S}_{m}\left(  ab-\bar{a}\bar{b}\right)
\Delta\tilde{u}\right)  \right\Vert _{\dot{B}_{p,1}^{\frac{n}{p}-1}}  &
\lesssim\left\Vert \left(  \dot{S}_{m}\left(  a\nabla b\right)  ,\dot{S}%
_{m}\left(  b\nabla a\right)  \right)  \right\Vert _{\dot{B}_{p,1}^{\frac
{n}{p}-\frac{1}{2}}}\left\Vert \Delta\tilde{u}\right\Vert _{\dot{B}%
_{p,1}^{\frac{n}{p}-\frac{3}{2}}}\nonumber\\
&  \lesssim\left\Vert \left(  \dot{S}_{m}\left(  a\nabla b\right)  ,\dot
{S}_{m}\left(  b\nabla a\right)  \right)  \right\Vert _{\dot{B}_{p,1}%
^{\frac{n}{p}-\frac{1}{2}}}\left\Vert \nabla\tilde{u}\right\Vert _{\dot
{B}_{p,1}^{\frac{n}{p}-\frac{1}{2}}}. \label{compresibil4}%
\end{align}
The last two terms of $\left(  \text{\ref{compresibil2.1}}\right)  $-$\left(
\text{\ref{compresibil2.3}}\right)  $ are estimated as follows:%
\begin{gather}
\left\Vert \mathcal{Q}\left(  \left(  Id-\dot{S}_{m}\right)  \left(  a\nabla
b\right)  D\left(  \tilde{u}\right)  \right)  +\mathcal{Q}\left(  \left(
Id-\dot{S}_{m}\right)  \left(  ab-\bar{a}\bar{b}\right)  \Delta\tilde
{u}\right)  \right\Vert _{\dot{B}_{p,1}^{\frac{n}{p}-1}} \label{comresibil5.1}%
\\
\lesssim\left(  \left\Vert \left(  Id-\dot{S}_{m}\right)  \left(  a\nabla
b\right)  \right\Vert _{\dot{B}_{p,1}^{\frac{n}{p}-1}}+\left\Vert \left(
Id-\dot{S}_{m}\right)  \left(  ab-\bar{a}\bar{b}\right)  \right\Vert _{\dot
{B}_{p,1}^{\frac{n}{p}}}\right)  \left\Vert \nabla\tilde{u}\right\Vert
_{\dot{B}_{p,1}^{\frac{n}{p}}}. \label{comresibil5.3}%
\end{gather}
Thus, putting together relations $\left(  \text{\ref{compresibil1}}\right)
$-$\left(  \text{\ref{comresibil5.3}}\right)  $ we get that:%
\begin{align}
\left\Vert \mathcal{Q}\left(  a\nabla\tilde{P}\right)  \right\Vert _{\dot
{B}_{p,1}^{\frac{n}{p}-1}}  &  \lesssim\left\Vert \mathcal{Q}\tilde
{f}\right\Vert _{\dot{B}_{p,1}^{\frac{n}{p}-1}}+\left\Vert \left(  \dot{S}%
_{m}\left(  a\nabla b\right)  ,\dot{S}_{m}\left(  b\nabla a\right)  \right)
\right\Vert _{\dot{B}_{p,1}^{\frac{n}{p}-\frac{1}{2}}}\left\Vert \nabla
\tilde{u}\right\Vert _{\dot{B}_{p,1}^{\frac{n}{p}-\frac{1}{2}}}\nonumber\\
&  +\left\Vert \nabla\tilde{u}\right\Vert _{\dot{B}_{p,1}^{\frac{n}{p}}%
}\left(  \left\Vert \left(  Id-\dot{S}_{m}\right)  \left(  a\nabla b,b\nabla
a\right)  \right\Vert _{\dot{B}_{p,1}^{\frac{n}{p}-1}}+\left\Vert \left(
Id-\dot{S}_{m}\right)  (ab-\bar{a}\bar{b})\right\Vert _{\dot{B}_{p,1}%
^{\frac{n}{p}}}\right)  . \label{estimare_parte_compresibila}%
\end{align}
Next, we turn our attention towards $\mathcal{P}\left(  a\nabla\tilde
{P}\right)  $. The $2D$ case respectively the $3D$ case have to be treated differently.

\subsubsection{The $3D$ case}

Noticing that%
\[
\mathcal{P}\left(  a\nabla\tilde{P}\right)  =\mathcal{P}\left(  \left(
Id-\dot{S}_{m}\right)  \left(  a-\bar{a}\right)  \nabla\tilde{P}\right)
+[\mathcal{P},\dot{S}_{m}\left(  a-\bar{a}\right)  ]\nabla\tilde{P},
\]
and using again Proposition \ref{A7} combined with Proposition
\ref{Estimare_eliptica1} and Proposition \ref{Estimare_eliptica2} we get that:%

\begin{align}
\left\Vert \nabla\tilde{P}\right\Vert _{\dot{B}_{p,2}^{\frac{3}{p}-\frac{3}%
{2}}}+\left\Vert \mathcal{P}\left(  a\nabla\tilde{P}\right)  \right\Vert
_{\dot{B}_{p,1}^{\frac{3}{p}-1}}  &  \lesssim\left\Vert \nabla\tilde
{P}\right\Vert _{\dot{B}_{p,2}^{\frac{3}{p}-\frac{3}{2}}}+\left\Vert
\mathcal{P}\left(  \left(  Id-\dot{S}_{m}\right)  \left(  a-\bar{a}\right)
\nabla\tilde{P}\right)  \right\Vert _{\dot{B}_{p,1}^{\frac{3}{p}-1}%
}+\left\Vert [\mathcal{P},\dot{S}_{m}\left(  a-\bar{a}\right)  ]\nabla
\tilde{P}\right\Vert _{\dot{B}_{p,1}^{\frac{3}{p}-1}}\label{incompresibil1}\\
&  \lesssim\left\Vert \left(  Id-\dot{S}_{m}\right)  \left(  a-\bar{a}\right)
\right\Vert _{\dot{B}_{p,1}^{\frac{3}{p}}}\left\Vert \nabla\tilde
{P}\right\Vert _{\dot{B}_{p,1}^{\frac{3}{p}-1}}+\left(  1+\left\Vert \dot
{S}_{m}\nabla a\right\Vert _{\dot{B}_{p,2}^{\frac{3}{p}-\frac{1}{2}}}\right)
\left\Vert \nabla\tilde{P}\right\Vert _{\dot{B}_{p,2}^{\frac{3}{p}-\frac{3}%
{2}}}\label{incompresibil1.0}\\
&  \lesssim\left\Vert \left(  Id-\dot{S}_{m}\right)  (a-\bar{a})\right\Vert
_{\dot{B}_{p,1}^{\frac{3}{p}}}\left(  \frac{1}{\bar{a}}+\left\Vert \frac{1}%
{a}-\frac{1}{\bar{a}}\right\Vert _{\dot{B}_{p,1}^{\frac{3}{p}}}\right)
\left\Vert a\nabla\tilde{P}\right\Vert _{\dot{B}_{p,1}^{\frac{3}{p}-1}%
}\label{incompresibil1.2}\\
+  &  \tilde{C}\left(  a\right)  \left(  1+\left\Vert \dot{S}_{m}\nabla
a\right\Vert _{\dot{B}_{p,2}^{\frac{3}{p}-\frac{1}{2}}}\right)  \left(
\left\Vert \tilde{f}\right\Vert _{\dot{B}_{p,2}^{\frac{3}{p}-\frac{3}{2}}%
}+\left\Vert a\operatorname{div}\left(  bD(\tilde{u})\right)  \right\Vert
_{\dot{B}_{p,2}^{\frac{3}{p}-\frac{3}{2}}}\right)  , \label{incompresibil1.3}%
\end{align}
where%
\[
\tilde{C}\left(  a\right)  =\left(  \frac{1}{\bar{a}}+\left\Vert \frac{1}%
{a}-\frac{1}{\bar{a}}\right\Vert _{\dot{B}_{p,1}^{\frac{3}{p}}}\right)
\left(  1+\frac{1}{a_{\star}}\left\Vert a-\bar{a}\right\Vert _{\dot{B}%
_{p,1}^{\frac{3}{p}}}\right)  .
\]
We observe that%
\begin{equation}
\left\Vert a\operatorname{div}\left(  bD(\tilde{u})\right)  \right\Vert
_{\dot{B}_{p,2}^{\frac{3}{p}-\frac{3}{2}}}\lesssim\left(  \bar{a}+\left\Vert
a-\bar{a}\right\Vert _{\dot{B}_{p,1}^{\frac{3}{p}}}\right)  \left(  \bar
{b}+\left\Vert b-\bar{b}\right\Vert _{\dot{B}_{p,1}^{\frac{3}{p}}}\right)
\left\Vert \nabla\tilde{u}\right\Vert _{\dot{B}_{p,1}^{\frac{3}{p}-\frac{1}%
{2}}}. \label{incompresibil2.2}%
\end{equation}
Putting together $\left(  \text{\ref{incompresibil1}}\right)  $-$\left(
\text{\ref{incompresibil1.3}}\right)  $ along with $\left(
\text{\ref{incompresibil2.2}}\right)  $ we get that%
\begin{gather}
\left\Vert \nabla\tilde{P}\right\Vert _{\dot{B}_{p,2}^{\frac{3}{p}-\frac{3}%
{2}}}+\left\Vert \mathcal{P}\left(  a\nabla\tilde{P}\right)  \right\Vert
_{\dot{B}_{p,1}^{\frac{3}{p}-1}}\lesssim\left\Vert \left(  Id-\dot{S}%
_{m}\right)  (a-\bar{a})\right\Vert _{\dot{B}_{p,1}^{\frac{3}{p}}}\left(
\frac{1}{\bar{a}}+\left\Vert \frac{1}{a}-\frac{1}{\bar{a}}\right\Vert
_{\dot{B}_{p,1}^{\frac{3}{p}}}\right)  \left\Vert a\nabla\tilde{P}\right\Vert
_{\dot{B}_{p,1}^{\frac{3}{p}-1}}\nonumber\\
+\tilde{C}\left(  a\right)  \left(  1+\left\Vert \dot{S}_{m}\nabla
a\right\Vert _{\dot{B}_{p,2}^{\frac{3}{p}-\frac{1}{2}}}\right)  \left(
\left\Vert \tilde{f}\right\Vert _{\dot{B}_{p,2}^{\frac{3}{p}-\frac{3}{2}}%
}+\left(  \bar{a}+\left\Vert a-\bar{a}\right\Vert _{\dot{B}_{p,1}^{\frac{3}%
{p}}}\right)  \left(  \bar{b}+\left\Vert b-\bar{b}\right\Vert _{\dot{B}%
_{p,1}^{\frac{3}{p}}}\right)  \left\Vert \nabla\tilde{u}\right\Vert _{\dot
{B}_{p,1}^{\frac{3}{p}-\frac{1}{2}}}\right)  .
\label{estimare_parte_incompresibila_presiune}%
\end{gather}
Combining $\left(  \text{\ref{estimare_parte_compresibila}}\right)  $ with
$\left(  \text{\ref{estimare_parte_incompresibila_presiune}}\right)  $ yields:%
\begin{gather*}
\left\Vert \nabla\tilde{P}\right\Vert _{\dot{B}_{p,2}^{\frac{3}{p}-\frac{3}%
{2}}}+\left\Vert a\nabla\tilde{P}\right\Vert _{\dot{B}_{p,1}^{\frac{3}{p}-1}%
}\lesssim T_{m}^{1}\left(  a,b\right)  \left\Vert a\nabla\tilde{P}\right\Vert
_{\dot{B}_{p,1}^{\frac{3}{p}-1}}\\
+T_{m}^{2}\left(  a,b\right)  \left\Vert \tilde{f}\right\Vert _{\dot{B}%
_{p,2}^{\frac{3}{p}-\frac{3}{2}}\cap\dot{B}_{p,1}^{\frac{3}{p}-1}}+T_{m}%
^{3}\left(  a,b\right)  \left\Vert \nabla\tilde{u}\right\Vert _{\dot{B}%
_{p,1}^{\frac{3}{p}-\frac{1}{2}}}+T_{m}^{4}\left(  a,b\right)  \left\Vert
\nabla\tilde{u}\right\Vert _{\dot{B}_{p,1}^{\frac{3}{p}}}%
\end{gather*}
where%
\begin{align*}
T_{m}^{1}\left(  a,b\right)   &  =\left\Vert \left(  Id-\dot{S}_{m}\right)
(a-\bar{a})\right\Vert _{\dot{B}_{p,1}^{\frac{3}{p}}}\left(  \frac{1}{\bar{a}%
}+\left\Vert \frac{1}{a}-\frac{1}{\bar{a}}\right\Vert _{\dot{B}_{p,1}%
^{\frac{3}{p}}}\right)  ,\\
T_{m}^{2}\left(  a,b\right)   &  =\tilde{C}\left(  a\right)  \left(
1+\left\Vert \dot{S}_{m}\nabla a\right\Vert _{\dot{B}_{p,2}^{\frac{3}{p}%
-\frac{1}{2}}}\right)  ,\\
T_{m}^{3}\left(  a,b\right)   &  =\left\Vert \left(  \dot{S}_{m}\left(
a\nabla b\right)  ,\dot{S}_{m}\left(  b\nabla a\right)  \right)  \right\Vert
_{\dot{B}_{p,1}^{\frac{3}{p}-\frac{1}{2}}}\\
&  +\tilde{C}\left(  a\right)  \left(  1+\left\Vert \dot{S}_{m}\nabla
a\right\Vert _{\dot{B}_{p,2}^{\frac{3}{p}-\frac{1}{2}}}\right)  \left(
\bar{a}+\left\Vert a-\bar{a}\right\Vert _{\dot{B}_{p,1}^{\frac{3}{p}}}\right)
\left(  \bar{b}+\left\Vert b-\bar{b}\right\Vert _{\dot{B}_{p,1}^{\frac{3}{p}}%
}\right)  ,\\
T_{m}^{4}\left(  a,b\right)   &  =\left\Vert \left(  Id-\dot{S}_{m}\right)
\left(  a\nabla b,b\nabla a\right)  \right\Vert _{\dot{B}_{p,1}^{\frac{3}%
{p}-1}}+\left\Vert \left(  Id-\dot{S}_{m}\right)  (ab-\bar{a}\bar
{b})\right\Vert _{\dot{B}_{p,1}^{\frac{3}{p}}}.
\end{align*}
Observe that $m$ could be chosen large enough such that $T_{m}^{1}\left(
a,b\right)  $ and $T_{m}^{4}\left(  a,b\right)  $ can be made arbitrarily
small. Thus, there exists a constant $C_{ab}$ depending on $a$ and $b$ such
that:%
\begin{equation}
\left\Vert \nabla\tilde{P}\right\Vert _{\dot{B}_{p,2}^{\frac{3}{p}-\frac{3}%
{2}}\cap\dot{B}_{p,1}^{\frac{3}{p}-1}}\leq C_{ab}\left(  \left\Vert \tilde
{f}\right\Vert _{\dot{B}_{p,2}^{\frac{3}{p}-\frac{3}{2}}\cap\dot{B}%
_{p,1}^{\frac{3}{p}-1}}+\left\Vert \nabla\tilde{u}\right\Vert _{\dot{B}%
_{p,1}^{\frac{3}{p}-\frac{1}{2}}}\right)  +\eta\left\Vert \nabla\tilde
{u}\right\Vert _{\dot{B}_{p,1}^{\frac{3}{p}}}, \label{Presiune1.1}%
\end{equation}
where $\eta$ can be made arbitrarily small (of course, with the price of
increasing the constant $C_{ab}$). Let us take a look at the $\dot{B}%
_{p,2}^{\frac{3}{p}-\frac{3}{2}}$-norm of $\tilde{f}$; we get that:%
\begin{align}
\left\Vert \tilde{f}\right\Vert _{\dot{B}_{p,2}^{\frac{3}{p}-\frac{3}{2}}}  &
\leq\left\Vert a\operatorname{div}(bD(u_{L}))-\bar{a}\operatorname{div}%
(\bar{b}D(u_{L}))\right\Vert _{\dot{B}_{p,2}^{\frac{3}{p}-\frac{3}{2}}%
}+\left\Vert (a-\bar{a})\nabla P_{L}\right\Vert _{\dot{B}_{p,2}^{\frac{3}%
{p}-\frac{3}{2}}}\nonumber\\
&  \lesssim\left(  \bar{a}+\left\Vert a-\bar{a}\right\Vert _{\dot{B}%
_{p,1}^{\frac{3}{p}}}\right)  \left(  \bar{b}+\left\Vert b-\bar{b}\right\Vert
_{\dot{B}_{p,1}^{\frac{3}{p}}}\right)  \left\Vert \nabla u_{L}\right\Vert
_{\dot{B}_{p,1}^{\frac{3}{p}-\frac{1}{2}}}+\left\Vert a-\bar{a}\right\Vert
_{\dot{B}_{p,1}^{\frac{3}{p}}}\left\Vert \nabla P_{L}\right\Vert _{\dot
{B}_{p,2}^{\frac{3}{p}-\frac{3}{2}}}. \label{estimare_L2_f.}%
\end{align}
As $u_{L}\in C([0,\infty),\dot{B}_{p,1}^{\frac{3}{p}-1})\cap L^{1}%
([0,\infty),\dot{B}_{p,1}^{\frac{3}{p}+1})$ and $\mathcal{Q}$ is continuous
operator on homogeneous Besov spaces from%
\[
\operatorname{div}\left(  u_{L}-R\right)  =0,
\]
we deduce that%
\[
\mathcal{P}\left(  u_{L}-R\right)  =u_{L}-R,
\]
which implies%
\[
\mathcal{Q}u_{L}=\mathcal{Q}R.
\]
By applying the operator $\mathcal{Q}$ in the first equation of System
$\left(  \text{\ref{LiniarStokes}}\right)  $ we get that:%
\begin{align*}
\bar{a}\nabla P_{L}  &  =\mathcal{Q}f-\mathcal{Q}\partial_{t}u_{L}+\bar{a}%
\bar{b}\mathcal{Q}\Delta u_{L}+\bar{a}\bar{b}\nabla\operatorname{div}R\\
&  =\mathcal{Q}f-\mathcal{Q}\partial_{t}R+2\bar{a}\bar{b}\nabla
\operatorname{div}R
\end{align*}
and thus%
\[
\left\Vert \nabla P_{L}\right\Vert _{\dot{B}_{p,2}^{\frac{3}{p}-\frac{3}{2}}%
}\leq\frac{1}{\bar{a}}\left\Vert \mathcal{Q}f\right\Vert _{\dot{B}%
_{p,2}^{\frac{3}{p}-\frac{3}{2}}}+\frac{1}{\bar{a}}\left\Vert \partial
_{t}\mathcal{Q}R\right\Vert _{\dot{B}_{p,2}^{\frac{3}{p}-\frac{3}{2}}}%
+2\bar{b}\left\Vert \nabla\operatorname{div}R\right\Vert _{\dot{B}%
_{p,2}^{\frac{3}{p}-\frac{3}{2}}}.
\]
In view of $\left(  \text{\ref{Presiune1.1}}\right)  $, $\left(
\text{\ref{estimare_Besov_f}}\right)  $, $\left(  \text{\ref{estimare_L2_f.}%
}\right)  $ and interpolation we gather that there exists a constant $C_{ab}$
such that:%
\begin{align}
\left\Vert \nabla\tilde{P}\right\Vert _{\dot{B}_{p,2}^{\frac{3}{p}-\frac{3}%
{2}}\cap\dot{B}_{p,1}^{\frac{3}{p}-1}}  &  \leq C_{ab}\left(  \left\Vert
\nabla u_{L}\right\Vert _{\dot{B}_{p,1}^{\frac{3}{p}-\frac{1}{2}}}+\left\Vert
\nabla P_{L}\right\Vert _{\dot{B}_{p,2}^{\frac{3}{p}-\frac{3}{2}}}+\left\Vert
(\nabla^{2}u_{L},\nabla P_{L})\right\Vert _{\dot{B}_{p,1}^{\frac{3}{p}-1}%
}+\left\Vert \nabla\tilde{u}\right\Vert _{\dot{B}_{p,1}^{\frac{3}{p}-\frac
{1}{2}}}\right)  +\eta\left\Vert \nabla\tilde{u}\right\Vert _{\dot{B}%
_{p,1}^{\frac{3}{p}}}.\label{Presiune_3d_1}\\
&  \leq C_{ab}\left\Vert (\mathcal{Q}f,\partial_{t}\mathcal{Q}R,\nabla
\operatorname{div}R)\right\Vert _{\dot{B}_{p,2}^{\frac{3}{p}-\frac{3}{2}}%
}+C_{ab}\left\Vert u_{L}\right\Vert _{\dot{B}_{p,1}^{\frac{3}{p}-1}%
}\label{Presiune_3d_3}\\
&  \text{ \ \ }+C_{ab}\left\Vert (\nabla^{2}u_{L},\nabla P_{L})\right\Vert
_{\dot{B}_{p,1}^{\frac{3}{p}-1}}+C_{ab}\left\Vert \tilde{u}\right\Vert
_{\dot{B}_{p,1}^{\frac{3}{p}-1}}+2\eta\left\Vert \nabla\tilde{u}\right\Vert
_{\dot{B}_{p,1}^{\frac{3}{p}}}. \label{Presiune_3d_4}%
\end{align}
where, again, at the price of increasing $C_{ab}$, $\eta$ can be made
arbitrarily small.

\subsubsection{The $2D$ case}

In this case, using again Proposition \ref{A7} combined with Proposition with
Proposition \ref{Estimare_eliptica1_2D} and Proposition
\ref{Estimare_eliptica2_2D} we get that:%
\begin{align*}
\left\Vert \nabla\tilde{P}\right\Vert _{\dot{B}_{p,2}^{\frac{2}{p}-1}%
}+\left\Vert \mathcal{P(}a\nabla\tilde{P})\right\Vert _{\dot{B}_{p,1}%
^{\frac{2}{p}-1}}  &  \lesssim\left\Vert \nabla\tilde{P}\right\Vert _{\dot
{B}_{p,2}^{\frac{2}{p}-1}}+\left\Vert \left(  Id-\dot{S}_{m}\right)  \left(
a-\bar{a}\right)  \right\Vert _{\dot{B}_{p,1}^{\frac{2}{p}}}\left\Vert
\nabla\tilde{P}\right\Vert _{\dot{B}_{p,1}^{\frac{2}{p}-1}}+\left\Vert
[\mathcal{P},\dot{S}_{m}\left(  a-\bar{a}\right)  ]\nabla\tilde{P}\right\Vert
_{\dot{B}_{p,1}^{\frac{2}{p}-1}}\\
&  \lesssim\left\Vert \left(  Id-\dot{S}_{m}\right)  \left(  a-\bar{a}\right)
\right\Vert _{\dot{B}_{p,1}^{\frac{2}{p}}}\left\Vert \nabla\tilde
{P}\right\Vert _{\dot{B}_{p,1}^{\frac{2}{p}-1}}+\left(  1+\left\Vert
\nabla\dot{S}_{m}a\right\Vert _{\dot{B}_{p,2}^{\frac{2}{p}}}\right)
\left\Vert \nabla\tilde{P}\right\Vert _{\dot{B}_{p,2}^{\frac{2}{p}-1}}\\
&  \lesssim\left\Vert \left(  Id-\dot{S}_{m}\right)  \left(  a-\bar{a}\right)
\right\Vert _{\dot{B}_{p,1}^{\frac{2}{p}}}\left\Vert \nabla\tilde
{P}\right\Vert _{\dot{B}_{p,1}^{\frac{2}{p}-1}}\\
&  +\tilde{C}\left(  a\right)  \left(  1+\left\Vert \nabla\dot{S}%
_{m}a\right\Vert _{\dot{B}_{p,2}^{\frac{2}{p}}}\right)  \left(  \left\Vert
\tilde{f}\right\Vert _{\dot{B}_{2,2}^{\frac{2}{p}-1}}+\left\Vert
\mathcal{Q(}a\operatorname{div}\left(  bD(\tilde{u})\right)  )\right\Vert
_{\dot{B}_{p,2}^{\frac{2}{p}-1}}\right)
\end{align*}
where, as before%
\[
\tilde{C}\left(  a\right)  =\left(  \frac{1}{\bar{a}}+\left\Vert \frac{1}%
{a}-\frac{1}{\bar{a}}\right\Vert _{\dot{B}_{p,1}^{\frac{2}{p}}}\right)
\left(  1+\frac{1}{a_{\star}}\left\Vert a-\bar{a}\right\Vert _{\dot{B}%
_{p,1}^{\frac{2}{p}}}\right)  .
\]
As we have already estimated $\left\Vert \mathcal{Q(}a\operatorname{div}%
\left(  bD(\tilde{u})\right)  )\right\Vert _{\dot{B}_{p,2}^{\frac{2}{p}-1}}$
in $\left(  \text{\ref{estimare_parte_compresibila}}\right)  $, we gather that:%

\begin{align}
\left\Vert \nabla\tilde{P}\right\Vert _{\dot{B}_{p,2}^{\frac{2}{p}-1}%
}+\left\Vert a\nabla\tilde{P}\right\Vert _{\dot{B}_{p,1}^{\frac{2}{p}}}  &
\lesssim T_{m}^{1}\left(  a,b\right)  \left\Vert a\nabla\tilde{P}\right\Vert
_{\dot{B}_{p,1}^{\frac{2}{p}-1}}+T_{m}^{2}\left(  a,b\right)  \left\Vert
\tilde{f}\right\Vert _{\dot{B}_{p,1}^{\frac{2}{p}-1}}\nonumber\\
&  +T_{m,M}^{3}\left(  a,b\right)  \left\Vert \nabla\tilde{u}\right\Vert
_{\dot{B}_{p,1}^{\frac{2}{p}-\frac{1}{2}}}+T_{m,M}^{4}\left(  a,b\right)
\left\Vert \nabla\tilde{u}\right\Vert _{\dot{B}_{p,1}^{\frac{2}{p}}},
\label{presiune_2d}%
\end{align}
where%
\begin{align*}
T_{m}^{1}\left(  a,b\right)   &  =\left\Vert \left(  Id-\dot{S}_{m}\right)
\left(  a-\bar{a}\right)  \right\Vert _{\dot{B}_{p,1}^{\frac{2}{p}}}\left(
\frac{1}{\bar{a}}+\left\Vert \frac{1}{a}-\frac{1}{\bar{a}}\right\Vert
_{\dot{B}_{p,1}^{\frac{2}{p}}}\right)  ,\\
T_{m}^{2}\left(  a,b\right)   &  =\tilde{C}\left(  a\right)  \left(
1+\left\Vert \nabla\dot{S}_{m}a\right\Vert _{\dot{B}_{p,2}^{\frac{2}{p}}%
}\right)  ,\\
T_{m,M}^{3}\left(  a,b\right)   &  =\left\Vert \left(  \dot{S}_{m}\left(
a\nabla b\right)  ,\dot{S}_{m}\left(  b\nabla a\right)  \right)  \right\Vert
_{\dot{B}_{p,1}^{\frac{2}{p}}}\\
+\tilde{C}\left(  a\right)   &  \left(  1+\left\Vert \nabla\dot{S}%
_{m}a\right\Vert _{\dot{B}_{p,2}^{\frac{2}{p}}}\right)  \left\Vert \left(
\dot{S}_{M}\left(  a\nabla b\right)  ,\dot{S}_{M}\left(  b\nabla a\right)
\right)  \right\Vert _{\dot{B}_{p,1}^{\frac{2}{p}}},\\
T_{m,M}^{4}\left(  a,b\right)   &  =\left\Vert (Id-\dot{S}_{m})(a\nabla
b)\right\Vert _{\dot{B}_{p,1}^{\frac{2}{p}-1}}+\left\Vert (Id-\dot{S}%
_{m})\left(  ab-\bar{a}\bar{b}\right)  \right\Vert _{\dot{B}_{p,1}^{\frac
{2}{p}}}\\
&  +\tilde{C}\left(  a\right)  \left(  1+\left\Vert \nabla\dot{S}%
_{m}a\right\Vert _{\dot{B}_{p,2}^{\frac{2}{p}}}\right)  \left(  \left\Vert
(Id-\dot{S}_{M})(a\nabla b)\right\Vert _{\dot{B}_{p,1}^{\frac{2}{p}-1}%
}+\left\Vert (Id-\dot{S}_{M})\left(  ab-\bar{a}\bar{b}\right)  \right\Vert
_{\dot{B}_{p,1}^{\frac{2}{p}-1}}\right)  .
\end{align*}
First, we fix an $\eta>0$. Let us fix an $m\in\mathbb{N}$ such that $T_{m}%
^{1}\left(  a,b\right)  \left\Vert a\nabla\tilde{P}\right\Vert _{\dot{B}%
_{p,1}^{\frac{2}{p}-1}}$ can be "absorbed" by the LHS of $\left(
\text{\ref{presiune_2d}}\right)  $ and that
\[
\left\Vert (Id-\dot{S}_{m})(a\nabla b)\right\Vert _{\dot{B}_{p,1}^{\frac{2}%
{p}-1}}+\left\Vert (Id-\dot{S}_{m})\left(  ab-\bar{a}\bar{b}\right)
\right\Vert _{\dot{B}_{p,1}^{\frac{2}{p}}}\leq\eta/2.
\]
Next, we see that by choosing $M$ large enough we have%
\[
T_{m,M}^{4}\left(  a,b\right)  \leq\eta.
\]
Thus, using interpolation we can write that:%
\begin{equation}
\left\Vert \nabla\tilde{P}\right\Vert _{\dot{B}_{p,2}^{\frac{2}{p}-1}%
}+\left\Vert a\nabla\tilde{P}\right\Vert _{\dot{B}_{p,1}^{\frac{2}{p}-1}}\leq
C_{ab}\left(  \left\Vert (\nabla^{2}u_{L},\nabla P_{L})\right\Vert _{\dot
{B}_{2,1}^{\frac{2}{p}-1}}+\left\Vert \tilde{u}\right\Vert _{\dot{B}%
_{p,1}^{\frac{2}{p}-1}}\right)  +2\eta\left\Vert \nabla^{2}\tilde
{u}\right\Vert _{\dot{B}_{p,1}^{\frac{2}{p}-1}}. \label{Presiune2d}%
\end{equation}

\subsubsection{End of the proof of Proposition \ref{Propozitia1Stokes}}

Obviously, combining the two estimates $\left(  \text{\ref{Presiune_3d_1}%
}\right)  $-$\left(  \text{\ref{Presiune_3d_4}}\right)  $ and $\left(
\text{\ref{Presiune2d}}\right)  $ we can continue in a unified manner the rest
of the proof of Proposition \ref{Propozitia1Stokes}. First, choose
$m\in\mathbb{N}$ large enough such that%
\[
\bar{a}\bar{b}+\dot{S}_{m}(ab-\bar{a}\bar{b})\geq\frac{a_{\star}b_{\star}}%
{2}.
\]
We apply $\dot{\Delta}_{j}$ to $\left(  \text{\ref{StokesModif}}\right)  $ and
we write that:%
\begin{align*}
\partial_{t}\tilde{u}_{j}-\operatorname{div}\left(  (\bar{a}\bar{b}+\dot
{S}_{m}(ab-\bar{a}\bar{b}))\nabla\tilde{u}_{j}\right)   &  =\tilde{f}_{j}%
-\dot{\Delta}_{j}\left(  a\nabla\tilde{P}\right) \\
&  +\dot{\Delta}_{j}\operatorname{div}\left(  (Id-\dot{S}_{m})(ab-\bar{a}%
\bar{b})\nabla\tilde{u}\right)  +\operatorname{div}\left[  \dot{\Delta}%
_{j},\dot{S}_{m}(ab-\bar{a}\bar{b})\right]  \nabla\tilde{u}\\
&  +\dot{\Delta}_{j}\left(  D\tilde{u}\dot{S}_{m}(b\nabla a)\right)
+\dot{\Delta}_{j}\left(  D\tilde{u}(Id-\dot{S}_{m})(b\nabla a)\right) \\
&  +\dot{\Delta}_{j}\left(  \nabla\tilde{u}\dot{S}_{m}(a\nabla b)\right)
+\dot{\Delta}_{j}\left(  \nabla\tilde{u}(Id-\dot{S}_{m})(a\nabla b)\right)  .
\end{align*}
Multiplying the last relation by $\left\vert \tilde{u}_{j}\right\vert
^{p-1}\operatorname*{sgn}\tilde{u}_{j}$, integrating and using Lemma $8$ from
the Appendix B of \cite{Dan10}, we get that:%
\begin{align*}
&  \left\Vert \tilde{u}_{j}\right\Vert _{L^{p}}+a_{\star}b_{\star}2^{2j}%
C\int_{0}^{t}\left\Vert \tilde{u}_{j}\right\Vert _{L^{p}}\\
&  \lesssim\int_{0}^{t}\left\Vert \tilde{f}_{j}\right\Vert _{L^{p}}+\int
_{0}^{t}\left\Vert \dot{\Delta}_{j}\left(  a\nabla\tilde{P}\right)
\right\Vert _{L^{p}}\\
&  +\int_{0}^{t}\left\Vert \operatorname{div}\left[  \dot{\Delta}_{j},\dot
{S}_{m}(ab-\bar{a}\bar{b})\right]  \nabla\tilde{u}\right\Vert _{L^{p}}%
+\int_{0}^{t}\left\Vert \dot{\Delta}_{j}\operatorname{div}\left(  (Id-\dot
{S}_{m})(ab-\bar{a}\bar{b})\nabla\tilde{u}\right)  \right\Vert _{L^{p}}\\
&  +\int_{0}^{t}\left\Vert \dot{\Delta}_{j}\left(  D\tilde{u}\dot{S}%
_{m}(b\nabla a)\right)  \right\Vert _{L^{p}}+\int_{0}^{t}\left\Vert
\dot{\Delta}_{j}\left(  D\tilde{u}(Id-\dot{S}_{m})(b\nabla a)\right)
\right\Vert _{L^{p}}\\
&  +\int_{0}^{t}\left\Vert \dot{\Delta}_{j}\left(  \nabla\tilde{u}\dot{S}%
_{m}(a\nabla b)\right)  \right\Vert _{L^{p}}+\int_{0}^{t}\left\Vert
\dot{\Delta}_{j}\left(  \nabla\tilde{u}(Id-\dot{S}_{m})(a\nabla b)\right)
\right\Vert _{L^{p}}.
\end{align*}
Multiplying the last relation by $2^{j(\frac{n}{p}-1)}$, performing an
$\ell^{1}\left(  \mathbb{Z}\right)  $-summation and using Proposition
\ref{A66} to deal with $\left\Vert \operatorname{div}\left[  \dot{\Delta}%
_{j},\dot{S}_{m}(ab-\bar{a}\bar{b})\right]  \nabla\tilde{u}\right\Vert
_{\dot{B}_{p,1}^{\frac{n}{p}-1}}$\ along with $\left(
\text{\ref{Presiune_3d_1}}\right)  $-$\left(  \text{\ref{Presiune_3d_4}%
}\right)  $ and $\left(  \text{\ref{presiune_2d}}\right)  $ to deal with the
pressure, we get that:%
\begin{align}
&  \left\Vert \tilde{u}\right\Vert _{L_{t}^{\infty}(\dot{B}_{p,1}^{\frac{n}%
{p}-1})}+a_{\star}b_{\star}C\left\Vert \nabla^{2}\tilde{u}\right\Vert
_{L_{t}^{1}(\dot{B}_{p,1}^{\frac{n}{p}-1})}\nonumber\\
&  \lesssim\left\Vert \tilde{f}\right\Vert _{L_{t}^{1}(\dot{B}_{p,1}^{\frac
{n}{p}-1})}+C\int_{0}^{t}\left\Vert a\nabla\tilde{P}\right\Vert _{\dot
{B}_{p,1}^{\frac{n}{p}-1}}\nonumber\\
&  +\int_{0}^{t}\left\Vert \left(  \dot{S}_{m}(b\nabla a),\dot{S}_{m}(a\nabla
b)\right)  \right\Vert _{\dot{B}_{p,1}^{\frac{n}{p}}}\left\Vert \nabla
\tilde{u}\right\Vert _{\dot{B}_{p,1}^{\frac{n}{p}-1}}+T_{m}\left(  a,b\right)
\left\Vert \nabla^{2}\tilde{u}\right\Vert _{L_{t}^{1}(\dot{B}_{p,1}^{\frac
{n}{p}-1})}\nonumber\\
&  \leq C_{ab}\left(  1+t\right)  \left(  \left\Vert u_{0}\right\Vert
_{\dot{B}_{p,1}^{\frac{n}{p}-1}}+\left\Vert \left(  f,\partial_{t}%
R,\nabla\operatorname{div}R\right)  \right\Vert _{L_{t}^{1}(\dot{B}%
_{p,2}^{\frac{n}{p}-\frac{n}{2}}\cap\dot{B}_{p,1}^{\frac{n}{p}-1})}\right)
\label{viteza1}\\
&  \text{ \ \ \ }+C_{ab}\int_{0}^{t}\left\Vert \tilde{u}\right\Vert _{\dot
{B}_{p,1}^{\frac{n}{p}-1}}+\left(  T_{m}\left(  a,b\right)  +\eta\right)
\left\Vert \nabla^{2}u\right\Vert _{L_{t}^{1}(\dot{B}_{p,1}^{\frac{n}{p}-1})}.
\label{viteza1.2}%
\end{align}
where%
\begin{align*}
T_{m}\left(  a,b\right)   &  =\left\Vert (Id-\dot{S}_{m})(b\nabla
a)\right\Vert _{\dot{B}_{p,1}^{\frac{n}{p}-1}}+\left\Vert (Id-\dot{S}%
_{m})(a\nabla b)\right\Vert _{\dot{B}_{p,1}^{\frac{n}{p}-1}}\\
&  +\left\Vert (Id-\dot{S}_{m})(ab-\bar{a}\bar{b})\right\Vert _{\dot{B}%
_{p,1}^{\frac{n}{p}-1}}.
\end{align*}
Assuming that $m$ is large enough respectively $\eta$ is small enough, we can
"absorb" $\left(  T_{m}\left(  a,b\right)  +\eta\right)  \left\Vert \nabla
^{2}u\right\Vert _{L_{t}^{1}(\dot{B}_{p,1}^{\frac{n}{p}-1})}$ in the LHS of
$\left(  \text{\ref{viteza1}}\right)  $. Thus, we end up with
\begin{align*}
\left\Vert \tilde{u}\right\Vert _{L_{t}^{\infty}(\dot{B}_{p,1}^{\frac{n}{p}%
-1})}+a_{\star}b_{\star}\frac{C}{2}\left\Vert \nabla^{2}\tilde{u}\right\Vert
_{L_{t}^{1}(\dot{B}_{p,1}^{\frac{n}{p}-1})}  &  \leq C_{ab}\left(  1+t\right)
\left(  \left\Vert u_{0}\right\Vert _{\dot{B}_{p,1}^{\frac{n}{p}-1}%
}+\left\Vert \left(  f,\partial_{t}R,\nabla\operatorname{div}R\right)
\right\Vert _{L_{t}^{1}(\dot{B}_{p,2}^{\frac{n}{p}-\frac{n}{2}}\cap\dot
{B}_{p,1}^{\frac{n}{p}-1})}\right) \\
&  +C_{ab}\int_{0}^{t}\left\Vert \tilde{u}\right\Vert _{\dot{B}_{p,1}%
^{\frac{n}{p}-1}}%
\end{align*}
such that using Gronwall's lemma, $\left(  \text{\ref{estimare_uL}}\right)  $
and the classical inequality%
\[
1+t^{\alpha}\leq C_{\alpha}\exp\left(  t\right)
\]
yields:%
\begin{equation}
\left\Vert \tilde{u}\right\Vert _{L_{t}^{\infty}(\dot{B}_{p,1}^{\frac{n}{p}%
-1})}+a_{\star}b_{\star}\frac{C}{2}\left\Vert \nabla^{2}\tilde{u}\right\Vert
_{L_{t}^{1}(\dot{B}_{p,1}^{\frac{n}{p}-1})}\leq C_{ab}\left(  \left\Vert
u_{0}\right\Vert _{\dot{B}_{p,1}^{\frac{n}{p}-1}}+\left\Vert \left(
f,\partial_{t}R,\nabla\operatorname{div}R\right)  \right\Vert _{L_{t}^{1}%
(\dot{B}_{p,2}^{\frac{n}{p}-\frac{n}{2}}\cap\dot{B}_{p,1}^{\frac{n}{p}-1}%
)}\right)  \exp\left(  C_{ab}t\right)  . \label{viteza_tilde}%
\end{equation}
Using the fact that $u=u_{L}+\tilde{u}$ along with $\left(
\text{\ref{estimare_uL}}\right)  $\ and $\left(  \text{\ref{viteza_tilde}%
}\right)  $ gives us:%
\begin{equation}
\left\Vert u\right\Vert _{L_{t}^{\infty}(\dot{B}_{p,1}^{\frac{n}{p}-1}%
)}+a_{\star}b_{\star}\frac{C}{2}\left\Vert \nabla^{2}u\right\Vert _{L_{t}%
^{1}(\dot{B}_{p,1}^{\frac{n}{p}-1})}\leq C_{ab}\left(  \left\Vert
u_{0}\right\Vert _{\dot{B}_{p,1}^{\frac{n}{p}-1}}+\left\Vert \left(
f,\partial_{t}R,\nabla\operatorname{div}R\right)  \right\Vert _{L_{t}^{1}%
(\dot{B}_{p,2}^{\frac{n}{p}-\frac{n}{2}}\cap\dot{B}_{p,1}^{\frac{n}{p}-1}%
)}\right)  \exp\left(  C_{ab}t\right)  . \label{viteza2}%
\end{equation}
Next, using $\left(  \text{\ref{Presiune_3d_1}}\right)  $-$\left(
\text{\ref{Presiune_3d_4}}\right)  $\ and $\left(  \text{\ref{presiune_2d}%
}\right)  $ combined with $\left(  \text{\ref{estimare_uL}}\right)  $ and
interpolation, we infer that:%
\begin{align}
&  \left\Vert \nabla P\right\Vert _{L_{t}^{1}((\dot{B}_{p,2}^{\frac{n}%
{p}-\frac{n}{2}}\cap\dot{B}_{p,1}^{\frac{n}{p}-1})}\leq C_{a}\left\Vert
a\nabla P_{L}\right\Vert _{L_{t}^{1}(\dot{B}_{p,2}^{\frac{n}{p}-\frac{n}{2}%
}\cap\dot{B}_{p,1}^{\frac{n}{p}-1})}+C_{a}\left\Vert a\nabla\tilde
{P}\right\Vert _{L_{t}^{1}(\dot{B}_{p,2}^{\frac{n}{p}-\frac{n}{2}}\cap\dot
{B}_{p,1}^{\frac{n}{p}-1})}\\
&  \leq C_{ab}\left(  \left\Vert u_{0}\right\Vert _{\dot{B}_{p,1}^{\frac{n}%
{p}-1}}+\left\Vert \left(  f,\partial_{t}R,\nabla\operatorname{div}R\right)
\right\Vert _{L_{t}^{1}(\dot{B}_{p,2}^{\frac{n}{p}-\frac{n}{2}}\cap\dot
{B}_{p,1}^{\frac{n}{p}-1})}\right)  \exp\left(  C_{ab}t\right)  .
\label{estimare_viteza3}%
\end{align}
Combing $\left(  \text{\ref{estimare_viteza3}}\right)  $ with $\left(
\text{\ref{viteza2}}\right)  $ we finally get that:%

\begin{gather}
\left\Vert u\right\Vert _{L_{t}^{\infty}(\dot{B}_{p,1}^{\frac{n}{p}-1}%
)}+\left\Vert \nabla^{2}u\right\Vert _{L_{t}^{1}(\dot{B}_{p,1}^{\frac{n}{p}%
-1})}+\left\Vert \nabla P\right\Vert _{L_{t}^{1}((\dot{B}_{p,2}^{\frac{n}%
{p}-\frac{n}{2}}\cap\dot{B}_{p,1}^{\frac{n}{p}-1})}\nonumber\\
\leq\left(  \left\Vert u_{0}\right\Vert _{\dot{B}_{p,1}^{\frac{n}{p}-1}%
}+\left\Vert \left(  f,\partial_{t}R,\nabla\operatorname{div}R\right)
\right\Vert _{L_{t}^{1}(\dot{B}_{p,2}^{\frac{n}{p}-\frac{n}{2}}\cap\dot
{B}_{p,1}^{\frac{n}{p}-1})}\right)  \exp\left(  C_{ab}\left(  t+1\right)
\right)  . \label{estimare_viteza1}%
\end{gather}
Obviously, by obtaining the last estimate we conclude the proof of Proposition
\ref{Propozitia1Stokes}.

Next, let us deal with the existence part of the Stokes problem with the
coefficients having regularity as in Proposition \ref{Propozitia1Stokes}. More
precisely, we have:

\begin{proposition}
\label{Propozitia2Stokes}Let us consider $\left(  a,b,u_{0}%
,f,R\right)  $ as in the statement of Proposition \ref{Propozitia1Stokes}.
Then, there exists a unique solution $\left(  u,\nabla P\right)  \in E_{loc}$
of the Stokes system $\left(  \text{\ref{Stokes2}}\right)  $. Furthermore,
there exists a constant $C_{ab}$ depending on $a$ and $b$ such that:%
\begin{gather}
\left\Vert u\right\Vert _{L_{t}^{\infty}(\dot{B}_{p,1}^{\frac{n}{p}-1}%
)}+\left\Vert \nabla^{2}u\right\Vert _{L_{t}^{1}(\dot{B}_{p,1}^{\frac{n}{p}%
-1})}+\left\Vert \nabla P\right\Vert _{L_{t}^{1}((\dot{B}_{p,2}^{\frac{n}%
{p}-\frac{n}{2}}\cap\dot{B}_{p,1}^{\frac{n}{p}-1})}\nonumber\\
\leq\left(  \left\Vert u_{0}\right\Vert _{\dot{B}_{p,1}^{\frac{n}{p}-1}%
}+\left\Vert \left(  f,\partial_{t}R,\nabla\operatorname{div}R\right)
\right\Vert _{L_{t}^{1}(\dot{B}_{p,2}^{\frac{n}{p}-\frac{n}{2}}\cap\dot
{B}_{p,1}^{\frac{n}{p}-1})}\right)  \exp\left(  C_{ab}\left(  t+1\right)
\right)  .
\end{gather}
for all $t>0$.
\end{proposition}

The uniqueness property is a direct consequence of the estimates of
Proposition \ref{Propozitia1Stokes}. The proof of existence relies on
Proposition \ref{Propozitia1Stokes} combined with a continuity argument as
used in \cite{Dan2}, see also \cite{Krylov}. Let us introduce
\[
\left(  a_{\theta},b_{\theta}\right)  =\left(  1-\theta\right)  \left(
\bar{a},\bar{b}\right)  +\theta\left(  a,b\right)
\]
and let us consider the following Stokes systems%
\begin{equation}
\left\{
\begin{array}
[c]{r}%
\partial_{t}u-a_{\theta}(\operatorname{div}\left(  b_{\theta}D(u)\right)
-\nabla P)=f,\\
\operatorname{div}u=\operatorname{div}R,\\
u_{|t=0}=u_{0}.
\end{array}
\right.  \tag{$\mathcal{S}_\theta$}\label{Stokes_theta}%
\end{equation}
First of all, a more detailed analysis of the estimates established in
Proposition \ref{Propozitia1Stokes} enables us to conclude that the constant
$C_{a_{\theta}b_{\theta}}$ appearing in $\left(  \text{\ref{estimare_viteza1}%
}\right)  $ is uniformly bounded with respect to $\theta\in\left[  0,1\right]
$ by a constant $c=c_{ab}$. Indeed, when repeating the estimation process
carried out in Proposition \ref{Propozitia1Stokes} with $\left(  a_{\theta
},b_{\theta}\right)  $ instead of $\left(  a,b\right)  $ amounts in replacing
$\left(  a-\bar{a}\right)  $ and $\left(  b-\bar{b}\right)  $ with
$\theta\left(  a-\bar{a}\right)  $ and $\theta\left(  b-\bar{b}\right)  $.
Taking in account Proposition \ref{compunere} and the remark that follows we
get that there exists
\[
c:=\sup_{\theta\in\left[  0,1\right]  }C_{a_{\theta}b_{\theta}}<\infty\text{.}%
\]

Let us take $T>0$ and let us consider $\mathcal{E}_{T}$ the set of those
$\theta\in\left[  0,1\right]  $ such that for any $\left(  u_{0},f,R\right)  $
as in the statement of Proposition \ref{Propozitia1Stokes} Problem $\left(
\text{\ref{Stokes_theta}}\right)  $ admits a unique solution $\left(  u,\nabla
P\right)  \in E_{T}$ which satisfies
\begin{gather}
\left\Vert u\right\Vert _{L_{t}^{\infty}(\dot{B}_{p,1}^{\frac{n}{p}-1}%
)}+\left\Vert \nabla^{2}u\right\Vert _{L_{t}^{1}(\dot{B}_{p,1}^{\frac{n}{p}%
-1})}+\left\Vert \nabla P\right\Vert _{L_{t}^{1}((\dot{B}_{p,2}^{\frac{n}%
{p}-\frac{n}{2}}\cap\dot{B}_{p,1}^{\frac{n}{p}-1})}\nonumber\\
\leq\left(  \left\Vert u_{0}\right\Vert _{\dot{B}_{p,1}^{\frac{n}{p}-1}%
}+\left\Vert \left(  f,\partial_{t}R,\nabla\operatorname{div}R\right)
\right\Vert _{L_{t}^{1}(\dot{B}_{p,2}^{\frac{n}{p}-\frac{n}{2}}\cap\dot
{B}_{p,1}^{\frac{n}{p}-1})}\right)  \exp\left(  c\left(  t+1\right)  \right)
, \label{estimare_solutie}%
\end{gather}
for all $t\in\left[  0,T\right]  $. According to Proposition
\ref{Stokes_constant}, $0\in\mathcal{E}_{T}$.

Let us suppose that $\theta\in\mathcal{E}_{T}$. First of all, we denote by
$\left(  u_{\theta},\nabla P_{\theta}\right)  \in E_{T}$ the unique solution
of $\left(  \text{\ref{Stokes_theta}}\right)  $. We consider the space%
\[
E_{T,\operatorname{div}}=\{\left(  \tilde{w},\nabla\tilde{Q}\right)  \in
E_{T}:\operatorname{div}\tilde{w}=0\}
\]
and let $S_{\theta\theta^{\prime}}$ be the operator which associates to
$\left(  \tilde{w},\nabla\tilde{Q}\right)  \in E_{T,\operatorname{div}}$,
$\left(  \tilde{u},\nabla\tilde{P}\right)  $ the unique solution of
\begin{equation}
\left\{
\begin{array}
[c]{r}%
\partial_{t}\tilde{u}-a_{\theta}(\operatorname{div}\left(  b_{\theta}%
D(\tilde{u})\right)  -\nabla\tilde{P})=g_{\theta\theta^{\prime}}\left(
u_{\theta},\nabla P_{\theta}\right)  +g_{\theta\theta^{\prime}}\left(
\tilde{w},\nabla\tilde{Q}\right)  ,\\
\operatorname{div}u=0,\\
u_{|t=0}=0.
\end{array}
\right.
\end{equation}
where
\begin{equation}
g_{\theta\theta^{\prime}}\left(  u,\nabla P\right)  =\left(  a_{\theta
}-a_{\theta^{\prime}}\right)  \nabla P+a_{\theta^{\prime}}\operatorname{div}%
\left(  b_{\theta^{\prime}}D\left(  u\right)  \right)  -a_{\theta
}\operatorname{div}\left(  b_{\theta}D\left(  u\right)  \right)  .
\label{operator}%
\end{equation}
Obviously, $S_{\theta\theta^{\prime}}$ maps $E_{T,\operatorname{div}}$ into
$E_{T,\operatorname{div}}$. We claim that there exists a positive quantity
$\varepsilon=\varepsilon\left(  T\right)  >0$ such that if $\left\vert
\theta-\theta^{\prime}\right\vert \leq\varepsilon\left(  T\right)  $ then
$S_{\theta\theta^{\prime}}$ has a fixed point $\left(  \tilde{u}^{\star
},\nabla\tilde{P}^{\star}\right)  $ in a suitable ball centered at the origin
of the space $E_{T,\operatorname{div}}$. Obviously,
\[
\left(  \tilde{u}^{\star}+u_{\theta},\nabla\tilde{P}^{\star}+\nabla P_{\theta
}\right)
\]
will solve $\left(  \mathcal{S}_{\theta^{\prime}}\right)  $ in $E_{T}$.

First, we note that, as a consequence of Proposition \ref{Propozitia1Stokes},
we have that:%
\begin{equation}
\left\Vert \left(  \tilde{u},\nabla\tilde{P}\right)  \right\Vert _{E_{T}}%
\leq\left(  \left\Vert g_{\theta\theta^{\prime}}\left(  u_{\theta},\nabla
P_{\theta}\right)  \right\Vert _{L_{T}^{1}(\dot{B}_{p,2}^{\frac{n}{p}-\frac
{n}{2}}\cap\dot{B}_{p,1}^{\frac{n}{p}-1})}+\left\Vert g_{\theta\theta^{\prime
}}\left(  \tilde{w},\nabla\tilde{Q}\right)  \right\Vert _{L_{T}^{1}(\dot
{B}_{p,2}^{\frac{n}{p}-\frac{n}{2}}\cap\dot{B}_{p,1}^{\frac{n}{p}-1})}\right)
\exp\left(  c(T+1)\right)  . \label{estimare_pct_fix}%
\end{equation}
Let us observe that%
\begin{equation}
\left\Vert \left(  a_{\theta}-a_{\theta^{\prime}}\right)  \nabla P\right\Vert
_{L_{T}^{1}(\dot{B}_{p,2}^{\frac{n}{p}-\frac{n}{2}}\cap\dot{B}_{p,1}^{\frac
{n}{p}-1})}\leq\left\vert \theta-\theta^{\prime}\right\vert \left\Vert
a-\bar{a}\right\Vert _{\dot{B}_{p,1}^{\frac{n}{p}}}\left\Vert \nabla
P\right\Vert _{L_{T}^{1}(\dot{B}_{p,2}^{\frac{n}{p}-\frac{n}{2}}\cap\dot
{B}_{p,1}^{\frac{n}{p}-1})}. \label{presiune}%
\end{equation}
Next, we write that:%
\[
a_{\theta^{\prime}}\operatorname{div}\left(  b_{\theta^{\prime}}D\left(
u\right)  \right)  -a_{\theta}\operatorname{div}\left(  b_{\theta}D\left(
u\right)  \right)  =\left(  a_{\theta^{\prime}}-a_{\theta}\right)
\operatorname{div}\left(  b_{\theta^{\prime}}D\left(  u\right)  \right)
+a_{\theta}\operatorname{div}\left(  \left(  b_{\theta^{\prime}}-b_{\theta
}\right)  D\left(  u\right)  \right)  .
\]
The first term of the last identity is estimated as follows:%
\[
\left\Vert \left(  a_{\theta^{\prime}}-a_{\theta}\right)  \operatorname{div}%
\left(  b_{\theta^{\prime}}D\left(  u\right)  \right)  \right\Vert _{L_{T}%
^{1}(\dot{B}_{p,1}^{\frac{n}{p}-1})}\leq\left\vert \theta-\theta^{\prime
}\right\vert \left\Vert a-\bar{a}\right\Vert _{\dot{B}_{p,1}^{\frac{n}{p}}%
}\left(  \bar{b}+\left\Vert b-\bar{b}\right\Vert _{\dot{B}_{p,1}^{\frac{n}{p}%
}}\right)  \left\Vert D(u)\right\Vert _{L_{T}^{1}(\dot{B}_{p,1}^{\frac{n}{p}%
})}.
\]
Regarding the second term, we have that:%
\[
\left\Vert a_{\theta}\operatorname{div}\left(  \left(  b_{\theta^{\prime}%
}-b_{\theta}\right)  D\left(  u\right)  \right)  \right\Vert _{L_{T}^{1}%
(\dot{B}_{p,1}^{\frac{n}{p}-1})}\leq\left\vert \theta-\theta^{\prime
}\right\vert \left\Vert b-\bar{b}\right\Vert _{\dot{B}_{p,1}^{\frac{n}{p}}%
}\left(  \bar{a}+\left\Vert a-\bar{a}\right\Vert _{\dot{B}_{p,1}^{\frac{n}{p}%
}}\right)  \left\Vert D(u)\right\Vert _{L_{T}^{1}(\dot{B}_{p,1}^{\frac{n}{p}%
})}%
\]
and thus:%
\begin{align}
&  \left\Vert a_{\theta^{\prime}}\operatorname{div}\left(  b_{\theta^{\prime}%
}D\left(  u\right)  \right)  -a_{\theta}\operatorname{div}\left(  b_{\theta
}D\left(  u\right)  \right)  \right\Vert _{L_{T}^{1}(\dot{B}_{p,1}^{\frac
{n}{p}-1})}\nonumber\\
&  \leq\left\vert \theta-\theta^{\prime}\right\vert \left(  \bar{a}+\left\Vert
a-\bar{a}\right\Vert _{\dot{B}_{p,1}^{\frac{n}{p}}}\right)  \left(  \bar
{b}+\left\Vert b-\bar{b}\right\Vert _{\dot{B}_{p,1}^{\frac{n}{p}}}\right)
\left\Vert Du\right\Vert _{L_{T}^{1}(\dot{B}_{p,1}^{\frac{n}{p}})}.
\label{viteza2.1}%
\end{align}
The only thing left is to treat the $L_{t}^{1}(\dot{B}_{p,2}^{\frac{3}%
{p}-\frac{3}{2}})$-norm of $a_{\theta^{\prime}}\operatorname{div}\left(
b_{\theta^{\prime}}D\left(  u\right)  \right)  -a_{\theta}\operatorname{div}%
\left(  b_{\theta}D\left(  u\right)  \right)  $ in the case where $n=3$. Using
the fact that $\nabla u\in L_{T}^{4}(\dot{B}_{2,1}^{\frac{3}{p}-\frac{1}{2}})$
write that:%
\begin{align}
\left\Vert \left(  a_{\theta^{\prime}}-a_{\theta}\right)  \operatorname{div}%
\left(  b_{\theta^{\prime}}D\left(  u\right)  \right)  \right\Vert _{L_{T}%
^{1}(\dot{B}_{p,2}^{\frac{3}{p}-\frac{3}{2}})}  &  \leq\left\vert
\theta-\theta^{\prime}\right\vert \left\Vert a-\bar{a}\right\Vert _{\dot
{B}_{p,1}^{\frac{3}{p}}}\left\Vert \operatorname{div}\left(  b_{\theta
^{\prime}}D\left(  u\right)  \right)  \right\Vert _{L_{T}^{1}(\dot{B}%
_{p,1}^{\frac{3}{p}-\frac{3}{2}})}\nonumber\\
&  \leq\left\vert \theta-\theta^{\prime}\right\vert \left\Vert a-\bar
{a}\right\Vert _{\dot{B}_{p,1}^{\frac{3}{p}}}\left(  \bar{b}+\left\Vert
b-\bar{b}\right\Vert _{\dot{B}_{p,1}^{\frac{3}{p}}}\right)  \left\Vert
Du\right\Vert _{L_{T}^{1}(\dot{B}_{p,1}^{\frac{3}{p}-\frac{1}{2}})}\\
&  \leq\left\vert \theta-\theta^{\prime}\right\vert \left\Vert a-\bar
{a}\right\Vert _{\dot{B}_{p,1}^{\frac{3}{p}}}\left(  \bar{b}+\left\Vert
b-\bar{b}\right\Vert _{\dot{B}_{p,1}^{\frac{3}{p}}}\right)  T^{\frac{3}{4}%
}\left\Vert u\right\Vert _{L_{T}^{\infty}(\dot{B}_{p,1}^{\frac{3}{p}-1}%
)}^{\frac{1}{4}}\left\Vert u\right\Vert _{L_{T}^{1}(\dot{B}_{p,1}^{\frac{3}%
{p}+1})}^{\frac{3}{4}}\label{viteza2.20}\\
&  \leq\left\vert \theta-\theta^{\prime}\right\vert C\left(  T,a,b\right)
\left(  \left\Vert u\right\Vert _{L_{T}^{\infty}(\dot{B}_{p,1}^{\frac{3}{p}%
-1})}+\left\Vert \nabla^{2}u\right\Vert _{L_{T}^{1}(\dot{B}_{p,1}^{\frac{3}%
{p}+1})}\right)  \label{viteza2.2}%
\end{align}
and, proceeding in a similar manner we can estimate $\left\Vert a_{\theta
}\operatorname{div}\left(  \left(  b_{\theta^{\prime}}-b_{\theta}\right)
D\left(  u\right)  \right)  \right\Vert _{L_{T}^{1}(\dot{B}_{p,2}^{\frac{3}%
{p}-\frac{3}{2}})}$.

Combining $\left(  \text{\ref{presiune}}\right)  $, $\left(
\text{\ref{viteza2.1}}\right)  $ along with $\left(  \text{\ref{viteza2.2}%
}\right)  $ we get that:%
\begin{equation}
\left\Vert g_{\theta\theta^{\prime}}\left(  u,\nabla P\right)  \right\Vert
_{L_{T}^{1}(\dot{B}_{p,2}^{\frac{3}{p}-\frac{3}{2}}\cap\dot{B}_{p,1}^{\frac
{3}{p}-1})}\leq\left\vert \theta-\theta^{\prime}\right\vert C\left(
T,a,b\right)  \left(  \left\Vert u\right\Vert _{L_{T}^{\infty}(\dot{B}%
_{p,1}^{\frac{3}{p}-1})}+\left\Vert \nabla^{2}u\right\Vert _{L_{T}^{1}(\dot
{B}_{p,1}^{\frac{3}{p}+1})}+\left\Vert \nabla P\right\Vert _{L_{T}^{1}(\dot
{B}_{p,2}^{\frac{3}{p}-\frac{3}{2}}\cap\dot{B}_{p,1}^{\frac{3}{p}-1})}\right)
. \label{estimare_g}%
\end{equation}
Let us replace this into $\left(  \text{\ref{estimare_pct_fix}}\right)  $ to
get that
\[
\left\Vert \left(  \tilde{u},\nabla\tilde{P}\right)  \right\Vert _{E_{T}}%
\leq\left\vert \theta-\theta^{\prime}\right\vert C\left(  T,a,b\right)
\left(  \left\Vert \left(  u_{\theta},\nabla P_{\theta}\right)  \right\Vert
_{E_{T}}+\left\Vert \left(  \tilde{w},\nabla\tilde{Q}\right)  \right\Vert
_{E_{T}}\right)
\]
and by linearity%
\[
\left\Vert \left(  \tilde{u}^{1}-\tilde{u}^{2},\nabla\tilde{P}^{1}%
-\nabla\tilde{P}^{2}\right)  \right\Vert _{E_{T}}\leq\left\vert \theta
-\theta^{\prime}\right\vert C\left(  T,a,b\right)  \left\Vert \left(
\tilde{w}^{1}-\tilde{w}^{2},\nabla\tilde{Q}^{1}-\nabla\tilde{Q}^{2}\right)
\right\Vert _{E_{T}}%
\]
where for $k=1,2$:
\[
\left(  \tilde{u}^{i},\nabla\tilde{P}^{i}\right)  =S_{\theta\theta^{\prime}%
}\left(  \left(  \tilde{w}^{i},\nabla\tilde{Q}^{i}\right)  \right)
\]
Thus one can choose $\varepsilon\left(  T\right)  $ small enough such that
$\left\vert \theta-\theta^{\prime}\right\vert \leq\varepsilon\left(  T\right)
$ gives us a fixed point of the solution operator $S_{\theta\theta^{\prime}}$
in $B_{E_{T,\operatorname{div}}}\left(  0,2\left\Vert \left(  u_{\theta
},\nabla P_{\theta}\right)  \right\Vert _{E_{T}}\right)  $.

Thus, for all $T>0$, $E_{T}=\left[  0,1\right]  $ and owing to the uniqueness
property and to Proposition \ref{Propozitia1Stokes}, we can construct a unique
solution $\left(  u,\nabla P\right)  \in E_{loc}$ to $\left(
\text{\ref{Stokes2}}\right)  $ such that for all $t>0$ the estimate $\left(
\text{\ref{Estimare_Prop1}}\right)  $ is valid. This ends the proof of
Proposition \ref{Propozitia2Stokes}.

\subsection{The proof of Theorem \ref{Teorema_Stokes} in the case $n=3$}

As it was discussed earlier, in dimension $n=3$, Proposition
\ref{Propozitia1Stokes} is weaker than Theorem \ref{Teorema_Stokes} as one
requires additional low frequency informations on the data $\left(
f,\partial_{t}R,\nabla\operatorname{div}R\right)  \in L_{t}^{1}(\dot{B}%
_{p,2}^{\frac{3}{p}-\frac{3}{2}})$. Thus, we have to bring an extra argument
in order to conclude the validity of Theorem \ref{Teorema_Stokes}. This is the
object of interest of this section.

\subsubsection{The existence part}

We begin by taking $m\in\mathbb{N}$ large enough and owing to Proposition
\ref{Stokes_perturb} we can consider $\left(  u^{1},\nabla P^{1}\right)  $ the
unique solution with $u^{1}\in\mathcal{C(}\mathbb{R}^{+};\dot{B}_{p,1}%
^{\frac{3}{p}-1})$ and $\left(  \partial_{t}u^{1},\nabla^{2}u^{1},\nabla
P^{1}\right)  \in L_{loc}^{1}(\dot{B}_{p,1}^{\frac{3}{p}-1})$ of the system
\[
\left\{
\begin{array}
[c]{r}%
\partial_{t}u-\bar{a}\bar{b}\operatorname{div}D(u)+\left(  \bar{a}+\dot
{S}_{-m}\left(  a-\bar{a}\right)  \right)  \nabla P=f,\\
\operatorname{div}u=\operatorname{div}R,\\
u_{|t=0}=u_{0},
\end{array}
\right.
\]
which also satisfies:%
\[
\left\Vert u^{1}\right\Vert _{L_{T}^{\infty}(\dot{B}_{p,1}^{\frac{3}{p}-1}%
)}+\left\Vert \left(  \partial_{t}u^{1},\bar{a}\bar{b}\nabla^{2}u^{1},\bar
{a}\nabla P^{1}\right)  \right\Vert _{L_{T}^{1}(\dot{B}_{p,1}^{\frac{3}{p}%
-1})}\leq C(\left\Vert u_{0}\right\Vert _{\dot{B}_{p,1}^{\frac{3}{p}-1}%
}+\left\Vert \left(  f,\partial_{t}R,\bar{a}\bar{b}\nabla\operatorname{div}%
R\right)  \right\Vert _{L_{T}^{1}(\dot{B}_{p,1}^{\frac{3}{p}-1})}),
\]
for all $T>0$. Let us consider
\[
G\left(  u^{1},\nabla P^{1}\right)  =a\operatorname{div}\left(  bD(u^{1}%
)\right)  -\bar{a}\operatorname{div}(\bar{b}D(u^{1}))-\left(  \left(
Id-\dot{S}_{-m}\right)  \left(  a-\bar{a}\right)  \right)  \nabla P^{1}.
\]
We claim that $G\left(  u^{1},\nabla P^{1}\right)  \in L_{loc}^{1}(\dot
{B}_{p,2}^{\frac{3}{p}-\frac{3}{2}}\cap\dot{B}_{p,1}^{\frac{3}{p}-1})$. Indeed%
\[
a\operatorname{div}\left(  bD(u^{1})\right)  -\bar{a}\operatorname{div}%
(\bar{b}D(u^{1}))=\left(  a-\bar{a}\right)  \operatorname{div}\left(
bD(u^{1})\right)  +\bar{a}\operatorname{div}\left(  \left(  b-\bar{b}\right)
D\left(  u^{1}\right)  \right)
\]
and proceeding as in $\left(  \text{\ref{viteza2.1}}\right)  $ and $\left(
\text{\ref{viteza2.20}}\right)  $ we get that%
\begin{align}
\left\Vert a\operatorname{div}\left(  bD(u^{1})\right)  -\bar{a}%
\operatorname{div}(\bar{b}D(u^{1}))\right\Vert _{L_{t}^{1}(\dot{B}%
_{p,2}^{\frac{3}{p}-\frac{3}{2}}\cap\dot{B}_{p,1}^{\frac{3}{p}-1})}  &  \leq
C_{ab}\left(  1+t^{\frac{3}{4}}\right)  \left(  \left\Vert u^{1}\right\Vert
_{L_{t}^{\infty}(\dot{B}_{p,1}^{\frac{3}{p}-1})}+\left\Vert u^{1}\right\Vert
_{L_{t}^{1}(\dot{B}_{p,1}^{\frac{3}{p}+1})}\right) \nonumber\\
&  \leq\exp\left(  C_{ab}\left(  t+1\right)  \right)  (\left\Vert
u_{0}\right\Vert _{\dot{B}_{p,1}^{\frac{3}{p}-1}}+\left\Vert \left(
f,\partial_{t}R,\nabla\operatorname{div}R\right)  \right\Vert _{L_{T}^{1}%
(\dot{B}_{p,1}^{\frac{3}{p}-1})}). \label{vitez1}%
\end{align}
Next, we obviously have%
\begin{equation}
\left\Vert \left(  \left(  Id-\dot{S}_{-m}\right)  \left(  a-\bar{a}\right)
\right)  \nabla P^{1}\right\Vert _{L_{t}^{1}(\dot{B}_{p,1}^{\frac{3}{p}-1}%
)}\leq C\left\Vert \left(  a-\bar{a}\right)  \right\Vert _{\dot{B}%
_{p,1}^{\frac{3}{p}}}\left\Vert \nabla P^{1}\right\Vert _{L_{t}^{1}(\dot
{B}_{p,1}^{\frac{3}{p}-1})}. \label{pres1}%
\end{equation}
Using the fact that the product maps $\dot{B}_{p,1}^{\frac{3}{p}-\frac{1}{2}%
}\times\dot{B}_{p,1}^{\frac{3}{p}-1}\rightarrow\dot{B}_{p,2}^{\frac{3}%
{p}-\frac{3}{2}}$ we get that:
\begin{equation}
\left\Vert \left(  \left(  Id-\dot{S}_{-m}\right)  \left(  a-\bar{a}\right)
\right)  \nabla P^{1}\right\Vert _{L_{t}^{1}(\dot{B}_{p,2}^{\frac{3}{p}%
-\frac{3}{2}})}\leq C\left\Vert \left(  Id-\dot{S}_{-m}\right)  \left(
a-\bar{a}\right)  \right\Vert _{\dot{B}_{p,1}^{\frac{3}{p}-\frac{1}{2}}%
}\left\Vert \nabla P^{1}\right\Vert _{L_{t}^{1}(\dot{B}_{p,1}^{\frac{3}{p}%
-1})}. \label{pres2}%
\end{equation}
Of course%
\begin{align*}
\left\Vert \left(  Id-\dot{S}_{-m}\right)  \left(  a-\bar{a}\right)
\right\Vert _{\dot{B}_{p,1}^{\frac{3}{p}-\frac{1}{2}}}  &  \leq C\sum
_{j\geq-m}2^{j\left(  \frac{3}{p}-\frac{1}{2}\right)  }\left\Vert \dot{\Delta
}_{j}\left(  a-\bar{a}\right)  \right\Vert _{L^{2}}\leq C2^{\frac{m}{2}}%
\sum_{j\geq-m}2^{\frac{3}{p}j}\left\Vert \dot{\Delta}_{j}\left(  a-\bar
{a}\right)  \right\Vert _{L^{2}}.\\
&  \leq C2^{\frac{m}{2}}\left\Vert a-\bar{a}\right\Vert _{\dot{B}_{p,1}%
^{\frac{3}{p}}}%
\end{align*}
so that the first term in the RHS of $\left(  \text{\ref{pres2}}\right)  $ is
finite. We thus gather from $\left(  \text{\ref{vitez1}}\right)  $, $\left(
\text{\ref{pres1}}\right)  $ and $\left(  \text{\ref{pres2}}\right)  $ that
$G\left(  u^{1},\nabla P^{1}\right)  \in L_{loc}^{1}(\dot{B}_{p,2}^{\frac
{3}{p}-\frac{3}{2}}\cap\dot{B}_{p,1}^{\frac{3}{p}-1})$ and that for all $t>0$
there exists a constant $C_{ab}$ such that
\[
\left\Vert G\left(  u^{1},\nabla P^{1}\right)  \right\Vert _{L_{t}^{1}(\dot
{B}_{p,2}^{\frac{3}{p}-\frac{3}{2}}\cap\dot{B}_{p,1}^{\frac{3}{p}-1})}%
\leq(\left\Vert u_{0}\right\Vert _{\dot{B}_{p,1}^{\frac{3}{p}-1}}+\left\Vert
\left(  f,\partial_{t}R,\nabla\operatorname{div}R\right)  \right\Vert
_{L_{t}^{1}(\dot{B}_{p,1}^{\frac{3}{p}-1})})\exp\left(  C_{ab}\left(
t+1\right)  \right)  .
\]
According to Proposition \ref{Propozitia2Stokes}, there exists a unique
solution $\left(  u^{2},\nabla P^{2}\right)  \in E_{loc}$ of the system:%
\[
\left\{
\begin{array}
[c]{r}%
\partial_{t}u-a\operatorname{div}(bD(u))+a\nabla P=G\left(  u^{1},\nabla
P^{1}\right)  ,\\
\operatorname{div}u=0,\\
u_{|t=0}=0,
\end{array}
\right.
\]
which satisfies the following estimate%
\begin{align*}
\left\Vert u^{2}\right\Vert _{L_{t}^{\infty}(\dot{B}_{p,1}^{\frac{3}{p}-1}%
)}+\left\Vert \left(  \nabla^{2}u^{2},\nabla P^{2}\right)  \right\Vert
_{L_{t}^{1}(\dot{B}_{p,1}^{\frac{3}{p}-1})}  &  \leq\left\Vert G\left(
u^{1},\nabla P^{1}\right)  \right\Vert _{L_{t}^{1}(\dot{B}_{p,2}^{\frac{3}%
{p}-\frac{3}{2}}\cap\dot{B}_{p,1}^{\frac{3}{p}-1})}\exp\left(  C_{ab}\left(
t+1\right)  \right) \\
&  \leq(\left\Vert u_{0}\right\Vert _{\dot{B}_{p,1}^{\frac{3}{p}-1}%
}+\left\Vert \left(  f,\partial_{t}R,\nabla\operatorname{div}R\right)
\right\Vert _{L_{t}^{1}(\dot{B}_{p,1}^{\frac{3}{p}-1})})\exp\left(
C_{ab}\left(  t+1\right)  \right)  .
\end{align*}
We observe that
\[
\left(  u,\nabla P\right)  :=\left(  u^{1}+u^{2},\nabla P^{1}+\nabla
P^{2}\right)
\]
is a solution of $\left(  \text{\ref{Stokes2}}\right)  $ which satisfies%
\begin{equation}
\left\Vert u\right\Vert _{L_{t}^{\infty}(\dot{B}_{p,1}^{\frac{3}{p}-1}%
)}+\left\Vert \left(  \nabla^{2}u,\nabla P\right)  \right\Vert _{L_{t}%
^{1}(\dot{B}_{p,1}^{\frac{3}{p}-1})}\leq(\left\Vert u_{0}\right\Vert _{\dot
{B}_{p,1}^{\frac{3}{p}-1}}+\left\Vert \left(  f,\partial_{t}R,\nabla
\operatorname{div}R\right)  \right\Vert _{L_{t}^{1}(\dot{B}_{p,1}^{\frac{3}%
{p}-1})})\exp\left(  C_{ab}\left(  t+1\right)  \right)  . \label{1}%
\end{equation}
\ Of course, using again the first equation of $\left(  \text{\ref{Stokes2}%
}\right)  $ we get that
\[
\left\Vert \partial_{t}u\right\Vert _{L_{t}^{1}(\dot{B}_{p,1}^{\frac{3}{p}%
-1})}\leq C_{ab}\left\Vert \left(  f,\nabla^{2}u,\nabla P\right)  \right\Vert
_{L_{t}^{1}(\dot{B}_{p,1}^{\frac{3}{p}-1})}%
\]
and thus, we get the estimate%
\begin{equation}
\left\Vert u\right\Vert _{L_{t}^{\infty}(\dot{B}_{p,1}^{\frac{3}{p}-1}%
)}+\left\Vert \left(  \partial_{t}u,\nabla^{2}u,\nabla P\right)  \right\Vert
_{L_{t}^{1}(\dot{B}_{p,1}^{\frac{3}{p}-1})}\leq(\left\Vert u_{0}\right\Vert
_{\dot{B}_{p,1}^{\frac{3}{p}-1}}+\left\Vert \left(  f,\partial_{t}%
R,\nabla\operatorname{div}R\right)  \right\Vert _{L_{t}^{1}(\dot{B}%
_{p,1}^{\frac{3}{p}-1})})\exp\left(  C_{ab}\left(  t+1\right)  \right)  .
\label{2}%
\end{equation}

\subsubsection{Uniqueness}

Next, let us prove the uniqueness property. Let us suppose that there exists a
$T>0$ and a pair $\left(  u,\nabla P\right)  $ that solves
\begin{equation}
\left\{
\begin{array}
[c]{r}%
\partial_{t}u-a\operatorname{div}(bD(u))+a\nabla P=0,\\
\operatorname{div}u=0,\\
u_{|t=0}=0,
\end{array}
\right.  \label{unicitate}%
\end{equation}
with%
\[
u\in C_{T}(\dot{B}_{p,1}^{\frac{3}{p}-1})\text{ and }\left(  \partial
_{t}u,\nabla^{2}u,\nabla P\right)  \in L_{T}^{1}(\dot{B}_{p,1}^{\frac{3}{p}%
-1}).
\]
Observe that we cannot directly conclude to the uniqueness property by
appealing to Proposition \ref{Propozitia2Stokes} because the pressure does not
belong (a priori) to $L_{T}^{1}(\dot{B}_{p,2}^{\frac{3}{p}-\frac{3}{2}})$.
Recovering this low frequency information is done in the following lines. Let
us suppose that $3<p<4$. Applying the operator $\mathcal{Q}$ in the first
equation of $\left(  \text{\ref{unicitate}}\right)  $ we write that:%
\[
\mathcal{Q}((\bar{a}+\dot{S}_{-m}(a-\bar{a}))\nabla P)=Q\mathcal{(}%
a\operatorname{div}(bD(u)))-\mathcal{Q}((Id-\dot{S}_{-m})\left(  a-\bar
{a}\right)  \nabla P)
\]
where $m\in\mathbb{N}$ will be fixed later. We observe that:
\begin{align*}
\left\Vert \mathcal{Q}(\left(  \bar{a}+\dot{S}_{-m}\left(  a-\bar{a}\right)
\right)  \nabla P)\right\Vert _{L_{T}^{1}(\dot{B}_{p,1}^{\frac{3}{p}-\frac
{3}{2}})} &  \lesssim\left\Vert Q\mathcal{(}a\operatorname{div}%
(bD(u)))\right\Vert _{L_{T}^{1}(\dot{B}_{p,1}^{\frac{3}{p}-\frac{3}{2}}%
)}+\left\Vert \mathcal{Q}((Id-\dot{S}_{-m})\left(  a-\bar{a}\right)  \nabla
P)\right\Vert _{L_{T}^{1}(\dot{B}_{p,1}^{\frac{3}{p}-\frac{3}{2}})}\\
&  \lesssim T^{\frac{1}{4}}\left(  \bar{a}+\left\Vert a-\bar{a}\right\Vert
_{\dot{B}_{p,1}^{\frac{3}{p}}}\right)  \left(  \bar{b}+\left\Vert b-\bar
{b}\right\Vert _{\dot{B}_{p,1}^{\frac{3}{p}}}\right)  \left\Vert \nabla
u\right\Vert _{L_{T}^{\frac{4}{3}}(\dot{B}_{p,1}^{\frac{3}{p}-\frac{1}{2}})}\\
&  +\left\Vert \left(  Id-\dot{S}_{-m}\right)  \left(  a-\bar{a}\right)
\right\Vert _{\dot{B}_{p,1}^{\frac{3}{p}-\frac{1}{2}}}\left\Vert \nabla
P\right\Vert _{L_{T}^{1}(\dot{B}_{p,1}^{\frac{3}{p}-1})}.
\end{align*}
Consequently, we get that
\begin{equation}
\mathcal{Q}((\bar{a}+\dot{S}_{-m}(a-\bar{a}))\nabla P)\in L_{T}^{1}(\dot
{B}_{p,1}^{\frac{3}{p}-\frac{3}{2}}).\label{compresibilP}%
\end{equation}
Let us observe that the condition $p\in\left(  3,4\right)  $ ensures that
$\dot{B}_{p,1}^{\frac{3}{p}}$ is contained in the multiplier space of $\dot
{B}_{p^{\prime},2}^{-\frac{3}{p}+1}=\dot{B}_{p^{\prime},2}^{\frac{3}%
{p^{\prime}}-2}$. More precisely, we get

\begin{proposition}
Let us consider $\left(  u,v\right)  \in\dot{B}_{p,1}^{\frac{3}{p}}\times
\dot{B}_{p^{\prime},2}^{-\frac{3}{p}+1}$. Then $uv\in\dot{B}_{p^{\prime}%
,2}^{-\frac{3}{p}+1}$ and
\[
\left\Vert uv\right\Vert _{\dot{B}_{p^{\prime},2}^{-\frac{3}{p}+1}}%
\lesssim\left\Vert u\right\Vert _{\dot{B}_{p,1}^{\frac{3}{p}}}\left\Vert
v\right\Vert _{\dot{B}_{p^{\prime},2}^{-\frac{3}{p}+1}}.
\]

\end{proposition}

\begin{proof}
Indeed, considering $\left(  u,v\right)  \in\dot{B}_{p,1}^{\frac{3}{p}}%
\times\dot{B}_{p^{\prime},2}^{-\frac{3}{p}+1}$ and using the Bony
decomposition we get that%
\[
\left\Vert \dot{T}_{u}v\right\Vert _{\dot{B}_{p^{\prime},2}^{-\frac{3}{p}+1}%
}\lesssim\left\Vert u\right\Vert _{L^{\infty}}\left\Vert v\right\Vert
_{\dot{B}_{p^{\prime},2}^{-\frac{3}{p}+1}}.
\]
Next, considering
\[
\frac{1}{p^{\prime}}=\frac{1}{2}+\frac{1}{p^{\star}}.
\]
we see that
\begin{align*}
2^{j\left(  -\frac{3}{p}+1\right)  }\left\Vert \dot{\Delta}_{j}\dot{T}%
_{v}^{\prime}u\right\Vert _{L^{p^{\prime}}}  &  \lesssim\sum_{\ell\geq
j-3}2^{\left(  -\frac{3}{p}+1\right)  \left(  j-\ell\right)  }2^{\left(
-\frac{3}{p}+1\right)  \ell}\left\Vert S_{\ell+1}v\right\Vert _{L^{2}%
}\left\Vert \dot{\Delta}_{\ell}u\right\Vert _{L^{p^{\star}}}\\
&  =\sum_{\ell\geq j-3}2^{\left(  -\frac{3}{p}+1\right)  \left(
j-\ell\right)  }2^{-\frac{1}{2}\ell}\left\Vert S_{\ell+1}v\right\Vert _{L^{2}%
}2^{\frac{3}{p^{\star}}\ell}\left\Vert \dot{\Delta}_{\ell}u\right\Vert
_{L^{p^{\star}}},
\end{align*}
and consequently, we get%
\[
\left\Vert \dot{T}_{v}^{\prime}u\right\Vert _{\dot{B}_{p^{\prime},2}%
^{-\frac{3}{p}+1}}\lesssim\left\Vert v\right\Vert _{\dot{H}^{-\frac{1}{2}}%
}\left\Vert u\right\Vert _{\dot{B}_{p^{\star},1}^{\frac{3}{p^{\star}}}%
}\lesssim\left\Vert v\right\Vert _{\dot{B}_{p^{\prime},2}^{-\frac{3}{p}+1}%
}\left\Vert u\right\Vert _{\dot{B}_{p,1}^{\frac{3}{p}}}.
\]

\end{proof}

\begin{proposition}
\label{Elippp}Let us consider $p\in\left(  3,4\right)  $. Furthermore,
consider a constant $\bar{c}>0$ and $c\in\dot{B}_{p,1}^{\frac{3}{p}}$. Then
there exists an universal constant $\eta>0$ such that if
\[
\left\Vert c\right\Vert _{\dot{B}_{p,1}^{\frac{3}{p}}}\leq\eta,
\]
then for any $\psi\in\dot{B}_{p^{\prime},2}^{\frac{3}{p^{\prime}}-\frac{3}{2}%
}\cap\dot{B}_{p^{\prime},2}^{\frac{3}{p^{\prime}}-2}$ there exists a unique
solution $\nabla P\in\dot{B}_{p^{\prime},2}^{\frac{3}{p^{\prime}}-\frac{3}{2}%
}\cap\dot{B}_{p^{\prime},2}^{\frac{3}{p^{\prime}}-2}$ of the eliptic equation%
\[
\operatorname{div}\left(  \left(  \bar{c}+c\right)  \nabla P\right)
=\operatorname{div}\psi.
\]
Moreover, the following estimate holds true%
\[
\left\Vert \nabla P\right\Vert _{\dot{B}_{p^{\prime},2}^{\frac{3}{p^{\prime}%
}-\sigma}}\lesssim\left\Vert \mathcal{Q}\psi\right\Vert _{\dot{B}_{p^{\prime
},2}^{\frac{3}{p^{\prime}}-\sigma}},
\]
where $\sigma\in\left\{  \frac{3}{2},2\right\}  $.

\begin{proof}
The proof is standard. Under some smallness condition on $c\in\dot{B}%
_{p,1}^{\frac{3}{p}}$ the operator
\[
\nabla R\rightarrow\nabla P=\frac{1}{\bar{c}}\mathcal{Q}\left(  \psi-c\nabla
R\right)
\]
has a fixed point in a suitable chosen ball of the space $\dot{B}_{p^{\prime
},2}^{\frac{3}{p^{\prime}}-\frac{3}{2}}\cap\dot{B}_{p^{\prime},2}^{\frac
{3}{p^{\prime}}-2}$.
\end{proof}
\end{proposition}

We choose $m\in\mathbb{N}$ such that $\left\Vert \dot{S}_{-m}\left(  a-\bar
{a}\right)  \right\Vert _{\dot{B}_{p,1}^{\frac{3}{p}}}$ is small enough such
that we can apply Proposition \ref{Elippp} with $\bar{a}$ and $\dot{S}%
_{-m}\left(  a-\bar{a}\right)  $ instead of $\bar{c}$ and $c$. Let us consider
$\psi$ a vector field with coefficients in $\mathcal{S}$. As the Schwartz
class is included in $\dot{B}_{p^{\prime},2}^{\frac{3}{p^{\prime}}-\frac{3}%
{2}}\cap\dot{B}_{p^{\prime},2}^{\frac{3}{p^{\prime}}-2}$, let us consider
$\nabla P_{\psi}\in$ $\dot{B}_{p^{\prime},2}^{\frac{3}{p^{\prime}}-\frac{3}%
{2}}\cap\dot{B}_{p^{\prime},2}^{\frac{3}{p^{\prime}}-2}$ the solution of the
equation
\[
\operatorname{div}\left(  \left(  \bar{a}+\dot{S}_{-m}\left(  a-\bar
{a}\right)  \right)  \nabla P_{\psi}\right)  =\operatorname{div}\psi,
\]
the existence of which is granted by Proposition \ref{Elippp}. Then, using
Proposition \ref{Caracterizare} and Proposition \ref{autoadjunct}, we write
that\footnote{We denote $\tilde{\Delta}_{j}:=\dot{\Delta}_{j-1}+\dot{\Delta
}_{j}+\dot{\Delta}_{j+1}.$}:%
\begin{align}
\left\langle \nabla P,\psi\right\rangle _{\mathcal{S}^{\prime}\times
\mathcal{S}} &  =\sum_{j}\left\langle \dot{\Delta}_{j}\nabla P,\tilde{\Delta
}_{j}\psi\right\rangle =\sum_{j}-\left\langle \dot{\Delta}_{j}P,\tilde{\Delta
}_{j}\operatorname{div}\psi\right\rangle \label{unic1.0}\\
&  =\sum_{j}-\left\langle \dot{\Delta}_{j}P,\tilde{\Delta}_{j}%
\operatorname{div}\left(  \left(  \bar{a}+\dot{S}_{-m}\left(  a-\bar
{a}\right)  \right)  \nabla P_{\psi}\right)  \right\rangle =\sum
_{j}\left\langle \dot{\Delta}_{j}\nabla P,\tilde{\Delta}_{j}\left(  \left(
\bar{a}+\dot{S}_{-m}\left(  a-\bar{a}\right)  \right)  \nabla P_{\psi}\right)
\right\rangle \label{unic1.1}\\
&  =\sum_{j}\left\langle \dot{\Delta}_{j}\left(  \bar{a}+\dot{S}_{-m}\left(
a-\bar{a}\right)  \right)  \nabla P,\tilde{\Delta}_{j}\nabla P_{\psi
}\right\rangle =\sum_{j}\left\langle \dot{\Delta}_{j}\mathcal{Q}\left(
\left(  \bar{a}+\dot{S}_{-m}\left(  a-\bar{a}\right)  \right)  \nabla
P\right)  ,\tilde{\Delta}_{j}\nabla P_{\psi}\right\rangle \label{unic1.2}\\
&  \lesssim\left\Vert \mathcal{Q}\left(  \left(  \bar{a}+\dot{S}_{-m}\left(
a-\bar{a}\right)  \right)  \nabla P\right)  \right\Vert _{\dot{B}_{p,2}%
^{\frac{3}{p}-\frac{3}{2}}}\left\Vert \nabla P_{\psi}\right\Vert _{\dot
{B}_{p^{\prime},1}^{\frac{3}{p^{\prime}}-\frac{3}{2}}}\\
&  \lesssim\left\Vert \mathcal{Q}\left(  \left(  \bar{a}+\dot{S}_{-m}\left(
a-\bar{a}\right)  \right)  \nabla P\right)  \right\Vert _{\dot{B}_{p,2}%
^{\frac{3}{p}-\frac{3}{2}}}\left\Vert \psi\right\Vert _{\dot{B}_{p^{\prime}%
,1}^{\frac{3}{p^{\prime}}-\frac{3}{2}}}.
\end{align}
Taking the supremum over all $\psi\in\mathcal{S}$ with $\left\Vert
\psi\right\Vert _{\dot{B}_{p^{\prime},2}^{\frac{3}{p^{\prime}}-\frac{3}{2}}%
}\leq1$, owing to $\left(  \text{\ref{compresibilP}}\right)  $\ and
Proposition \ref{Caracterizare},\ it follows that $\nabla P\in L_{T}^{1}%
(\dot{B}_{p,2}^{\frac{3}{p}-\frac{3}{2}})$ and that
\[
\left\Vert \nabla P\right\Vert _{L_{T}^{1}(\dot{B}_{p,2}^{\frac{3}{p}-\frac
{3}{2}})}\lesssim\left\Vert \mathcal{Q}\left(  \left(  \bar{a}+\dot{S}%
_{-m}\left(  a-\bar{a}\right)  \right)  \nabla P\right)  \right\Vert
_{L_{T}^{1}(\dot{B}_{p,2}^{\frac{3}{p}-\frac{3}{2}})}.
\]
According to the uniqueness property of Proposition \ref{Propozitia2Stokes} we
conclude that $\left(  u,\nabla P\right)  =\left(  0,0\right)  $.

Let us observe that in the case $p\in\left(  \frac{6}{5},3\right]  $, owing to
the fact that $\dot{B}_{p,1}^{\frac{3}{p}-1}\hookrightarrow$ $\dot{B}%
_{q,1}^{\frac{3}{q}-1}$ for any $q\in\left(  3,4\right)  $ and $u\in
C_{T}(\dot{B}_{p,1}^{\frac{3}{p}-1})$ along with $\left(  \partial_{t}%
u,\nabla^{2}u,\nabla P\right)  \in L_{T}^{1}(\dot{B}_{p,1}^{\frac{3}{p}-1})$
we get that $u\in C_{T}(\dot{B}_{q,1}^{\frac{3}{q}-1})$ along with $\left(
\partial_{t}u,\nabla^{2}u,\nabla P\right)  \in L_{T}^{1}(\dot{B}_{q,1}%
^{\frac{3}{q}-1})$. Thus, owing to the uniqueness property for the case
$q\in\left(  3,4\right)  $ we conclude that $\left(  u,\nabla P\right)  $ is
identically null for $p\in\left(  \frac{6}{5},3\right]  $.

\section{Proof of Theorem \ref{NS1}\label{Section3}}

In the rest of the paper we aim at proving Theorem \ref{NS1}. Thus, from now
on we will work in a $3$ dimensional framework.

\subsection{The linear theory}

Let us introduce the space $F_{T}$ of $\left(  \tilde{w},\nabla\tilde
{Q}\right)  $ with $\tilde{w}\in\mathcal{C}_{T}(\dot{B}_{p,1}^{\frac{3}{p}%
-1})$ and $\left(  \partial_{t}\tilde{w},\nabla^{2}\tilde{w},\nabla\tilde
{Q}\right)  \in L_{T}^{1}(\dot{B}_{p,1}^{\frac{3}{p}-1})$ with the norm
\[
\left\Vert \left(  \tilde{w},\nabla\tilde{Q}\right)  \right\Vert _{F_{T}%
}=\left\Vert \tilde{w}\right\Vert _{L_{T}^{\infty}(\dot{B}_{p,1}^{\frac{3}%
{p}-1})}+\left\Vert \left(  \partial_{t}\tilde{w},\nabla^{2}\tilde{w}%
,\nabla\tilde{Q}\right)  \right\Vert _{L_{T}^{1}(\dot{B}_{p,1}^{\frac{3}{p}%
-1})}.
\]
Before attacking the well-posedness of $\left(  \text{\ref{NavierStokes_modif}%
}\right)  $, we first have to solve the following linear system:
\begin{equation}
\left\{
\begin{array}
[c]{r}%
\rho_{0}\partial_{t}\bar{u}-\operatorname{div}\left(  \mu\left(  \rho
_{0}\right)  A_{\bar{v}}D_{A_{\bar{v}}}\left(  \bar{u}\right)  \right)
+A_{\bar{v}}^{T}\nabla\bar{P}=0,\\
\operatorname{div}\left(  A_{\bar{v}}\bar{u}\right)  =0,\\
\bar{u}_{|t=0}=u_{0}.
\end{array}
\right.  \label{NSL1}%
\end{equation}
where $\bar{v}\in\mathcal{C}_{T}(\dot{B}_{p,1}^{\frac{3}{p}-1})$ with
$\nabla\bar{v}\in L_{T}^{1}(\dot{B}_{p,1}^{\frac{3}{p}})\cap$ $L_{T}^{2}%
(\dot{B}_{p,1}^{\frac{3}{p}-1})$,%
\[
X_{\bar{v}}\left(  t,y\right)  =y+\int_{0}^{t}\bar{v}\left(  \tau,y\right)
d\tau,
\]
with $\det DX_{\bar{v}}=1$ and $A_{\bar{v}}=\left(  DX_{\bar{v}}\right)
^{-1}$. Moreover, we suppose that:%
\begin{equation}
\left\Vert \nabla\bar{v}\right\Vert _{L_{T}^{2}(\dot{B}_{p,1}^{\frac{3}{p}%
-1})}+\left\Vert \nabla\bar{v}\right\Vert _{L_{T}^{1}(\dot{B}_{p,1}^{\frac
{3}{p}})}\leq2\alpha\label{smallness1}%
\end{equation}
for a suitably small $\alpha$. Obviously, this will be achieved using the
estimates of the Stokes system established in the previous section, see
Theorem \ref{Teorema_Stokes}. Let us write $\left(  \text{\ref{NSL1}}\right)
$ in the form
\[
\left\{
\begin{array}
[c]{r}%
\partial_{t}\bar{u}-\frac{1}{\rho_{0}}\operatorname{div}\left(  \mu\left(
\rho_{0}\right)  D\left(  \bar{u}\right)  \right)  +\frac{1}{\rho_{0}}%
\nabla\bar{P}=\frac{1}{\rho_{0}}F_{\bar{v}}\left(  \bar{u},\nabla\bar
{P}\right)  ,\\
\operatorname{div}\bar{u}=\operatorname{div}\left(  \left(  Id-A_{\bar{v}%
}\right)  \bar{u}\right)  ,\\
\bar{u}_{t=0}=u_{0}.
\end{array}
\right.
\]
with%
\[
F_{\bar{v}}\left(  \bar{w},\nabla\bar{Q}\right)  :=\operatorname{div}\left(
\mu\left(  \rho_{0}\right)  A_{\bar{v}}D_{A_{\bar{v}}}\left(  \bar{w}\right)
-\mu\left(  \rho_{0}\right)  D\left(  \bar{w}\right)  \right)  +\left(
Id-A_{\bar{v}}^{T}\right)  \nabla\bar{Q}.
\]
Let us consider $\left(  u_{L},\nabla P_{L}\right)  $ with $u_{L}%
\in\mathcal{C}(\mathbb{R}^{+},\dot{B}_{p,1}^{\frac{3}{p}-1})$ and $\left(
\partial_{t}u_{L},\nabla^{2}u_{L},\nabla P_{L}\right)  \in L_{loc}^{1}(\dot
{B}_{p,1}^{\frac{3}{p}-1})$ the unique solution of
\begin{equation}
\left\{
\begin{array}
[c]{r}%
\partial_{t}u_{L}-\frac{1}{\rho_{0}}\operatorname{div}\left(  \mu\left(
\rho_{0}\right)  D\left(  u_{L}\right)  \right)  +\frac{1}{\rho_{0}}\nabla
P_{L}=0,\\
\operatorname{div}u_{L}=0,\\
u_{L|t=0}=u_{0},
\end{array}
\right.  \label{solution1}%
\end{equation}
for which we know that:%
\[
\left\Vert \left(  u_{L},\nabla P_{L}\right)  \right\Vert _{E_{T}}%
\leq\left\Vert u_{0}\right\Vert _{\dot{B}_{p,1}^{\frac{3}{p}-1}}\exp\left(
C_{\rho_{0}}\left(  T+1\right)  \right)  .
\]
Moreover, $T$ can be chosen small enough such that%
\begin{equation}
\left\Vert \nabla u_{L}\right\Vert _{L_{T}^{2}(\dot{B}_{p,1}^{\frac{3}{p}-1}%
)}+\left\Vert \left(  \partial_{t}u_{L},\nabla^{2}u_{L},\nabla P_{L}\right)
\right\Vert _{L_{T}^{1}(\dot{B}_{p,1}^{\frac{3}{p}-1})}\leq\alpha
\label{smallness2}%
\end{equation}
Following the idea in \cite{Dan5}, and owing to Theorem \ref{Teorema_Stokes}%
,\ we consider the operator $\Phi$ which associates to $\left(  \tilde
{w},\nabla\tilde{Q}\right)  \in F_{T}$, the unique solution $\left(  \tilde
{u},\nabla\tilde{P}\right)  \in F_{T}$ of:
\[
\left\{
\begin{array}
[c]{r}%
\partial_{t}\tilde{u}-\frac{1}{\rho_{0}}\operatorname{div}\left(  \mu\left(
\rho_{0}\right)  D\left(  \tilde{u}\right)  \right)  +\frac{1}{\rho_{0}}%
\nabla\tilde{P}=\frac{1}{\rho_{0}}F_{\bar{v}}\left(  u_{L}+\tilde{w},\nabla
P_{L}+\nabla\tilde{Q}\right)  ,\\
\operatorname{div}\tilde{u}=\operatorname{div}\left(  \left(  Id-A_{\bar{v}%
}\right)  (u_{L}+\tilde{w})\right)  ,\\
\tilde{u}_{|t=0}=0.
\end{array}
\right.
\]
We will show in the following that for a sufficiently small $T>0$, there
exists a fixed point for $\Phi$ in the unit ball centered at the origin of
$F_{T}$. More precisely, according to Theorem \ref{Teorema_Stokes} we get that%
\begin{align}
\left\Vert \Phi\left(  \tilde{w},\nabla\tilde{Q}\right)  \right\Vert _{F_{T}}
&  \leq\left\Vert \frac{1}{\rho_{0}}F_{\bar{v}}\left(  u_{L}+\tilde{w},\nabla
P_{L}+\nabla\tilde{Q}\right)  \right\Vert _{L_{T}^{1}(\dot{B}_{p,1}^{\frac
{3}{p}-1})}+\left\Vert \partial_{t}\left(  Id-A_{\bar{v}}\right)
(u_{L}+\tilde{w})\right\Vert _{L_{T}^{1}(\dot{B}_{p,1}^{\frac{3}{p}-1}%
)}\nonumber\\
&  +\left\Vert \nabla\operatorname{div}\left(  \left(  Id-A_{\bar{v}}\right)
(u_{L}+\tilde{w})\right)  \right\Vert _{L_{T}^{1}(\dot{B}_{p,1}^{\frac{3}%
{p}-1})}. \label{operator_2}%
\end{align}

We begin by treating the first term:%
\begin{equation}
\left\Vert \frac{1}{\rho_{0}}F_{\bar{v}}\left(  u_{L}+\tilde{w},\nabla
P_{L}+\nabla\tilde{Q}\right)  \right\Vert _{L_{T}^{1}(\dot{B}_{p,1}^{\frac
{3}{p}-1})}\lesssim\left(  \frac{1}{\bar{\rho}}+\left\Vert \frac{1}{\rho_{0}%
}-\frac{1}{\bar{\rho}}\right\Vert _{\dot{B}_{p,1}^{\frac{3}{p}}}\right)
\left\Vert F_{\bar{v}}\left(  u_{L}+\tilde{w},\nabla P_{L}+\nabla\tilde
{Q}\right)  \right\Vert _{L_{T}^{1}(\dot{B}_{p,1}^{\frac{3}{p}-1})}.
\label{T1.1}%
\end{equation}
We write that%
\begin{align*}
T_{1}  &  =\operatorname{div}\left(  \mu\left(  \rho_{0}\right)  A_{\bar{v}%
}D_{A_{\bar{v}}}\left(  u_{L}+\tilde{w}\right)  \right)  -\operatorname{div}%
\left(  \mu\left(  \rho_{0}\right)  D\left(  u_{L}+\tilde{w}\right)  \right)
\\
&  =\operatorname{div}\left(  \mu\left(  \rho_{0}\right)  \left(  A_{\bar{v}%
}-Id\right)  D_{A_{\bar{v}}}\left(  u_{L}+\tilde{w}\right)  \right)
+\operatorname{div}\left(  \mu\left(  \rho_{0}\right)  D_{A_{\bar{v}}%
-Id}\left(  u_{L}+\tilde{w}\right)  \right) \\
&  =\operatorname{div}\left(  \mu\left(  \rho_{0}\right)  \left(  A_{\bar{v}%
}-Id\right)  D_{A_{\bar{v}}-Id}\left(  u_{L}+\tilde{w}\right)  \right)
+\operatorname{div}\left(  \mu\left(  \rho_{0}\right)  \left(  A_{\bar{v}%
}-Id\right)  D\left(  u_{L}+\tilde{w}\right)  \right) \\
&  \text{ \ \ \ }+\operatorname{div}\left(  \mu\left(  \rho_{0}\right)
D_{A_{\bar{v}}-Id}\left(  u_{L}+\tilde{w}\right)  \right)  .
\end{align*}
Thus, using $\left(  \text{\ref{A1}}\right)  $ we get the following bound for
$T_{1}$:%
\begin{align}
\left\Vert T_{1}\right\Vert _{L_{T}^{1}(\dot{B}_{p,1}^{\frac{3}{p}-1})}  &
\lesssim C_{\rho_{0}}\left\Vert A_{\bar{v}}-Id\right\Vert _{L_{T}^{\infty
}(\dot{B}_{p,1}^{\frac{3}{p}})}\left(  1+\left\Vert A_{\bar{v}}-Id\right\Vert
_{L_{T}^{\infty}(\dot{B}_{p,1}^{\frac{3}{p}})}\right)  \left(  \left\Vert
\nabla u_{L}\right\Vert _{L_{T}^{1}(\dot{B}_{p,1}^{\frac{3}{p}})}+\left\Vert
\nabla\tilde{w}\right\Vert _{L_{T}^{1}(\dot{B}_{p,1}^{\frac{3}{p}})}\right)
\nonumber\\
&  \lesssim C_{\rho_{0}}\left\Vert \nabla\bar{v}\right\Vert _{L_{T}^{1}%
(\dot{B}_{p,1}^{\frac{3}{p}})}\left(  1+\left\Vert \nabla\bar{v}\right\Vert
_{L_{T}^{1}(\dot{B}_{p,1}^{\frac{3}{p}})}\right)  \left(  \left\Vert \nabla
u_{L}\right\Vert _{L_{T}^{1}(\dot{B}_{p,1}^{\frac{3}{p}})}+\left\Vert
\nabla\tilde{w}\right\Vert _{L_{T}^{1}(\dot{B}_{p,1}^{\frac{3}{p}})}\right)
\nonumber\\
&  \lesssim C_{\rho_{0}}\alpha\left(  \alpha+\left\Vert (\tilde{w}%
,\nabla\tilde{Q})\right\Vert _{F_{T}}\right)  . \label{T1.2}%
\end{align}
The second term is estimated as follows:%
\begin{align}
\left\Vert \left(  Id-A_{\bar{v}}^{T}\right)  (\nabla P_{L}+\nabla\tilde
{Q})\right\Vert _{L_{T}^{1}(\dot{B}_{p,1}^{\frac{3}{p}-1})}  &  \lesssim
\left\Vert \nabla\bar{v}\right\Vert _{L_{T}^{1}(\dot{B}_{p,1}^{\frac{3}{p}}%
)}\left(  \left\Vert \nabla P_{L}\right\Vert _{L_{T}^{1}(\dot{B}_{p,1}%
^{\frac{3}{p}-1})}+\left\Vert \nabla\tilde{Q}\right\Vert _{L_{T}^{1}(\dot
{B}_{p,1}^{\frac{3}{p}-1})}\right) \nonumber\\
&  \lesssim\alpha\left(  \alpha+\left\Vert (\tilde{w},\nabla\tilde
{Q})\right\Vert _{F_{T}}\right)  \label{T1.3}%
\end{align}
such that combining $\left(  \text{\ref{T1.1}}\right)  $, $\left(
\text{\ref{T1.2}}\right)  $ and $\left(  \text{\ref{T1.3}}\right)  $ we get
that:%
\begin{equation}
\left\Vert \frac{1}{\rho_{0}}F_{\bar{v}}\left(  u_{L}+\tilde{w},\nabla
P_{L}+\nabla\tilde{Q}\right)  \right\Vert _{L_{T}^{1}(\dot{B}_{p,1}^{\frac
{3}{p}-1})}\lesssim C_{\rho_{0}}\alpha\left(  \alpha+\left\Vert (\tilde
{w},\nabla\tilde{Q})\right\Vert _{F_{T}}\right)  . \label{T1}%
\end{equation}
In order to treat the second term of $\left(  \text{\ref{operator_2}}\right)
$ we use relations $\left(  \text{\ref{A1}}\right)  $, $\left(  \text{\ref{A2}%
}\right)  $ along with interpolation in order to obtain:%
\begin{align}
\left\Vert \partial_{t}\left(  Id-A_{\bar{v}}\right)  (u_{L}+\tilde
{w})\right\Vert _{L_{T}^{1}(\dot{B}_{p,1}^{\frac{3}{p}-1})}  &  \lesssim
\left\Vert \partial_{t}A_{\bar{v}}(u_{L}+\tilde{w})\right\Vert _{L_{T}%
^{1}(\dot{B}_{p,1}^{\frac{3}{p}-1})}+\left\Vert \left(  Id-A_{\bar{v}}\right)
(\partial_{t}u_{L}+\partial_{t}\tilde{w})\right\Vert _{L_{T}^{1}(\dot{B}%
_{p,1}^{\frac{3}{p}-1})}\nonumber\\
&  \lesssim\left\Vert \partial_{t}A_{\bar{v}}\right\Vert _{L_{T}^{2}(\dot
{B}_{p,1}^{\frac{3}{p}-1})}\left\Vert u_{L}+\tilde{w}\right\Vert _{L_{T}%
^{2}(\dot{B}_{p,1}^{\frac{3}{p}})}+\left\Vert Id-A_{\bar{v}}\right\Vert
_{L_{T}^{\infty}(\dot{B}_{p,1}^{\frac{3}{p}})}\left\Vert \partial_{t}%
u_{L}+\partial_{t}\tilde{w}\right\Vert _{L_{T}^{1}(\dot{B}_{p,1}^{\frac{3}%
{p}-1})}\nonumber\\
&  \lesssim\left\Vert \nabla\bar{v}\right\Vert _{L_{T}^{2}(\dot{B}%
_{p,1}^{\frac{3}{p}-1})}\left(  \alpha+\left\Vert (\tilde{w},\nabla\tilde
{Q})\right\Vert _{F_{T}}\right)  +\alpha\left(  \alpha+\left\Vert (\tilde
{w},\nabla\tilde{Q})\right\Vert _{F_{T}}\right) \nonumber\\
&  \lesssim\alpha\left(  \alpha+\left\Vert (\tilde{w},\nabla\tilde
{Q})\right\Vert _{F_{T}}\right)  . \label{T2}%
\end{align}
Treating the last term of $\left(  \text{\ref{operator_2}}\right)  $ is done
using the following formula:
\[
\operatorname{div}\left(  \left(  Id-A_{\bar{v}}\right)  (u_{L}+\tilde
{w})\right)  =\left(  Du_{L}+D\tilde{w}\right)  :\left(  Id-A_{\bar{v}%
}\right)
\]
which is a consequence of the fact that $\det DX_{\bar{v}}=1$ and Proposition
\ref{formula_fund}. Thus, we may write:
\begin{align}
\left\Vert \nabla\operatorname{div}\left(  \left(  Id-A_{\bar{v}}\right)
(u_{L}+\tilde{w})\right)  \right\Vert _{L_{T}^{1}(\dot{B}_{p,1}^{\frac{3}%
{p}-1})}  &  \lesssim\left\Vert \left(  Du_{L}+D\tilde{w}\right)  :\left(
Id-A_{\bar{v}}\right)  \right\Vert _{L_{T}^{1}(\dot{B}_{p,1}^{\frac{3}{p}}%
)}\nonumber\\
&  \lesssim\left\Vert Id-A_{\bar{v}}\right\Vert _{L_{T}^{1}(\dot{B}%
_{p,1}^{\frac{3}{p}})}\left\Vert Du_{L}+D\tilde{w}\right\Vert _{L_{T}^{1}%
(\dot{B}_{p,1}^{\frac{3}{p}})}\nonumber\\
&  \lesssim\alpha\left(  \alpha+\left\Vert (\tilde{w},\nabla\tilde
{Q})\right\Vert _{F_{T}}\right)  . \label{T3}%
\end{align}
Combining the estimates $\left(  \text{\ref{T1}}\right)  $, $\left(
\text{\ref{T2}}\right)  $ and $\left(  \text{\ref{T3}}\right)  $ we get that:%
\[
\left\Vert \Phi\left(  \tilde{w},\nabla\tilde{Q}\right)  \right\Vert _{E_{T}%
}\lesssim\alpha\left(  \alpha+\left\Vert (\tilde{w},\nabla\tilde
{Q})\right\Vert _{F_{T}}\right)  .
\]
Thus, for a suitably small $\alpha$ the operator $\Phi$ maps the unit ball
centered at the origin of $F_{T}$ into itself. Due to the linearity of $\Phi$
one can repeat the above arguments in order to show that for small values of
$\alpha$, $\Phi$ is a contraction. This concludes the existence of a fixed
point of $\Phi$, say $\left(  \tilde{u}^{\star},\nabla\tilde{P}^{\star
}\right)  \in F_{T}$. Of course,
\[
\left(  \bar{u},\nabla\bar{P}\right)  =\left(  \tilde{u}^{\star},\nabla
\tilde{P}^{\star}\right)  +\left(  u_{L},\nabla P_{L}\right)
\]
is a solution of $\left(  \text{\ref{NSL1}}\right)  $.

\subsection{Proof of Theorem \ref{NS1}}

Let us consider $T$ small enough such that $\left(  u_{L},\nabla P_{L}\right)
$ the solution of $\left(  \text{\ref{solution1}}\right)  $ satisfies
\[
\left\Vert \nabla u_{L}\right\Vert _{L_{T}^{2}(\dot{B}_{p,1}^{\frac{3}{p}-1}%
)}+\left\Vert \nabla u_{L}\right\Vert _{L_{T}^{1}(\dot{B}_{p,1}^{\frac{3}{p}%
})}\leq\alpha
\]
and let us consider the closed set:%
\[
\tilde{F}_{T}\left(  R\right)  =\left\{  \left(  \tilde{v},\nabla\tilde
{Q}\right)  \in F_{T}:\tilde{v}_{|t=0}=0,\det DX_{(u_{L}+\tilde{v})}=1,\text{
}\left\Vert \left(  \tilde{v},\nabla\tilde{Q}\right)  \right\Vert _{F_{T}}\leq
R\right\}
\]
with $R\leq\alpha$ sufficiently small such that:
\begin{equation}
\left\Vert \nabla\tilde{v}\right\Vert _{L_{T}^{2}(\dot{B}_{p,1}^{\frac{3}%
{p}-1})}+\left\Vert \nabla\bar{v}\right\Vert _{L_{T}^{1}(\dot{B}_{p,1}%
^{\frac{3}{p}})}\leq\alpha.
\end{equation}
\ Let us consider the operator $S$ which associates to $\left(  \tilde
{v},\nabla\tilde{Q}\right)  \in\tilde{F}_{T}\left(  R\right)  $, the solution
of:
\[
\left\{
\begin{array}
[c]{r}%
\partial_{t}\tilde{u}-\frac{1}{\rho_{0}}\operatorname{div}\left(  \mu\left(
\rho_{0}\right)  D\left(  \tilde{u}\right)  \right)  +\frac{1}{\rho_{0}}%
\nabla\tilde{P}=\frac{1}{\rho_{0}}F_{(u_{L}+\tilde{v})}\left(  u_{L}+\tilde
{u},\nabla P_{L}+\nabla\tilde{P}\right)  ,\\
\operatorname{div}\left(  A_{(u_{L}+\tilde{v})}(u_{L}+\tilde{u})\right)  =0,\\
\tilde{u}_{|t=0}=0,
\end{array}
\right.
\]
constructed in the previous section. We will show that that for a suitably
small $T$, the operator $S$ maps the closed set $\tilde{F}_{T}\left(
R\right)  $ into itself and that $S$ is a contraction. First of all, recalling
that $\left(  \tilde{u},\nabla\tilde{P}\right)  $ is in fact the fixed point
of the operator $\Phi$ defined above and using the estimates established in
the last section we conclude that
\[
\left\Vert S\left(  \tilde{v},\nabla\tilde{Q}\right)  \right\Vert _{F_{T}}\leq
R
\]
for some small enough $T$. Moreover, because
\[
\det DX_{(u_{L}+\tilde{v})}=1\text{ and }\operatorname{div}\left(
A_{(u_{L}+\tilde{v})}(u_{L}+\tilde{u})\right)  =0
\]
we invoke Proposition \ref{formula_fund} in order to conclude that%
\[
\det DX_{(u_{L}+\tilde{u})}=1
\]
so that $S$ maps $\tilde{F}_{T}\left(  R\right)  $ into itself.

Next, we will deal with the stability estimates. For $i=1,2,$ let us consider
$\left(  \tilde{v}_{i},\nabla\tilde{Q}_{i}\right)  \in\tilde{F}_{T}\left(
R\right)  $ and $\left(  \tilde{u}_{i},\nabla\tilde{P}_{i}\right)  =S\left(
\tilde{v}_{i},\nabla\tilde{Q}_{i}\right)  $. Denoting by%
\begin{align*}
\left(  \delta\tilde{v},\nabla\delta\tilde{Q}\right)   &  =\left(  \tilde
{v}_{1}-\tilde{v}_{2},\nabla\tilde{Q}_{1}-\nabla\tilde{Q}_{2}\right)  ,\\
\left(  \delta\tilde{u},\nabla\delta\tilde{P}\right)   &  =\left(  \tilde
{u}_{1}-\tilde{u}_{2},\nabla\tilde{P}_{1}-\nabla\tilde{P}_{2}\right)  ,
\end{align*}
we see that:%
\[
\left\{
\begin{array}
[c]{r}%
\partial_{t}\delta\tilde{u}-\frac{1}{\rho_{0}}\operatorname{div}\left(
\mu\left(  \rho_{0}\right)  D\left(  \delta\tilde{u}\right)  \right)
+\frac{1}{\rho_{0}}\nabla\delta\tilde{P}=\frac{1}{\rho_{0}}\tilde{F},\\
\operatorname{div}\left(  A_{(u_{L}+\tilde{v}_{1})}\delta\tilde{u})\right)
=\operatorname{div}\tilde{G},\\
\delta\tilde{u}_{|t=0}=0,
\end{array}
\right.
\]
where%
\begin{align*}
\tilde{F}  &  =F_{1}(\delta\tilde{v},u_{L}+\tilde{u}_{1})+F_{1}(u_{L}%
+\tilde{v}_{2},\delta\tilde{u})\\
&  +F_{2}\left(  \delta\tilde{v},\nabla P_{L}+\nabla\tilde{P}_{1}\right)
+F_{2}(u_{L}+\tilde{v}_{2},\nabla\delta\tilde{P}),\\
\tilde{G}  &  =-\left(  A_{(u_{L}+\tilde{v}_{1})}-A_{(u_{L}+\tilde{v}_{2}%
)}\right)  \left(  u_{L}+\tilde{u}_{2}\right)  ,
\end{align*}
and
\begin{align*}
F_{1}(\bar{v},\bar{w})  &  =\operatorname{div}\left(  \mu\left(  \rho
_{0}\right)  A_{\bar{v}}D_{A_{\bar{v}}}\left(  \bar{w}\right)  -\mu\left(
\rho_{0}\right)  D\left(  \bar{w}\right)  \right)  ,\\
F_{2}(\bar{v},\nabla\bar{Q})  &  =\left(  Id-A_{\bar{v}}^{T}\right)
\nabla\bar{Q}.
\end{align*}
According to Theorem \ref{Teorema_Stokes} we get that%
\begin{equation}
\left\Vert \left(  \delta\tilde{u},\nabla\delta\tilde{P}\right)  \right\Vert
_{F_{T}}\lesssim C_{\rho_{0}}\left(  \left\Vert \tilde{F}\right\Vert
_{L_{T}^{1}(\dot{B}_{p,1}^{\frac{3}{p}-1})}+\left\Vert \nabla
\operatorname{div}\tilde{G}\right\Vert _{L_{T}^{1}(\dot{B}_{p,1}^{\frac{3}%
{p}-1})}+\left\Vert \partial_{t}\tilde{G}\right\Vert _{L_{T}^{1}(\dot{B}%
_{p,1}^{\frac{3}{p}-1})}\right)  . \label{Relatie}%
\end{equation}
Proceeding as in relations $\left(  \text{\ref{T1.1}}\right)  $ and $\left(
\text{\ref{T1.2}}\right)  $ we get that%
\begin{align}
\left\Vert \tilde{F}\right\Vert _{L_{T}^{1}(\dot{B}_{p,1}^{\frac{3}{p}-1})}
&  \lesssim\left\Vert \nabla\delta\tilde{v}\right\Vert _{L_{T}^{1}(\dot
{B}_{p,1}^{\frac{3}{p}})}\left\Vert \nabla u_{L}+\nabla\tilde{u}%
_{1}\right\Vert _{L_{T}^{1}(\dot{B}_{p,1}^{\frac{3}{p}})}+\left\Vert \nabla
u_{L}+\nabla\tilde{v}_{2}\right\Vert _{L_{T}^{1}(\dot{B}_{p,1}^{\frac{3}{p}}%
)}\left\Vert \nabla\delta\tilde{u}\right\Vert _{L_{T}^{1}(\dot{B}_{p,1}%
^{\frac{3}{p}})}\nonumber\\
&  +\left\Vert \nabla\delta\tilde{v}\right\Vert _{L_{T}^{1}(\dot{B}%
_{p,1}^{\frac{3}{p}})}\left\Vert \nabla P_{L}+\nabla\tilde{P}_{1}\right\Vert
_{L_{T}^{1}(\dot{B}_{p,1}^{\frac{3}{p}})}+\left\Vert \nabla u_{L}+\nabla
\tilde{v}_{2}\right\Vert _{L_{T}^{1}(\dot{B}_{p,1}^{\frac{3}{p}})}\left\Vert
\nabla\delta\tilde{P}\right\Vert _{L_{T}^{1}(\dot{B}_{p,1}^{\frac{3}{p}}%
)}\nonumber\\
&  \lesssim\alpha\left\Vert \left(  \nabla\delta\tilde{v},\nabla\delta
\tilde{Q}\right)  \right\Vert _{L_{T}^{1}(\dot{B}_{p,1}^{\frac{3}{p}})}%
+\alpha\left\Vert \left(  \nabla\delta\tilde{u},\nabla\delta\tilde{P}\right)
\right\Vert _{L_{T}^{1}(\dot{B}_{p,1}^{\frac{3}{p}})}. \label{F}%
\end{align}
Of course, we will use the smallness of $\alpha$ to absorb $\alpha\left\Vert
\left(  \nabla\delta\tilde{u},\nabla\delta\tilde{P}\right)  \right\Vert
_{L_{T}^{1}(\dot{B}_{p,1}^{\frac{3}{p}})}$ into the LHS of $\left(
\text{\ref{Relatie}}\right)  $.

Next, we treat $\left\Vert \nabla\operatorname{div}\tilde{G}\right\Vert
_{L_{T}^{1}(\dot{B}_{p,1}^{\frac{3}{p}-1})}$. Using Proposition
\ref{formula_fund}, we write that%
\[
\operatorname{div}\tilde{G}=(Du_{L}+D\tilde{u}_{2}):\left(  A_{(u_{L}%
+\tilde{v}_{1})}-A_{(u_{L}+\tilde{v}_{2})}\right)
\]
and thus, using $\left(  \text{\ref{A4}}\right)  $%
\begin{align}
\left\Vert \nabla\operatorname{div}\tilde{G}\right\Vert _{L_{T}^{1}(\dot
{B}_{p,1}^{\frac{3}{p}-1})}  &  =\left\Vert (Du_{L}+D\tilde{u}_{2}):\left(
A_{(u_{L}+\tilde{v}_{1})}-A_{(u_{L}+\tilde{v}_{2})}\right)  \right\Vert
_{L_{T}^{1}(\dot{B}_{p,1}^{\frac{3}{p}-1})}\nonumber\\
&  \lesssim\left(  \left\Vert Du_{L}\right\Vert _{L_{T}^{1}(\dot{B}%
_{p,1}^{\frac{3}{p}})}+\left\Vert D\tilde{u}_{2}\right\Vert _{L_{T}^{1}%
(\dot{B}_{p,1}^{\frac{3}{p}})}\right)  \left\Vert A_{(u_{L}+\tilde{v}_{1}%
)}-A_{(u_{L}+\tilde{v}_{2})}\right\Vert _{L_{T}^{\infty}(\dot{B}_{p,1}%
^{\frac{3}{p}})}\nonumber\\
&  \lesssim\alpha\left\Vert \nabla\delta\tilde{v}\right\Vert _{L_{T}^{1}%
(\dot{B}_{p,1}^{\frac{3}{p}})}. \label{G1}%
\end{align}
Finally, we write that%
\begin{align*}
\left(  A_{(u_{L}+\tilde{v}_{1})}-A_{(u_{L}+\tilde{v}_{2})}\right)  \left(
u_{L}+\tilde{u}_{2}\right)   &  =\left(  \partial_{t}A_{(u_{L}+\tilde{v}_{1}%
)}-\partial_{t}A_{(u_{L}+\tilde{v}_{2})}\right)  \left(  u_{L}+\tilde{u}%
_{2}\right) \\
&  +\left(  A_{(u_{L}+\tilde{v}_{1})}-A_{(u_{L}+\tilde{v}_{2})}\right)
\left(  \partial_{t}u_{L}+\partial_{t}\tilde{u}_{2}\right)  .
\end{align*}
Using $\left(  \text{\ref{A4}}\right)  $, $\left(  \text{\ref{A5}}\right)  $
and $\left(  \text{\ref{A6}}\right)  $ gives us%
\begin{align*}
\left\Vert \left(  \partial_{t}A_{(u_{L}+\tilde{v}_{1})}-\partial_{t}%
A_{(u_{L}+\tilde{v}_{2})}\right)  \left(  u_{L}+\tilde{u}_{2}\right)
\right\Vert _{L_{T}^{1}(\dot{B}_{p,1}^{\frac{3}{p}-1})}  &  \lesssim\left\Vert
\partial_{t}A_{(u_{L}+\tilde{v}_{1})}-\partial_{t}A_{(u_{L}+\tilde{v}_{2}%
)}\right\Vert _{L_{T}^{2}(\dot{B}_{p,1}^{\frac{3}{p}-1})}\left\Vert
u_{L}\right\Vert _{L_{T}^{2}(\dot{B}_{p,1}^{\frac{3}{p}})}\\
&  +\left\Vert \partial_{t}A_{(u_{L}+\tilde{v}_{1})}-\partial_{t}%
A_{(u_{L}+\tilde{v}_{2})}\right\Vert _{L_{T}^{1}(\dot{B}_{p,1}^{\frac{3}{p}}%
)}\left\Vert \tilde{u}_{2}\right\Vert _{L_{T}^{\infty}(\dot{B}_{p,1}^{\frac
{3}{p}-1})}\\
&  \lesssim\left\Vert \nabla\delta\tilde{v}\right\Vert _{L_{T}^{2}(\dot
{B}_{p,1}^{\frac{3}{p}-1})}\left\Vert u_{L}\right\Vert _{L_{T}^{2}(\dot
{B}_{p,1}^{\frac{3}{p}})}+\left\Vert \nabla\delta\tilde{v}\right\Vert
_{L_{T}^{1}(\dot{B}_{p,1}^{\frac{3}{p}})}\left\Vert \tilde{u}_{2}\right\Vert
_{L_{T}^{\infty}(\dot{B}_{p,1}^{\frac{3}{p}-1})}\\
&  \lesssim\alpha\left\Vert \delta\tilde{v}\right\Vert _{L_{T}^{2}(\dot
{B}_{p,1}^{\frac{3}{p}})}+\alpha\left\Vert \nabla\delta\tilde{v}\right\Vert
_{L_{T}^{1}(\dot{B}_{p,1}^{\frac{3}{p}})}.
\end{align*}
Also, using $\left(  \text{\ref{A4}}\right)  $, we have that:%
\begin{align*}
&  \left\Vert \left(  A_{(u_{L}+\tilde{v}_{1})}-A_{(u_{L}+\tilde{v}_{2}%
)}\right)  \left(  \partial_{t}u_{L}+\partial_{t}\tilde{u}_{2}\right)
\right\Vert _{L_{T}^{1}(\dot{B}_{p,1}^{\frac{3}{p}-1})}\\
&  \lesssim\left\Vert A_{(u_{L}+\tilde{v}_{1})}-A_{(u_{L}+\tilde{v}_{2}%
)}\right\Vert _{L_{T}^{\infty}(\dot{B}_{p,1}^{\frac{3}{p}})}\left(  \left\Vert
\partial_{t}u_{L}\right\Vert _{L_{T}^{1}(\dot{B}_{p,1}^{\frac{3}{p}-1}%
)}+\left\Vert \partial_{t}u_{2}\right\Vert _{L_{T}^{1}(\dot{B}_{p,1}^{\frac
{3}{p}-1})}\right) \\
&  \lesssim\alpha\left\Vert \nabla\delta\tilde{v}\right\Vert _{L_{T}^{1}%
(\dot{B}_{p,1}^{\frac{3}{p}})}.
\end{align*}
The conclusion is that
\begin{equation}
\left\Vert \nabla\operatorname{div}\tilde{G}\right\Vert _{L_{T}^{1}(\dot
{B}_{p,1}^{\frac{3}{p}-1})}\lesssim\alpha\left\Vert \delta\tilde{v}\right\Vert
_{L_{T}^{2}(\dot{B}_{p,1}^{\frac{3}{p}})}+\alpha\left\Vert \nabla\delta
\tilde{v}\right\Vert _{L_{T}^{1}(\dot{B}_{p,1}^{\frac{3}{p}})}. \label{G2}%
\end{equation}
Gathering the information of $\left(  \text{\ref{F}}\right)  $, $\left(
\text{\ref{G1}}\right)  $ and $\left(  \text{\ref{G2}}\right)  $ we get that
if $\alpha$ is chosen sufficiently small then%
\begin{equation}
\left\Vert \left(  \left(  \delta\tilde{u},\nabla\delta\tilde{P}\right)
\right)  \right\Vert _{F_{T}}\leq\frac{1}{2}\left\Vert \left(  \left(
\delta\tilde{v},\nabla\delta\tilde{Q}\right)  \right)  \right\Vert _{F_{T}}
\label{contractie}%
\end{equation}
the operator $S$ is also a contraction over $\tilde{F}_{T}\left(  R\right)  $.
Thus, according to Banach's theorem there exists a fixed point $\left(
\bar{u}^{\star},\nabla\bar{P}^{\star}\right)  $ of $S$. Obviously,
\[
\left(  \bar{u},\nabla\bar{P}\right)  =\left(  u_{L},\nabla P_{L}\right)
+\left(  \bar{u}^{\star},\nabla\bar{P}^{\star}\right)
\]
is a solution of
\begin{equation}
\left\{
\begin{array}
[c]{r}%
\rho_{0}\partial_{t}\bar{u}-\operatorname{div}\left(  \mu\left(  \rho
_{0}\right)  A_{\bar{u}}D_{A_{\bar{u}}}\left(  \bar{u}\right)  \right)
+A_{\bar{u}}^{T}\nabla\bar{P}=0,\\
\operatorname{div}\left(  A_{\bar{u}}\bar{u}\right)  =0,\\
\bar{u}_{|t=0}=u_{0}.
\end{array}
\right.  \label{NSL2}%
\end{equation}
The only thing left to prove is the uniqueness property. Let us consider
$\left(  \bar{u}^{1},\nabla\bar{P}^{1}\right)  $, $\left(  \bar{u}^{2}%
,\nabla\bar{P}^{2}\right)  \in F_{T}$, two solutions of $\left(
\text{\ref{NSL2}}\right)  $ with the same initial data $u_{0}\in\dot{B}%
_{p,1}^{\frac{3}{p}-1}$. With $\left(  u_{L},\nabla P_{L}\right)  $ defined
above we let%
\[
\left(  \tilde{u}^{i},\nabla\tilde{P}^{i}\right)  =\left(  \bar{u}^{i}%
,\nabla\bar{P}^{i}\right)  -\left(  u_{L},\nabla P_{L}\right)  \text{ for
}i=1,2
\]
such that the system verified by $\left(  \tilde{u}^{i},\nabla\tilde{P}%
^{i}\right)  $ is
\[
\left\{
\begin{array}
[c]{r}%
\partial_{t}\tilde{u}^{i}-\frac{1}{\rho_{0}}\operatorname{div}\left(
\mu\left(  \rho_{0}\right)  D\left(  \tilde{u}^{i}\right)  \right)  +\frac
{1}{\rho_{0}}\nabla\tilde{P}^{i}=\frac{1}{\rho_{0}}F_{(u_{L}+\tilde{u}^{i}%
)}\left(  u_{L}+\tilde{u}^{i},\nabla P_{L}+\nabla\tilde{P}^{i}\right)  ,\\
\operatorname{div}\left(  A_{(u_{L}+\tilde{u}^{i})}(u_{L}+\tilde{u}%
^{i})\right)  =0,\\
\tilde{u}_{|t=0}=0.
\end{array}
\right.
\]
We are now in the position of performing exactly the same computations as
above such that we obtain a time $T^{^{\prime}}$ sufficiently small such that:%
\[
\left(  \bar{u}^{1},\nabla\bar{P}^{1}\right)  =\left(  \bar{u}^{2},\nabla
\bar{P}^{2}\right)  \text{ on }\left[  0,T^{\prime}\right]  .
\]
It is classical that the above local uniqueness property extends to all
$\left[  0,T\right]  $.

\subsubsection{Proof of Theorem \ref{NS2}}

Finally, we are in the position of proving the result announced in Theorem
\ref{NS2}. Considering $\left(  \rho_{0},u_{0}\right)  \in\dot{B}_{p,1}%
^{\frac{3}{p}}\times\dot{B}_{p,1}^{\frac{3}{p}-1}$ and applying Theorem
\ref{NS1}, there exists a positive $T>0$ such that we may construct a solution
$\left(  \bar{u},\nabla\bar{P}\right)  $ to the system $\left(
\text{\ref{NavierStokes_modif}}\right)  $ in $F_{T}$. Then, considering
$X_{\bar{u}}$, the "flow" of $\bar{u}$ defined by $\left(  \text{\ref{flow}%
}\right)  $ and using Proposition \ref{PropozitieFinal} from the Appendix, one
obtains that for all $t\in\lbrack0,T]$, $X_{\bar{u}}$ is a measure preserving
$C^{1}$-diffeormorphism over $\mathbb{R}^{n}$. Thus we may introduce the
Eulerian variable:%
\[
\rho\left(  t,x\right)  =\rho_{0}\left(  X_{\bar{u}}^{-1}\left(  t,x\right)
\right)  ,\text{ }u\left(  t,x\right)  =\bar{u}\left(  t,X_{\bar{u}}%
^{-1}\left(  t,x\right)  \right)  \text{ and }P\left(  t,x\right)  =\bar
{P}\left(  t,X_{\bar{u}}^{-1}\left(  t,x\right)  \right)  .
\]
Then, Proposition \ref{Formule} assures that $\left(  \rho,u,\nabla P\right)
$ is a solution of $\left(  \text{\ref{NavierStokes}}\right)  $. As
$DX_{\bar{u}}-Id$ belongs to $\dot{B}_{p,1}^{\frac{3}{p}}$ we may conclude
that $\left(  \rho,u,\nabla P\right)  $ has the announced regularity.

The uniqueness property comes from the fact that considering two solutions
$\left(  \rho^{i},u^{i},\nabla P^{i}\right)  $ of $\left(
\text{\ref{NavierStokes}}\right)  $, $i=1,2$, and considering $Y_{u^{i}}$ the
flow of $u^{i}$ we find that $\left(  u^{i}\left(  t,Y_{u^{i}}\left(
t,y\right)  \right)  ,\nabla P^{i}\left(  t,Y_{u^{i}}\left(  t,y\right)
\right)  \right)  $ are solutions of the system $\left(
\text{\ref{NavierStokes_modif}}\right)  $ with the same data. Thus, they are
equal according to the uniqueness property announced in Theorem \ref{NS1}.
Thus, on some nontrivial interval $\left[  0,T^{\prime}\right]  \subset\left[
0,T\right]  $, (chosen such as the condition $\left(  \text{\ref{smallnes12}%
}\right)  $ holds), the solutions $\left(  \rho^{i},u^{i},\nabla P^{i}\right)
$ are equal. This local uniqueness property obviously entails uniqueness on
all $\left[  0,T\right]  $.

\section{Appendix\label{Appendix}}

We present here a few results of Fourier analysis used through the text. The
full proofs along with other complementary results can be found in
\cite{Dan1}, Chapter $2$.

Let us introduce the dyadic partition of the space:

\begin{proposition}
\label{diadic}Let $\mathcal{C}$ be the annulus $\{\xi\in\mathbb{R}^{n}%
:3/4\leq\left\vert \xi\right\vert \leq8/3\}$. There exist a radial function
$\varphi\in\mathcal{D(C)}$ valued in the interval $\left[  0,1\right]  $ and
such that:%
\begin{align}
\forall\xi &  \in\mathbb{R}^{n}\backslash\{0\}\text{, \ }\sum_{j\in\mathbb{Z}%
}\varphi(2^{-j}\xi)=1\text{,}\label{26}\\
2  &  \leq\left\vert j-j^{\prime}\right\vert \Rightarrow\mathrm{Supp}%
(\varphi(2^{-j}\cdot))\cap\mathrm{Supp}(\varphi(2^{-j^{\prime}}\cdot
))=\emptyset. \label{28}%
\end{align}
Also, the following inequality holds true:%
\begin{equation}
\forall\xi\in\mathbb{R}^{n}\backslash\{0\}\text{, \ }\frac{1}{2}\leq\sum
_{j\in\mathbb{Z}}\varphi^{2}(2^{-j}\xi)\leq1\text{.} \label{211}%
\end{equation}

\end{proposition}

From now on we fix a functions $\chi$ and $\varphi$ satisfying the assertions
of the above proposition and let us denote by $\tilde{h}$ respectively $h$
their Fourier inverses.

The homogeneous dyadic blocks $\dot{\Delta}_{j}$ and the homogeneous low
frequency cutt-off operators $\dot{S}_{j}$ are defined below:%
\begin{align*}
\dot{\Delta}_{j}u  &  =\varphi\left(  2^{-j}D\right)  u=2^{jn}\int
_{\mathbb{R}^{n}}h\left(  2^{j}y\right)  u\left(  x-y\right)  dy\\
\dot{S}_{j}u  &  =\chi\left(  2^{-j}D\right)  u=2^{jn}\int_{\mathbb{R}^{n}%
}\tilde{h}\left(  2^{j}y\right)  u\left(  x-y\right)  dy
\end{align*}
for all $j\in\mathbb{Z}$.

\begin{definition}
We denote by $\mathcal{S}_{h}^{\prime}$ the space of tempered distributions
such that:%
\[
\lim_{j\rightarrow-\infty}\left\Vert \dot{S}_{j}u\right\Vert _{L^{\infty}}=0.
\]

\end{definition}

Let us now define the homogeneous Besov spaces:

\begin{definition}
Let $s$ be a real number and $\left(  p,r\right)  \in\left[  1,\infty\right]
$. The homogenous Besov space $\dot{B}_{p,r}^{s}$ is the subset of tempered
distributions $u\in S_{h}^{\prime}$ such that:%
\[
\left\Vert u\right\Vert _{\dot{B}_{p,r}^{s}}:=\left\Vert \left(
2^{js}\left\Vert \dot{\Delta}_{j}u\right\Vert _{L^{2}}\right)  _{j\in
\mathbb{Z}}\right\Vert _{\ell^{r}\left(  \mathbb{Z}\right)  }<\infty.
\]

\end{definition}

The next propositions gather some basic properties of Besov spaces.

\begin{proposition}
\label{Banach}Let us consider $s\in\mathbb{R}$ and $p,r\in\left[
1,\infty\right]  $ such that
\begin{equation}
s<\frac{n}{p}\text{ or }s=\frac{n}{p}\text{ and }r=1.
\label{lowfrequencypartcondition}%
\end{equation}
Then $\left(  \dot{B}_{p,r}^{s},\left\Vert \cdot\right\Vert _{\dot{B}%
_{p,r}^{s}}\right)  $ is a Banach space.
\end{proposition}

\begin{proposition}
\label{propA}A tempered distribution $u\in S_{h}^{\prime}$ belongs to $\dot
{B}_{p,r}^{s}\left(  \mathbb{R}^{n}\right)  $ if and only if there exists a
sequence $\left(  c_{j}\right)  _{j}$ such that $\left(  2^{js}c_{j}\right)
_{j}\in\ell^{r}(\mathbb{Z)}$ with norm $1$ and a constant $C=C\left(
u\right)  >0$ such that for any $j\in\mathbb{Z}$ we have%
\[
\left\Vert \dot{\Delta}_{j}u\right\Vert _{L^{p}}\leq Cc_{j}.
\]

\end{proposition}

\begin{proposition}
\label{interpolation} Let us consider $s_{1}$ and $s_{2}$ two real
numbers such that $s_{1}<s_{2}$ and $\theta\in\left(  0,1\right)  $. Then,
there exists a constant $C>0$ such that for all $r\in\left[  1,\infty\right]
$ we have:%
\begin{align*}
\left\Vert u\right\Vert _{\dot{B}_{p,r}^{\theta s_{1}+\left(  1-\theta\right)
s_{2}}}  &  \leq\left\Vert u\right\Vert _{\dot{B}_{p,r}^{s_{1}}}^{\theta
}\left\Vert u\right\Vert _{\dot{B}_{p,r}^{s_{2}}}^{1-\theta}\text{ and}\\
\left\Vert u\right\Vert _{\dot{B}_{p,1}^{\theta s_{1}+\left(  1-\theta\right)
s_{2}}}  &  \leq\frac{C}{s_{2}-s_{1}}\left(  \frac{1}{\theta}+\frac
{1}{1-\theta}\right)  \left\Vert u\right\Vert _{\dot{B}_{p,\infty}^{s_{1}}%
}^{\theta}\left\Vert u\right\Vert _{\dot{B}_{p,\infty}^{s_{2}}}^{1-\theta}%
\end{align*}

\end{proposition}

\begin{proposition}
\label{Embedding}Let $1\leq p_{1}\leq p_{2}\leq\infty$ and $1\leq r_{1}\leq
r_{2}\leq\infty$. Then, for any real number $s$, the space $\dot{B}%
_{p_{1},r_{1}}^{s}$ is continuously embedded in $\dot{B}_{p_{2},r_{2}%
}^{s-n\left(  \frac{1}{p_{1}}-\frac{1}{p_{2}}\right)  }$.
\end{proposition}

\begin{proposition}
\label{Caracterizare}For all $1\leq p,r\leq\infty$ and $s\in\mathbb{R}$,%
\begin{equation}
\left\{
\begin{array}
[c]{c}%
\dot{B}_{p,r}^{s}\times\dot{B}_{p^{\prime},r^{\prime}}^{-s}\rightarrow
\mathbb{R},\\
\left(  u,v\right)  \rightarrow%
{\displaystyle\sum\limits_{j}}
\left\langle \dot{\Delta}_{j}u,\tilde{\Delta}_{j}v\right\rangle ,
\end{array}
\right.  \label{braket}%
\end{equation}
where $\tilde{\Delta}_{j}:=\dot{\Delta}_{j-1}+\dot{\Delta}_{j}+\dot{\Delta
}_{j+1}$, defines a continuous bilinear functional on $\dot{B}_{p,r}^{s}%
\times\dot{B}_{p^{\prime},r^{\prime}}^{-s}$. Denote by $Q_{p^{\prime
},r^{\prime}}^{-s}$ the set of functions $\phi\in\mathcal{S\cap}\dot
{B}_{p^{\prime},r^{\prime}}^{-s}$ such that $\left\Vert \phi\right\Vert
_{\dot{B}_{p^{\prime},r^{\prime}}^{-s}}\leq1$. If $u\in\mathcal{S}_{h}%
^{\prime}$, then we have%
\[
\left\Vert u\right\Vert _{\dot{B}_{p,r}^{s}}\lesssim\sup_{\phi\in
Q_{p^{\prime},r^{\prime}}^{-s}}\left\langle u,\phi\right\rangle _{\mathcal{S}%
^{\prime}\times\mathcal{S}}.
\]

\end{proposition}

\begin{proposition}
\label{autoadjunct}Let us consider $1<p,r<\infty$ and $s\in\mathbb{R}$.
Furthermore, let $u\in\dot{B}_{p,r}^{s}$, $v\in\dot{B}_{p^{\prime},r^{\prime}%
}^{-s}$ and $\rho\in L^{\infty}\cap\mathcal{M}\left(  \dot{B}_{p,r}%
^{s}\right)  \cap\mathcal{M}\left(  \dot{B}_{p^{\prime},r^{\prime}}%
^{-s}\right)  $. Then, we have that%
\begin{equation}
\left(  \rho u,v\right)  =%
{\displaystyle\sum\limits_{j}}
\left\langle \dot{\Delta}_{j}(\rho u),\tilde{\Delta}_{j}v\right\rangle =%
{\displaystyle\sum\limits_{j}}
\left\langle \dot{\Delta}_{j}u,\tilde{\Delta}_{j}(\rho v)\right\rangle
=\left(  u,\rho v\right)  .\label{autoadj1}%
\end{equation}

\end{proposition}

The proof of Proposition \ref{autoadjunct} follows from a density argument.
Relation $\left(  \text{\ref{autoadj1}}\right)  $ clearly holds for functions
from the Schwartz class: then we may write%
\[
\int_{\mathbb{R}^{n}}\rho uv=\left(  \rho u,v\right)  =\left(  u,\rho
v\right)  .
\]
The condition $1<p,r<\infty$ and $s\in\mathbb{R}$ ensures that $u$ and $v$ may
be approximated by Schwartz functions.

An important feature of Besov spaces with negative index of regularity is the following:

\begin{proposition}
\label{snegative}Let $s<0$ and $1\leq p,r\leq\infty$. Let $u$ be a
distribution in $S_{h}^{\prime}$. Then, $u$ belongs to $\dot{B}_{p,r}^{s}$ if
and only if%
\[
\left(  2^{js}\left\Vert \dot{S}_{j}u\right\Vert _{L^{p}}\right)
_{j\in\mathbb{Z}}\in\ell^{r}\left(  \mathbb{Z}\right)  .
\]
Moreover, there exists a constant $C$ depending only on the dimension $n$ such
that:%
\[
C^{-\left\vert s\right\vert +1}\left\Vert u\right\Vert _{\dot{B}_{p,r}^{s}%
}\leq\left\Vert \left(  2^{js}\left\Vert \dot{S}_{j}u\right\Vert _{L^{p}%
}\right)  _{j\in\mathbb{Z}}\right\Vert _{\ell^{r}\left(  \mathbb{Z}\right)
}\leq C\left(  1+\frac{1}{\left\vert s\right\vert }\right)  \left\Vert
u\right\Vert _{\dot{B}_{p,r}^{s}}.
\]

\end{proposition}

The next proposition tells us how certain multipliers act on Besov spaces.

\begin{proposition}
Let us consider $A$ a smooth function on $\mathbb{R}^{n}\backslash\{0\}$ which
is homogeneous of degree $m$. Then, for any $\left(  s,p,r\right)
\in\mathbb{R\times}\left[  1,\infty\right]  ^{2}$ such that%
\[
s-m<\frac{n}{p}\text{ or }s-m=\frac{n}{p}\text{ and }r=1
\]
the operator\footnote{$A\left(  D\right)  w=\mathcal{F}^{-1}\left(
A\mathcal{F}w\right)  $} $A\left(  D\right)  $ maps $\dot{B}_{p,r}^{s}$
continuously into $\dot{B}_{p,r}^{s-m}$.
\end{proposition}

The next proposition describes how smooth functions act on homogeneous Besov spaces.

\begin{proposition}
\label{compunere}Let $f$ be a smooth function on $\mathbb{R}$ which vanishes
at $0$. Let us consider $\left(  s,p,r\right)  \in\mathbb{R\times}\left[
1,\infty\right]  ^{2}$ such that%
\[
0<s<\frac{n}{p}\text{ or }s=\frac{n}{p}\text{ and }r=1.
\]
Then for any real-valued function $u\in\dot{B}_{p,r}^{s}\cap L^{\infty}$, the
function $f\circ u\in\dot{B}_{p,r}^{s}\cap L^{\infty}$ and we have%
\[
\left\Vert f\circ u\right\Vert _{\dot{B}_{p,r}^{s}}\leq C\left(  f^{\prime
},\left\Vert u\right\Vert _{L^{\infty}}\right)  \left\Vert u\right\Vert
_{\dot{B}_{p,r}^{s}}.
\]

\end{proposition}

\begin{remark}
The constant $C\left(  f^{\prime},\left\Vert u\right\Vert _{L^{\infty}%
}\right)  $ appearing above can be taken to be%
\[
\sup_{i\in\overline{1,[s]+1}}\left\Vert f^{(i)}\right\Vert _{L^{\infty}\left(
|[-M\left\Vert u\right\Vert _{L^{\infty}},-M\left\Vert u\right\Vert
_{L^{\infty}}]\right)  }%
\]
where $M$ is a constant depending only on the dimension $n$.
\end{remark}

\subsection{Commutator and product estimates}

Next, we want to see how the product acts in Besov spaces. The Bony
decomposition, introduced in \cite{Bony1} offers a mathematical framework to
obtain estimates of the product of two distributions, when the later is defined.

\begin{definition}
\label{paraprodus_rest}Given two tempered distributions $u,v\in S_{h}^{\prime
}$, the homogeneous paraproduct of $v$ by $u$ is defined as:%
\begin{equation}
\dot{T}_{u}v=\sum_{j\in\mathbb{Z}}\dot{S}_{j-1}u\dot{\Delta}_{j}v.
\label{Paraproduct}%
\end{equation}
The homogeneous remainder of $u$ and $v$ is defined by:%
\begin{equation}
\dot{R}\left(  u,v\right)  =\sum_{j\in\mathbb{Z}}\dot{\Delta}_{j}u\dot{\Delta
}_{j}^{\prime}v \label{Reminder}%
\end{equation}
where%
\[
\dot{\Delta}_{j}^{\prime}=\dot{\Delta}_{j-1}+\dot{\Delta}_{j}+\dot{\Delta
}_{j+1}.
\]

\end{definition}

\begin{remark}
 Notice that at a formal level, one has the following decomposition of
the product of two (sufficiently well-behaved) distributions:%
\[
uv=\dot{T}_{u}v+\dot{T}_{v}u+\dot{R}\left(  u,v\right)  =\dot{T}_{u}v+\dot
{T}_{v}^{\prime}u.
\]

\end{remark}

The next result describes how the paraproduct and remainder behave.

\begin{proposition}
\label{ParaRem}1) Assume that $(s,p,p_{1},p_{2},r)\in\mathbb{R\times}\left[
1,\infty\right]  ^{4}$ such that:%
\[
\frac{1}{p}=\frac{1}{p_{1}}+\frac{1}{p_{2}},\text{ }s<\frac{n}{p}\text{ or
}s=\frac{n}{p}\text{ and }r=1\text{.}%
\]
Then, the paraproduct maps $L^{p_{1}}\times\dot{B}_{p_{2},r}^{s}$ into
$\dot{B}_{p,r}^{s}$ and the following estimates hold true:%
\[
\left\Vert \dot{T}_{f}g\right\Vert _{\dot{B}_{p,r}^{s}}\lesssim\left\Vert
f\right\Vert _{L^{p_{1}}}\left\Vert g\right\Vert _{\dot{B}_{p_{2},r}^{s}}.
\]

2) Assume that $(s,p,p_{1},p_{2},r,r_{1},r_{2})\in\mathbb{R\times}\left[
1,\infty\right]  ^{6}$ and $\nu>0$ such that%
\[
\frac{1}{p}=\frac{1}{p_{1}}+\frac{1}{p_{2}},\text{ }\frac{1}{r}=\frac{1}%
{r_{1}}+\frac{1}{r_{2}}%
\]
and%
\[
s<\frac{n}{p}-\nu\text{ or }s=\frac{n}{p}-\nu\text{ and }r=1\text{.}%
\]
Then, the paraproduct maps $\dot{B}_{p_{1},r_{1}}^{-\nu}\times\dot{B}%
_{p_{2},r_{2}}^{s+\nu}$ into $\dot{B}_{p,r}^{s}$ and the following estimate
holds true:%
\[
\left\Vert \dot{T}_{f}g\right\Vert _{\dot{B}_{p,r}^{s}}\lesssim\left\Vert
f\right\Vert _{\dot{B}_{p_{1},r_{1}}^{-\nu}}\left\Vert g\right\Vert _{\dot
{B}_{p_{2},r_{2}}^{s+\nu}}.
\]

3) Let us consider $\left(  s_{1},s_{2,}p,p_{1},p_{2},r,r_{1},r_{2}\right)
\in\mathbb{R}^{2}\times\left[  1,\infty\right]  ^{6}$ such that
\[
0<s_{1}+s_{2}<\frac{n}{p}\text{ or }s_{1}+s_{2}=\frac{n}{p}\text{ and }r=1.
\]
Then, the remainder maps $\dot{B}_{p_{1},r_{1}}^{s_{1}}\times\dot{B}%
_{p_{2},r_{2}}^{s_{2}}$ into $\dot{B}_{p,r}^{s_{1}+s_{2}}$ and
\[
\left\Vert \dot{R}\left(  f,g\right)  \right\Vert _{\dot{B}_{p,r}^{s_{1}%
+s_{2}}}\leq\left\Vert f\right\Vert _{\dot{B}_{p_{1},r_{1}}^{s_{1}}}\left\Vert
g\right\Vert _{\dot{B}_{p_{2},r_{2}}^{s_{2}}}.
\]

\end{proposition}

As a consequence we obtain the following product rules in Besov space:

\begin{proposition}
\label{Produs}Consider $p\in\left[  1,\infty\right]  $ and the real numbers
$\nu_{1}\geq0$ and $\nu_{2}\geq0$%
\[
\nu_{1}+\nu_{2}<\frac{n}{p}+\min\left\{  \frac{n}{p},\frac{n}{p^{\prime}%
}\right\}  .
\]
Then, the following estimate holds true:%
\[
\left\Vert fg\right\Vert _{\dot{B}_{p,1}^{\frac{n}{p}-\nu_{1}-\nu_{2}}%
}\lesssim\left\Vert f\right\Vert _{\dot{B}_{p,1}^{\frac{n}{p}-\nu_{1}}%
}\left\Vert g\right\Vert _{\dot{B}_{p,1}^{\frac{n}{p}-\nu_{2}}}.
\]

\end{proposition}

\begin{proposition}
\label{Comutator1}Let us consider $\theta$ a $\mathcal{C}^{1}$ function on
$\mathbb{R}^{n}$ such that $\left(  1+\left\vert \cdot\right\vert \right)
\hat{\theta}\in L^{1}$. Let us also consider $p,q\in\left[  1,\infty\right]  $
such that:%
\[
\frac{1}{r}:=\frac{1}{p}+\frac{1}{q}\leq1.
\]
Then, there exists a constant $C$ such that for any Lipschitz function $a$
with gradient in $L^{p}$, any function $b\in L^{q}$ and any positive $\lambda
$:%
\[
\left\Vert \left[  \theta\left(  \lambda^{-1}D\right)  ,a\right]  b\right\Vert
_{L^{r}}\leq C\lambda^{-1}\left\Vert \nabla a\right\Vert _{L^{p}}\left\Vert
b\right\Vert _{L^{q}}.
\]
In particular, when $\theta=\varphi$ and $\lambda=2^{j}$ we get that:%
\[
\left\Vert \left[  \dot{\Delta}_{j},a\right]  b\right\Vert _{L^{r}}\leq
C2^{-j}\left\Vert \nabla a\right\Vert _{L^{p}}\left\Vert b\right\Vert _{L^{q}%
}.
\]

\end{proposition}

\begin{proposition}
\label{A66}Assume that $s,\nu$ and $p\in\left[  1,\infty\right]  $ are such
that%
\[
0\leq\nu\leq\frac{n}{p}\text{ and }-1-\min\left\{  \frac{n}{p},\frac
{n}{p^{\prime}}\right\}  <s\leq\frac{n}{p}-\nu.
\]
Then, there exists a constant $C$ depending only on $s,\nu,p$ and $n$ such
that for all $l\in\overline{1,n}$ we have for some sequence $\left(
c_{j}\right)  _{j\in\mathbb{Z}}$ with $\left\Vert \left(  c_{j}\right)
_{j\in\mathbb{Z}}\right\Vert _{\ell^{1}\left(  \mathbb{Z}\right)  }=1$:%
\[
\left\Vert \partial_{l}\left[  a,\dot{\Delta}_{j}\right]  w\right\Vert
_{L^{p}}\leq Cc_{j}2^{-js}\left\Vert \nabla a\right\Vert _{\dot{B}%
_{p,1}^{\frac{n}{p}-\nu}}\left\Vert w\right\Vert _{\dot{B}_{p,1}^{s+\nu}}%
\]
for all $j\in\mathbb{Z}$.
\end{proposition}

For a proof of the above results we refer the reader to the Appendix of
\cite{Dan2}, Lemma $A.5.$ and Lemma $A.6$.

\begin{proposition}
\label{A71}Let us consider a homogeneous function $A:\mathbb{R}^{n}%
\backslash\{0\}\rightarrow\mathbb{R}$ of degree $0$. Let us consider
$s\in\mathbb{R}$, $0<\nu\leq1$ and $p,r,r_{1},r_{2}\in\left[  1,\infty\right]
$ such that%
\[
\frac{1}{r}=\frac{1}{r_{1}}+\frac{1}{r_{2}}%
\]
and%
\begin{equation}
s<\frac{n}{p}-\nu\text{ or }s=\frac{n}{p}-\nu\text{ and }r_{2}=1\text{. }
\label{restrictie_indici}%
\end{equation}
Moreover, assume that $w\in\dot{B}_{p,r_{2}}^{s+\nu}$ and that $a\in
L^{\infty}$ with $\nabla a\in\dot{B}_{\infty,r_{1}}^{-\nu}$. Then, the
following estimate holds true:%
\begin{equation}
\left\Vert \lbrack A\left(  D\right)  ,\dot{T}_{a}]w\right\Vert _{\dot
{B}_{p,r}^{s+1}}\lesssim\left\Vert \nabla a\right\Vert _{\dot{B}_{\infty
,r_{1}}^{-\nu}}\left\Vert w\right\Vert _{\dot{B}_{p,r_{2}}^{s+\nu}}.
\label{estimA71}%
\end{equation}

\end{proposition}

As this result is of great importance in the analysis of the pressure term, we
present a sketched proof below (see also \cite{Dan1}, Chapter 2, Lemma $2.99$)

\begin{proof}
The fact that $a\in L^{\infty}$ along with relation $\left(
\text{\ref{restrictie_indici}}\right)  $ guarantees that $A\left(  D\right)
w\in\dot{B}_{p,r}^{s+\nu}$ and that the paraproducts $\dot{T}_{a}w$ and
$\dot{T}_{a}A\left(  D\right)  w$ are well-defined. We observe that there
exists a function $\tilde{\varphi}$ supported in some annulus which equals $1$
on the support of $\varphi$ such that one may write (of course it is here that
we use the homogeneity of $A$):%
\[
\lbrack A\left(  D\right)  ,\dot{T}_{a}]w=\sum_{j}\left[  (A\tilde{\varphi
})(2^{-j}D),\dot{S}_{j-1}a\right]  \dot{\Delta}_{j}w.
\]
But according to Lemma \ref{Comutator1} we have%
\[
2^{j(s+1)}\left\Vert \left[  (A\tilde{\varphi})(2^{-j}D),\dot{S}%
_{j-1}a\right]  \dot{\Delta}_{j}w\right\Vert _{L^{p}}\lesssim2^{-j\nu
}\left\Vert \nabla\dot{S}_{j-1}a\right\Vert _{L^{\infty}}2^{j(s+\nu
)}\left\Vert \dot{\Delta}_{j}w\right\Vert _{L^{p}}.
\]
The last relation obviously implies $\left(  \text{\ref{estimA71}}\right)  $.
\end{proof}

As a consequence of the above proposition and Proposition \ref{ParaRem} we get
the following:

\begin{proposition}
\label{A7}Let us consider a homogeneous function $A:\mathbb{R}^{n}%
\backslash\{0\}\rightarrow\mathbb{R}$ of degree $0$, $s\in\mathbb{R}$,
$0<\nu\leq1$ and $p,r,r_{1},r_{2}\in\left[  1,\infty\right]  $ such that%
\[
\frac{1}{r}=\frac{1}{r_{1}}+\frac{1}{r_{2}}%
\]
and%
\begin{equation}
-1-\min\left\{  \frac{n}{p},\frac{n}{p^{\prime}}\right\}  <s<\frac{n}{p}%
-\nu\text{ or }s=\frac{n}{p}-\nu\text{ and }r=r_{2}=1\text{. }%
\end{equation}
assume that $w\in\dot{B}_{p,r_{2}}^{s+\nu}$ and that $a\in L^{\infty}$ with
$\nabla a\in\dot{B}_{\infty,r_{1}}^{-\nu}$. Then, the following estimate hold
true:%
\[
\left\Vert \lbrack A\left(  D\right)  ,a]w\right\Vert _{\dot{B}_{p,r}^{s+1}%
}\lesssim\left\Vert \nabla a\right\Vert _{\dot{B}_{p,r_{1}}^{\frac{n}{p}-\nu}%
}\left\Vert w\right\Vert _{\dot{B}_{p,r_{2}}^{s+\nu}}.
\]

\end{proposition}

\subsection{Properties of Lagrangian coordinates}

\begin{proposition}
\label{Comp}Let $X$ be a globally defined bi-lipschitz diffeomorphism of
$\mathbb{R}^{3}$ and $-\frac{3}{p^{\prime}}<s\leq\frac{3}{p}$. Then
$a\rightarrow a\circ X$ is a self map over $\dot{B}_{p,1}^{s}$ whenever%
\begin{align*}
1)\text{ }s  &  \in\left(  0,1\right)  ;\\
2)\text{ }s  &  \geq1\text{ and }\left(  DX-Id\right)  \in\dot{B}_{p,1}%
^{\frac{3}{p}}.
\end{align*}

\end{proposition}

\begin{proposition}
\label{Formule}Let $K$ be a $C^{1}$ scalar function over $\mathbb{R}^{3}$ and
$H$ a $C^{1}$ vector field. Let $X$ be a $C^{1}$ diffeomorphism such that
$\det\left(  DX\right)  =1$. Then, the following relations hold true:%
\begin{align*}
(\nabla K)\circ X  &  =\operatorname{div}\left(  DX^{-1}K\circ X\right)  ,\\
(\operatorname{div}H)\circ X  &  =\operatorname{div}\left(  DX^{-1}H\circ
X\right)  .
\end{align*}

\end{proposition}

\begin{proposition}
\label{formula_fund}Let us consider $v$ and $w$ two time-dependent vector
fields with coefficients in $L_{T}^{1}(C^{0,1})$. Let us denote by $Y_{v}$ and
$Y_{w}$ their corresponding flows. We denote by $A_{v}=\left(  DY_{v}\right)
^{-1}$ and $A_{w}=\left(  DY_{w}\right)  ^{-1}$. Let us also assume that
\[
\det DY_{v}=1\text{ and }\operatorname{div}\left(  A_{v}w\circ Y_{v}\right)
=0.
\]
Then,%
\[
\det DY_{w}=1
\]
and for any $C^{1}$-vector field $H$ one has%
\[
\operatorname{div}H=\left(  D\left(  H\circ Y_{v}\right)  :A_{v}\right)  \circ
Y_{v}^{-1}=\operatorname{div}\left(  A_{w}H\circ Y_{w}\right)  \circ
Y_{w}^{-1}.
\]

\end{proposition}

This result interferes in a crucial manner in the proof of the well-posedness
result for the inhomogeneous incompressible Navier-Stokes system. For a proof
and other remarks see Corollary $2$ from the Appendix of \cite{Dan5}.

For any $\bar{v}$ a time dependent vector field we set:
\begin{equation}
X_{\bar{v}}\left(  t,y\right)  =y+\int_{0}^{t}\bar{v}\left(  \tau,y\right)
d\tau\label{flow}%
\end{equation}
and we denote%
\begin{equation}
A_{\bar{v}}=\left(  DX_{\bar{v}}\right)  ^{-1}. \label{inversa}%
\end{equation}

\begin{proposition}
\label{PropozitieFinal}Let us consider $\bar{v}\in C_{b}(\left[  0,T\right]
,\dot{B}_{p,1}^{\frac{3}{p}-1})$ with $\partial_{t}\bar{v},$ $\nabla^{2}%
\bar{v}\in L_{T}^{1}(\dot{B}_{p,1}^{\frac{3}{p}-1})$. Then, there exists a
positive $\alpha$ such that if
\begin{equation}
\left\Vert \nabla\bar{v}\right\Vert _{L_{t}^{1}(\dot{B}_{p,1}^{\frac{3}{p}}%
)}\leq\alpha, \label{smallnes12}%
\end{equation}
then, $X_{\bar{v}}$ introduced in $\left(  \text{\ref{flow}}\right)  $ is a
global $C^{1}$-diffeormorphism over $\mathbb{R}^{3}$. Moreover, if
\[
\operatorname{div}\left(  A_{\bar{v}}\bar{v}\right)  =0
\]
then, $X_{\bar{v}}$ is measure preserving i.e.%
\[
\det DX_{\bar{v}}=1.
\]

\end{proposition}

\begin{proposition}
Let us consider $\bar{v}\in E_{T}$ satisfying the smallness condition $\left(
\text{\ref{smallness1}}\right)  $. Let $X_{v}$ be defined by $\left(
\text{\ref{flow}}\right)  $. Then for all $t\in\lbrack0,T]:$%
\begin{align}
\left\Vert Id-A_{\bar{v}}\left(  t\right)  \right\Vert _{\dot{B}_{p,1}%
^{\frac{3}{p}}}  &  \lesssim\left\Vert \nabla\bar{v}\right\Vert _{L_{t}%
^{1}(\dot{B}_{p,1}^{\frac{3}{p}})},\label{A1}\\
\left\Vert \partial_{t}A_{\bar{v}}\left(  t\right)  \right\Vert _{\dot
{B}_{p,1}^{\frac{3}{p}-1}}  &  \lesssim\left\Vert \nabla\bar{v}(t)\right\Vert
_{\dot{B}_{p,1}^{\frac{3}{p}-1}},\label{A2}\\
\left\Vert \partial_{t}A_{\bar{v}}\left(  t\right)  \right\Vert _{\dot
{B}_{p,1}^{\frac{3}{p}}}  &  \lesssim\left\Vert \nabla\bar{v}(t)\right\Vert
_{\dot{B}_{p,1}^{\frac{3}{p}}}. \label{A3}%
\end{align}

\end{proposition}

In order to establish stability estimates we use the following

\begin{proposition}
Let $\bar{v}_{1},\bar{v}_{2}\in E_{T}$ satisfying the smallness condition
$\left(  \text{\ref{smallnes12}}\right)  $ and $\delta v=\bar{v}_{2}-\bar
{v}_{1}$. Then we have:%
\begin{align}
\left\Vert A_{\bar{v}_{1}}-A_{\bar{v}_{2}}\right\Vert _{L_{T}^{\infty}(\dot
{B}_{p,1}^{\frac{3}{p}})}  &  \lesssim\left\Vert \nabla\delta v\right\Vert
_{L_{T}^{1}(\dot{B}_{p,1}^{\frac{3}{p}})},\label{A4}\\
\left\Vert \partial_{t}A_{\bar{v}_{1}}-\partial_{t}A_{\bar{v}_{2}}\right\Vert
_{L_{T}^{1}(\dot{B}_{p,1}^{\frac{3}{p}})}  &  \lesssim\left\Vert \nabla\delta
v\right\Vert _{L_{T}^{1}(\dot{B}_{p,1}^{\frac{3}{p}})},\label{A5}\\
\left\Vert \partial_{t}A_{\bar{v}_{1}}-\partial_{t}A_{\bar{v}_{2}}\right\Vert
_{L_{T}^{2}(\dot{B}_{p,1}^{\frac{3}{p}-1})}  &  \lesssim\left\Vert
\nabla\delta v\right\Vert _{L_{T}^{2}(\dot{B}_{p,1}^{\frac{3}{p}-1})}.
\label{A6}%
\end{align}

\end{proposition}

\end{document}